\documentclass[a4paper]{article}

\pdfoutput=1

\usepackage{mystyle}
\usepackage{subfiles}

\title{Monoidal Pull-Push II: Local Systems}
\author{Angus Rush}

\begin{document}

\maketitle 
\tableofcontents

\section{Introduction and conventions}
\label{sec:introduction_and_conventions}

\subsection{Introduction}
\label{ssc:introduction}

In topology, one often studies spaces via algebraic invariants associated to them. One such invariant of particular interest is the \emph{$n$th homology} of a space $X$, denoted $H_{n}(X)$. Roughly speaking, $H_{n}(X)$ is the abelian group consisting of formal sums of certain maps from the standard $n$-simplex $\Delta^{n}$ into $X$, considered modulo some equivalence relation.

It turns out to be fruitful to generalize $n$th homology. Formal sums are nothing else but $\Z$-linear combinations, and one can ask what happens when $\Z$ is replaced by some other commutative ring $R$. The associated invariant is then known as \emph{$n$th homology with coefficients in $R$}. One can even allow the ring $R$ to vary along $X$, or allow $R$ to be not only a ring, but an object of some appropriate category $\category{C}$. The structure which keeps track of these changing coefficients is called a $\category{C}$-\emph{local system}.

Classically speaking, when we say \emph{space,} we really mean \emph{topological space,} and $\category{C}$-local systems are modelled by locally-constant sheaves with values in $\category{C}$. In modern homotopy theory, one often takes \emph{space} to mean \emph{$\infty$-groupoid;} thought of in this way, a $\category{C}$-local system on a space $X$ is nothing more than a functor $X \to \category{C}$. This is the point of view that we will take.

With this point of view in mind, let $\category{C}$ be a cocomplete $\infty$-category. For any space $X$, denote by $\LS(\category{C})_{X}$ the $\infty$-category of $\category{C}$-local systems on $X$; that is, $\LS(\category{C})_{X} \simeq \Fun(X, \category{C})$. Recall that for any morphism $f\colon X \to Y$ of spaces, there are several associated functors between $\LS(\category{C})_{X}$ and $\LS(\category{C})_{Y}$. In particular:
\begin{itemize}
  \item The \emph{pullback functor} $f^{*}\colon \LS(\category{C})_{Y} \to \LS(\category{C})_{X}$ pulls back local systems on $Y$ to local systems on $X$, sending a local system $\mathcal{F}\colon Y \to \category{C}$ to the local system $f^{*}\mathcal{F}$ given by the composition
    \begin{equation*}
      \begin{tikzcd}
        X
        \arrow[r, "f"]
        & Y
        \arrow[r, "\mathcal{F}"]
        & \category{C}
      \end{tikzcd}.
    \end{equation*}

  \item The \emph{pushforward functor} $f_{!}\colon \LS(\category{C})_{X} \to \LS(\category{C})_{Y}$ pushes forward local systems via left Kan extension, sending a local system $\mathcal{G}\colon X \to \category{C}$ to the left Kan extension $f_{!}\mathcal{G}$.
    \begin{equation*}
      \begin{tikzcd}
        X
        \arrow[rr, "\mathcal{G}"]
        \arrow[dr, swap, "f"]
        && \category{C}
        \\
        & Y
        \arrow[ur, swap, "f_{!}G"]
      \end{tikzcd}
    \end{equation*}
\end{itemize}

Note that pulling back is contravariantly functorial, and pushing forward is covariantly functorial. One can combine these functorialities. Given the data of a diagram of spaces of the form
\begin{equation}
  \label{eq:span}
  \begin{tikzcd}
    & Y
    \arrow[dl, swap, "g"]
    \arrow[dr, "f"]
    \\
    X
    && X'
  \end{tikzcd},
\end{equation}
one can produce a functor $\LS(\category{C})_{X} \to \LS(\category{C})_{X'}$ via the composition
\begin{equation*}
  \begin{tikzcd}
    & \LS(\category{C})_{Y}
    \arrow[dr, "f_{!}"]
    \\
    \LS(\category{C})_{X}
    \arrow[ur, "g^{*}"]
    \arrow[rr, swap, "f_{!} \circ g^{*}"]
    && \LS(\category{C})_{X'}
  \end{tikzcd}.
\end{equation*}

The data of \hyperref[eq:span]{Diagram~\ref*{eq:span}} is known as a \emph{span} of spaces. It is natural to ask whether this construction can be extended to a functor $\Span(\S) \to \ICat$, where $\Span(\S)$ is an $\infty$-category whose objects are spaces, and whose morphisms are spans of spaces.

This is the second part of a two-part paper. In the first part, we provided a simplified proof of a theorem of Barwick \cite[Thm.~12.2]{spectralmackeyfunctors1}, providing several sufficient conditions for a functor of quasicategories $p\colon \category{C} \to \category{D}$ to yield a cocartesian fibration between $\infty$-categories of spans $\Span(p)\colon \Span(\category{C}) \to \Span(\category{D})$, and hence a functor $\hat{r}\colon \Span(\category{D}) \to \ICat$ via the Grothendieck construction. One such sufficient condition is that $p$ be a so-called \emph{Beck--Chevalley fibration} (\hyperref[def:beck_chevalley_fibration]{Definition~\ref*{def:beck_chevalley_fibration}}). Roughly speaking, a Beck--Chevalley fibration is a bicartesian fibration satisfying a straightened version of the Beck--Chevalley condition. More information about Beck--Chevalley fibrations is given in the introduction to \hyperref[sec:the_non_monoidal_construction]{Section~\ref*{sec:the_non_monoidal_construction}} of this work; a more in-depth introduction is given in \cite[Sec.~4.1]{luriehopkins2013ambidexterity}. It is an easy consequence of the theory of Beck--Chevalley fibrations that there exists a functor $\hat{r}\colon \Span(\category{S}) \to \ICat$, which sends a space $X$ to the category $\LS(\category{C})_{X}$ of local systems on $X$, and a morphism in $\Span(\S)$ represented by a span
\begin{equation*}
  \begin{tikzcd}
    & Y
    \arrow[dl, swap, "g"]
    \arrow[dr, "f"]
    \\
    X
    && X'
  \end{tikzcd}
\end{equation*}
to the functor $f_{!} \circ g^{*}\colon \LS(\category{C})_{X} \to \LS(\category{C})_{X'}$. This is the main result which we will prove in \hyperref[sec:the_non_monoidal_construction]{Section~\ref*{sec:the_non_monoidal_construction}}.

In this paper, we introduce a generalization of Beck--Chevalley fibrations, which we call \emph{monoidal Beck--Chevalley fibrations.} Roughly speaking, a monoidal Beck--Chevalley fibration $\tilde{p}\colon \category{C}_{\otimes} \to \category{D}_{\boxtimes}$ is an enhancement of an ordinary Beck--Chevalley fibration $p\colon \category{C} \to \category{D}$, which takes into account monoidal structures on $\category{C}$ and $\category{D}$, and provides a lax monoidal structure structure on the functor $\Span(\category{D}) \to \ICat$.

We use the theory of monoidal Beck--Chevalley fibrations to show that our functor $\hat{r}\colon \Span(\S) \to \ICat$ constructed in \hyperref[sec:the_non_monoidal_construction]{Section~\ref*{sec:the_non_monoidal_construction}} admits a lax monoidal structure with respect to a monoidal structure on $\Span(\S)$ induced by the cartesian structure on $\S$, and the cartesian structure on $\ICat$.

\subsection{Relation to previous work}
\label{ssc:relation_to_previous_work}

That pull-push of local systems can be written as a symmetric monoidal functor out of a category of spans is far from a new idea. Our approach differs from previous ones, however, in that our aim is to provide a simple, explicit construction of monoidal pull-push of local systems, without any pretensions of working in complete generality when it does not, in our view, lead to greater conceptual clarity.

There have been (to the knowledge of the author) two main works with results similar to those in this paper.
\begin{itemize}
  \item In \cite{spectralmackeyfunctors2}, similar results to those in \hyperref[ssc:monoidal_beck_chevalley_fibrations]{Subsection~\ref*{ssc:monoidal_beck_chevalley_fibrations}}, about monoidal Beck-Chevalley fibrations, are stated in a rather different situation. The application to local systems is not considered. We have treated these results in less generality.

  \item In \cite{gaitsgory2019study}, results about extending a functor out of a category of spans are proved, but more generally, more abstractly, and for the most part model-independently; there, spans come in $(\infty, 2)$-categories, modelled when necessary by complete 2-fold Segal spaces. Our approach is more explicit, based on hands-on combinatorial computations done in a quasicategorical twisted arrow category. We believe that the explicitness of our model will be useful in our future work.
\end{itemize}

\subsection{Outline}
\label{ssc:outline}

In \hyperref[sec:horn_filling_via_left_kan_extensions]{Section~\ref*{sec:horn_filling_via_left_kan_extensions}}, we generalize the universal property satisfied by Kan extensions of $1$-categories to $\infty$-categorical Kan extensions, and show how this universal property can be used to solve lifting problems in $\ICCat$.

Our proof that left Kan extensions enjoy this lifting property uses that left Kan extension is left adjoint to restriction; this can be used even if not all left Kan extensions along a functor exist. To formalize this, we introduce in \hyperref[ssc:adjunct_data]{Subsection~\ref*{ssc:adjunct_data}} the notion of a \emph{partial adjunction,} and use this to show that data on one side of an adjunction of $\infty$-categories can be used to fill adjunct data on the other side. Using this, we show in \hyperref[ssc:left_kan_implies_globally_left_kan]{Subsection~\ref*{ssc:left_kan_implies_globally_left_kan}} that this generalized lifting property really is satisfied.

In \hyperref[sec:local_systems]{Section~\ref*{sec:local_systems}}, we apply our results about left Kan extensions to show, via an explicit computation in an appropriate model (an $\infty$-categorical twisted arrow category), that left Kan extensions of local systems are $\infty$-functorial, i.e.\ extends to a functor $\S \to \ICat$. We then use this to show that pull-push functoriality extends to a functor $\Span(\S) \to \ICat$. In \hyperref[sec:the_monoidal_construction]{Section~\ref*{sec:the_monoidal_construction}}, we show that this construction can be endowed with a lax monoidal structure.

\subsection{General conventions and set-theoretic size issues}
\label{ssc:general_conventions}

We do our level best to adhere to the following terminological and typographical practices regarding various types of categories.
\begin{itemize}
  \item When we say `$\infty$-groupoid,' we will mean an $(\infty, 0)$-category, here always modelled by a Kan complex. We will in general denote $\infty$-groupoids by capital roman letters coming from the end of the alphabet: $X$, $Y$, \dots

  \item When we say `$\infty$-category,' we will mean an $(\infty, 1)$-category, usually modelled by a quasicategory. We will strive to denote $\infty$-categories by calligraphic letters typeset using the \textsf{eucal} package: $\category{C}$, $\category{D}$, \dots

  \item When we say `$\infty$-bicategory,' we will mean an $(\infty, 2)$-category in the sense of \cite{lurie2009infinity}, modelled by a fibrant scaled or marked-scaled\footnote{The theory of marked-scaled simplicial sets is given in \cite{garcia2cartesianfibrationsii}, and summarized in \hyperref[ssc:marked-scaled_model_structure]{Subsection~\ref*{ssc:marked-scaled_model_structure}}.} simplicial set; which one we mean should be clear from context. We will denote $(\infty, 2)$-categories by blackboard-bold letters typeset with the \textsf{mathbbol} package: $\CC$, $\DD$, \dots
\end{itemize}

We follow Lurie in our definitions of categories of $\infty$-categories (except for a few size-related transgressions explained below). We also adhere to the above typographical conventions in doing so. Thus:
\begin{itemize}
  \item The $(\infty, 1)$-category of spaces is $\S$. We model this as a quasicategory, constructed (as in \cite{highertopostheory}) as the homotopy-coherent nerve (as described in \cite[Sec.~1.1.5]{highertopostheory}, there called the simplicial nerve) of $\Kan$, the Kan-enriched category of Kan complexes.

  \item The $(\infty, 1)$-category of $(\infty,1)$-categories is $\ICat$. We model this as a quasicategory, defined to be the homotopy-coherent nerve of $\QCat$, the Kan-enriched category of quasicategories.

  \item The $(\infty,2)$-category of $(\infty, 1)$-categories is $\ICCat$. We model this as an $\infty$-bicategory, defined to be the scaled nerve (as described in \cite[Sec.~3.1]{lurie2009infinity}) of the $\SSetmk$-enriched category of quasicategories.

  \item The $(\infty, 1)$-category of $(\infty,2)$-categories is $\Cat_{(\infty, 2)}$. This is used only briefly, and will be modelled as the simplicial localization of $(\SSetms)^{f}$, the full subcategory of marked-scaled simplicial sets on fibrant objects, at the class of bicategorical equivalences.
\end{itemize}

For the most part, our conventions regarding set-theoretic size issues are standard. However, we will at several points point need to consider an $\infty$-category whose objects are large $\infty$-categories. The rigorous solution would be to introduce a series of nested Grothendieck universe, keep track of which one we are currently in, and carry this around as extra notation. However, there are no arguments in this paper which hinge on any set-theoretic size-issues, and the author feels that clarity is lost, rather than gained, by introducing this extraneous notation. Therefore, we will use the same notation for the large $\infty$-category of small $\infty$-categories and the `huge' $\infty$-category of large $\infty$-categories.

Suppose $K$ is a simplicial set and $\category{C}$ is an $\infty$-category. By $\Fun(K, \category{C})$, we mean the $\infty$-category of maps $K \to \category{C}$. By $\Map(K, \category{C})$, we mean the $\infty$-groupoid of such maps; that is,
\begin{equation*}
  \Map(K, \category{C}) = \Fun(K, \category{C})^{\simeq},
\end{equation*}
where $(-)^{\simeq}\colon \SSet \to \Kan$ denotes the \emph{core} functor, i.e.\ the functor associating to a simplicial set $K$ the largest Kan complex contained in it.

\subsection{The marked-scaled model structure}
\label{ssc:marked-scaled_model_structure}

We will need two different models for $(\infty,2)$-categories: Lurie's theory of scaled simplicial sets, as laid out in \cite{lurie2009infinity}; and Abell\'an--Stern's theory of marked-scaled simplicial sets, as explained in \cite{garcia2cartesianfibrationsii}. We will assume a knowledge of scaled simplicial sets. We give a basic outline of the portions of the theory of marked-scaled simplicial sets which we will need.

A marked-scaled simplicial set is a triple $(X, E_{X}, T_{X})$, where $E_{X} \subseteq X_{1}$ is a collection of edges of $X$ containing all degenerate edges, and $T_{X} \subseteq X_{2}$ is a collection of triangles of $X$ containing all degenerate triangles. We will use the following notation.
\begin{itemize}
  \item We will denote the category of marked-scaled simplicial sets by $\SSetms$.

  \item To save on notation, we will sometimes denote the marked-scaled simplicial set $(X, E_{X}, T_{X})$ by $X^{E_{X}}_{T_{X}}$, particularly in the case that $E_{X}$ or $T_{X}$ are $\sharp$ or $\flat$, the maximum (resp. minimum) markings and scalings. For example, the bimarked simplicial set $X^{\sharp}_{\flat}$ has all 1-simplices marked and only degenerate 2-simplices scaled.

  \item Let $(\Delta^{n}, E, T)$ be a marked-scaled $n$-simplex. For any simplicial subset $S \subseteq \Delta^{n}$, we will denote by $(S, E, T)$ the simplicial subset $S$ together with the inherited marking and scaling.
\end{itemize}

There is a set of marked-scaled anodyne morphisms, which have the left lifting property with respect to marked-scaled fibrations. We will not need to use the full power of the marked-scaled anodyne morphisms, so we content ourselves with an incomplete description.

\begin{definition}
  \label{def:ms-anodyne_morphisms}
  The set of ms-anodyne morphisms is a saturated set of morphisms between marked-scaled simplicial sets containing the following classes of morphisms:
  \begin{enumerate}[label=(A\arabic*)]
    \item\label{item:innerms} Inner horn inclusions
      \begin{equation*}
        (\Lambda^{n}_{i}, \flat, \{\Delta^{\{i-1,i,i+1\}}\}) \to (\Delta^{n}, \flat, \{\Delta^{\{i-1,i,i+1\}}\}),
      \end{equation*}
      for $n \geq 2$ and $0 < i < n$.

    \item\label{item:outerms} Outer horn inclusions
      \begin{equation*}
        (\Lambda^{n}_{n}, \{\Delta^{\{n-1,n\}}\}, \{\Delta^{\{0, n-1, n\}}\}),
      \end{equation*}
      for $n \geq 1$.
  \end{enumerate}
\end{definition}

\begin{example}
  \label{prop:sharp_marked_right_anodyne}
  The marked-scaled anodyne morphisms encapsulate both left- and marked-anodyne morphisms.
  \begin{itemize}
    \item For any right anodyne morphism $A \hookrightarrow B$, the morphism $A^{\sharp}_{\sharp} \hookrightarrow B^{\sharp}_{\sharp}$ is marked-scaled anodyne.

    \item For any marked anodyne morphism $(A, \mathcal{E}) \to (B, \mathcal{F})$, the morphism $A^{\mathcal{E}}_{\sharp} \to B^{\mathcal{F}}_{\sharp}$ is marked-scaled anodyne.
  \end{itemize}
\end{example}

\begin{theorem}
  There is a model structure on the category $\SSetms$, whose trivial cofibrations are contain the set of marked-scaled anodyne maps, and whose fibrant objects are $\infty$-bicategories with the equivalences marked and thin simplices scaled.
\end{theorem}

\begin{theorem}
  \label{thm:quillen_equiv_ms_and_scaled}
  There is a Quillen equivalence
  \begin{equation*}
    (-)_{\flat} : \SSetsc \longleftrightarrow \SSetms : G,
  \end{equation*}
  where $(-)^{\flat}$ endows any scaled simplicial set with the flat marking, and $G$ forgets the marking.
\end{theorem}

\subsection{Selected results from Monoidal Pull-Push I}
\label{ssc:selected_results_from_mppi}

We will at several points make use of the results of the first part of this paper, Monoidal Pull-Push I. We reproduce here our main theorems there, both of which will be instrumental in this paper. These give sufficient conditions on a functor $p\colon \category{C} \to \category{D}$ so that the induced functor $\Span(\category{C}) \to \Span(\category{D})$ is a cocartesian fibration. Here, the subcategories $\category{C}\downdag$ and $\category{C}\updag$ pick out distinguished classes of morphisms to which the legs of the spans of $\Span(\category{C})$ are allowed to belong; the `forwards-facing' legs of $\Span(\category{C})$ are restricted to come from the subcategory $\category{C}\downdag$, and the `forward-facing' legs from $\category{C}\updag$; and similarly for $\category{D}\downdag$ and $\category{D}\updag$.

\begin{theorem}
  \label{thm:old_barwick}
  Let $p\colon \triple{C} \to \triple{D}$ be a functor between adequate triples such that $p\colon \category{C} \to \category{D}$ is an inner fibration which satisfies the following conditions.
  \begin{enumerate}
    \item Each morphism $g \in \category{D}\downdag$ admits a lift to a morphism in $\category{C}\downdag$ (given a lift of the source) which is both $p$-cocartesian and $p\downdag$-cocartesian.

    \item Consider a commutative square
      \begin{equation*}
        \sigma = \quad
        \begin{tikzcd}
          y'
          \arrow[r, rightarrowtail, "f'"]
          \arrow[d, two heads, swap, "g'"]
          & x'
          \arrow[d, "g"]
          \\
          y
          \arrow[r, rightarrowtail, "f"]
          & x
        \end{tikzcd}
      \end{equation*}
      in $\category{C}$ where $g'$ belongs to $\category{C}\updag$, and $f$ and $f'$ belong to $\category{C}\downdag$. Suppose that $f$ is $p$-cocartesian. Then $f'$ is $p'$-cocartesian if and only if $\sigma$ is an ambigressive pullback square (and in particular $g \in \category{C}\updag$).
  \end{enumerate}

  Then spans of the form
  \begin{equation*}
    \begin{tikzcd}
      & z
      \arrow[dl, two heads, swap, "g"]
      \arrow[dr, rightarrowtail, "f"]
      \\
      x
      && y
    \end{tikzcd}
  \end{equation*}
  are cocartesian, where $g$ is $p\updag$-cartesian and $f$ is $p$-cocartesian.
\end{theorem}

\begin{theorem}
  \label{thm:new_barwick}
  Let $p\colon \triple{C} \to \triple{D}$ be a functor between adequate triples such that $p\colon \category{C} \to \category{D}$ is an inner fibration which satisfies the following conditions.
  \begin{enumerate}
    \item The subcategory $\category{C}\updag \subseteq \category{C}$ consists of all $p$-cartesian morphisms in $\category{C}$; that is, an $n$-simplex in $\category{C}$ belongs to $\category{C}\updag$ if and only if each $1$-simplex it contains is $p$-cartesian.

    \item The map $p\updag\colon \category{C}\updag \to \category{D}\updag$ is a cartesian fibration.

    \item Consider a square
      \begin{equation*}
          \sigma = \quad
          \begin{tikzcd}
            y'
            \arrow[r, "f'"]
            \arrow[d, two heads, swap, "g'"]
            & x'
            \arrow[d, two heads, "g"]
            \\
            y
            \arrow[r, rightarrowtail, "f"]
            & x
          \end{tikzcd}
      \end{equation*}
      in $\category{C}$ where $g$ and $g'$ belong to $\category{C}\updag$, and $f$ belongs to $\category{C}\downdag$. Further suppose that $f$ is $p$-cocartesian. Then $f'$ belongs to $\category{C}\downdag$, and is both $p$-cocartesian and $p\downdag$-cocartesian.
  \end{enumerate}
  Then spans of the form
  \begin{equation*}
    \begin{tikzcd}
      & z
      \arrow[dl, two heads, swap, "g"]
      \arrow[dr, rightarrowtail, "f"]
      \\
      x
      && y
    \end{tikzcd}
  \end{equation*}
  are cocartesian, where $g$ is $p\updag$-cartesian and $f$ is $p$-cocartesian.
\end{theorem}

\subsection{Acknowledgements}
\label{ssc:acknowledgements}

We would like to thank his advisor, Tobias Dyckerhoff, for his patience and guidance. We would also like to thank Fernando Abell{\'a}n Garc{\'i}a for many helpful conversations about the results contained in this paper, especially those pertaining to the twisted arrow category.

\section{Horn filling via left Kan extensions}
\label{sec:horn_filling_via_left_kan_extensions}

Kan extensions of $\infty$-categories are usually defined pointwise via the colimit formula \cite{highertopostheory} \cite{cisinski2019higher}. In this section, we will show that such pointwise $\infty$-Kan extensions enjoy a horn filling property in $\ICCat$ which generalizes the universal property for Kan extensions of functors between $1$-categories.

We will explain the precise meaning of the below definition in \hyperref[ssc:basic_facts_about_kan_extensions]{Subsection~\ref*{ssc:basic_facts_about_kan_extensions}}.

\begin{definition}
  \label{def:left_kan}
  We will say that a $2$-simplex $\tau\colon \Delta^{2}_{\flat} \to \ICCat$,
  \begin{equation}
    \label{eq:prototypical_left_kan_simplex}
    \tau =
    \begin{tikzcd}
      X
      \arrow[dr, swap, "f"]
      \arrow[rr, "F", ""{below, name=M}]
      && \category{C}
      \\
      & Y
      \arrow[ur, swap, "G"]
      \arrow[from=M, Rightarrow, swap, "\eta"]
    \end{tikzcd},
  \end{equation}
  is \defn{left Kan} if the natural transformation $\eta$ exhibits $G$ as the left Kan extension of $F$ along $f$.
\end{definition}

The goal of this section is to prove the following.

\begin{theorem}
  \label{thm:left_kan_implies_globally_left_kan}
  Let $\tau\colon \Delta^{2}_{\flat} \to \ICCat$ be a left Kan 2-simplex. Then for each $n \geq 2$, every solid lifting problem of the form
  \begin{equation}
    \label{eq:lifting_problems_for_global_kan_extensions}
    \begin{tikzcd}
      \Delta^{\{0,1,n+1\}}_{\flat}
      \arrow[d, hook]
      \arrow[dr, "\tau"]
      \\
      (\Lambda^{n+1}_{0})_{\flat}
      \arrow[r]
      \arrow[d, hook]
      & \ICCat
      \\
      \Delta^{n+1}_{\flat}
      \arrow[ur, dashed]
    \end{tikzcd}
  \end{equation}
  admits a dashed solution.
\end{theorem}

\begin{note}
  One can even show that if the terminal vertex in $\tau$ is mapped to an $\infty$-category $\category{C}$ with sufficient colimits (i.e.\ all colimits in $\category{C}$ indexed by the homotopy fibers of $f$ should exist), then $\tau$ is left kan \emph{if and only if} it enjoys these lifting properties. That is, these lifting properties generalize the universal property for Kan extensions of ordinary categories. However, we will not make use of this result.
\end{note}


Let us examine the lifting problems of \hyperref[eq:lifting_problems_for_global_kan_extensions]{Equation~\ref*{eq:lifting_problems_for_global_kan_extensions}} in the case $n = 2$. In this case we have the data of categories and functors
\begin{equation*}
  \begin{tikzcd}
    & Y
    \arrow[drr, "f_{!}F"]
    \\
    X
    \arrow[dr, swap, "h"]
    \arrow[ur, "f"]
    \arrow[rrr, near end, "F"]
    &&& \category{C}
    \\
    & Z
    \arrow[urr, swap, "G"]
    \arrow[from=uu, crossing over, near start, "g"]
  \end{tikzcd}
\end{equation*}
together with natural transformations $\eta\colon F \Rightarrow f_{!}F \circ f$, $\beta\colon h \Rightarrow g \circ f$, and $\delta\colon F \Rightarrow G \circ h$, making up the sides of $\Lambda^{3}_{0}$. We need to produce a natural transformation $\alpha\colon f_{!}F \Rightarrow G \circ g$ and a filling of the full 3-simplex, which is the data of a homomotopy-commutative diagram
\begin{equation*}
  \begin{tikzcd}
    F
    \arrow[r, "\eta"]
    \arrow[d, swap, "\delta"]
    & f_{!}F \circ f
    \arrow[d, "\alpha f"]
    \\
    G \circ h
    \arrow[r, "G \beta"]
    & G \circ g \circ f
  \end{tikzcd}
\end{equation*}
in $\Fun(X, \category{C})$.\footnote{In particular, if we take $h = f$, $g = \id_{Y}$, and $\beta$ to be the identity $f \Rightarrow f$, we recover the classical universal property satisfied by left Kan extension.}

We can rephrase this as follows. We are given the data
\begin{equation*}
  \overbrace{
  \begin{tikzcd}[ampersand replacement=\&]
    F
    \arrow[r, "\eta"]
    \arrow[d, swap, "\delta"]
    \& f_{!}F \circ f
    \\
    G \circ h
    \arrow[r, "G \beta"]
    \& G \circ g \circ f
  \end{tikzcd}}^{\text{in }\Fun(X, \category{C})}
  \qquad\qquad
  \overbrace{
  \begin{tikzcd}
    f_{!}F
    \\
    G \circ g
  \end{tikzcd}}^{\text{ in }\Fun(Y, \category{C})}
\end{equation*}
where the left-hand diagram is a map $LC^{2} \to \Fun(X, C)$, where $LC^{2}$ is the simplicial subset of the boundary of $\Delta^{1} \times \Delta^{1}$ except the right-hand face, and the right-hand diagram is a map $\partial \Delta^{1} \to \Fun(Y, C)$. We need to construct a filler $\partial \Delta^{1} \hookrightarrow \Delta^{1}$ on the right, and from it a filler $LC^{2} \hookrightarrow \Delta^{1} \times \Delta^{1}$ on the left. We do this in the following sequence of steps.

\begin{enumerate}
  \item We first can fill the lower-left half of the diagram on the left simply by taking the composition $G\beta \circ \delta$. Doing this, we are left with the filling problem
    \begin{equation*}
      \overbrace{
      \begin{tikzcd}[ampersand replacement=\&]
        F
        \arrow[r, "\eta"]
        \arrow[dr]
        \& f_{!}F \circ f
        \\
        \& G \circ g \circ f
      \end{tikzcd}}^{\text{in }\Fun(X, \category{C})}
      \qquad\qquad
      \overbrace{
      \begin{tikzcd}
        f_{!}F
        \\
        G \circ g
      \end{tikzcd}}^{\text{ in }\Fun(Y, \category{C})}.
    \end{equation*}

  \item Using the adjunction $f_{!} \dashv f^{*}$, we can extend the diagram on the right to a diagram which is adjunct to the diagram on the left (in the sense to be described in \hyperref[ssc:adjunct_data]{Subsection~\ref*{ssc:adjunct_data}}):
    \begin{equation*}
      \overbrace{
      \begin{tikzcd}[ampersand replacement=\&]
        F
        \arrow[r, "\eta"]
        \arrow[dr]
        \& f_{!}F \circ f
        \\
        \& G \circ g \circ f
      \end{tikzcd}}^{\text{in }\Fun(X, \category{C})}
      \qquad\qquad
      \overbrace{
        \begin{tikzcd}[ampersand replacement=\&]
          f_{!}F
          \arrow[r, "\id"]
          \arrow[dr]
        \& f_{!}F
        \\
        \& G \circ g
      \end{tikzcd}}^{\text{ in }\Fun(Y, \category{C})}
    \end{equation*}

  \item The diagram on the right has an obvious filler, which is adjunct to a filler on the left.
    \begin{equation*}
      \overbrace{
      \begin{tikzcd}[ampersand replacement=\&]
        F
        \arrow[r, "\eta"]
        \arrow[dr]
        \& f_{!}F \circ f
        \arrow[d, "\alpha f"]
        \\
        \& G \circ g \circ f
      \end{tikzcd}}^{\text{in }\Fun(X, \category{C})}
      \qquad\qquad
      \overbrace{
        \begin{tikzcd}[ampersand replacement=\&]
          f_{!}F
          \arrow[r, "\id"]
          \arrow[dr]
        \& f_{!}F
        \arrow[d, "\alpha"]
        \\
        \& G \circ g
      \end{tikzcd}}^{\text{ in }\Fun(Y, \category{C})}
    \end{equation*}
\end{enumerate}

The lifting problems which we have to solve in \hyperref[eq:lifting_problems_for_global_kan_extensions]{Equation~\ref*{eq:lifting_problems_for_global_kan_extensions}} for $n > 2$ amount to replacing $\Delta^{1} \times \Delta^{1}$ by $(\Delta^{1})^{n}$, etc; the basic process remains unchanged, but the combinatorics involved in filling the necessary cubes becomes more involved. In \hyperref[ssc:filling_cubes_relative_to_their_boundary]{Subsection~\ref*{ssc:filling_cubes_relative_to_their_boundary}} we explain the combinatorics of filling cubes relative to their boundaries. In \hyperref[ssc:adjunct_data]{Subsection~\ref*{ssc:adjunct_data}}, we give a formalization the concept of adjunct data, and provide a means of for filling partial data to total adjunct data. In \hyperref[ssc:left_kan_implies_globally_left_kan]{Subsection~\ref*{ssc:left_kan_implies_globally_left_kan}}, we show that we can always solve the lifting problems of \hyperref[eq:lifting_problems_for_global_kan_extensions]{Equation~\ref*{eq:lifting_problems_for_global_kan_extensions}}.

\subsection{Basic facts about Kan extensions}
\label{ssc:basic_facts_about_kan_extensions}

In this section, we recall a few basic facts about Kan extensions. These are mostly results found in \cite{kerodon} which we will need. Since Kan extensions are defined by the colimit formula, we will start by defining colimits. Clasically, colimits in $\infty$-categories are defined using colimit cones.

\begin{definition}
  \label{def:colimit_via_cocones}
  Let $F\colon K \to \category{C}$ be a diagram, where $\category{C}$ is an $\infty$-category. A cocone $\tilde{F}\colon K^{\triangleright} \to \category{C}$ is a \defn{colimit cone} if it is an initial object in the $\infty$-category $\category{C}_{F/}$.
\end{definition}

It is sometimes more convenient to define colimits via natural transformations to a constant functor.

\begin{definition}
  Let $f\colon K \to \category{C}$ be a map between simplicial sets, where $\category{C}$ is an $\infty$-category. Let $c \in \category{C}$ be an object, and denote by $\underline{c}$ the constant functor $K \to \category{C}$ with value $c$. A natural transformation $f \Rightarrow \underline{c}$ \defn{exhibits $c$ as the colimit of $f$} if it is an initial object in the $\infty$-category $\category{C}^{F/}$.
\end{definition}

Fortunately, these definitions are compatible: an object $K^{\triangleright} \to \category{C}$ in $\category{C}_{/F}$ is a colimit cone with cone tip $c$ if and only if the composite map
\begin{equation}
  \label{eq:natural_transformation_from_colimit_cone}
  K \times \Delta^{1} \to K \times \Delta^{1} \amalg_{K \times \{1\}} \Delta^{0} \to K \star \Delta^{0} \to \category{C}
\end{equation}
exhibits $c$ as a colimit of $F$. This follows from the fact that for any $\infty$-category $\category{C}$ and any functor $F\colon K \to \category{C}$, the comparison map $\category{C}_{/F} \to \category{C}^{/F}$ of \cite[Prop.~4.2.1.5]{highertopostheory} is a categorical equivalence.

\begin{note}
  It follows immediately from this description that colimit cones are homotopically invariant: if a natural transformation $\eta$ exhibits some object as a colimit of some diagram, than any natural transformation $\eta'$ which is homotopic to $\eta$ will do just as well.
\end{note}

We take a moment to record a definition, which we will need later.

\begin{definition}
  \label{def:preserve_colimits_in_each_slot}
  Let $\category{C}$ be an $\infty$-category, and let $F\colon \category{C} \times \category{C} \to \category{C}$ be a functor. We say that \defn{$F$ preserves colimits in each slot} if and only if the following condition is satisfed:
  \begin{itemize}
    \item Let $K$ and $K'$ be simplicial sets; let $G\colon K \to \category{C}$ and $G'\colon K' \to \category{C}$ be functors; and suppose $\eta\colon G \Rightarrow \underline{c}$ exhibits $c$ as the colimit of $G$, and $\eta'\colon G' \Rightarrow \underline{c'}$ exhibits $c'$ as the colimit of $G'$. Then $F(\eta \times \eta')\colon F \circ (G \times G') \Rightarrow \underline{F(c, c')}$ exhibits $F(c, c')$ as the colimit of $F \circ (G \times G')$.
  \end{itemize}
\end{definition}

\begin{notation}
  \label{notation:rund_um_undercategories}
  Let $f\colon X \to Y$ be a map of simplicial sets, and $y \in Y$ an object. We will use the following notation.
  \begin{itemize}
    \item Denote by $X_{/y}$ the fiber product $X \times_{Y} Y_{/y}$.

    \item Denote by $\pi\colon X_{/y} \to X$ the projection map.

    \item Denote by $\alpha\colon f \circ \pi \Rightarrow \underline{y}$ the natural transformation $X_{/y} \times \Delta^{1} \to X$ coming from the comparison map $X_{/y} \to X^{/y}$.
  \end{itemize}
\end{notation}

\begin{definition}[\protect{\cite[Variant~7.3.1.5]{kerodon}}]
  \label{def:nat_xfo_exhibiting_left_kan_ext}
  Let $X$, $Y$, and $\category{C}$ be $\infty$-categories, $f\colon X \to Y$, $F\colon X \to \category{C}$ and $G\colon Y \to \category{C}$ functors, and $\eta\colon F \Rightarrow G \circ f$ a natural transformation.
  \begin{equation*}
    \begin{tikzcd}
      X
      \arrow[dr, swap, "f"]
      \arrow[rr, "F", ""{below, name=M}]
      && \category{C}
      \\
      & Y
      \arrow[ur, swap, "G"]
      \arrow[from=M, Rightarrow, swap, "\eta"]
    \end{tikzcd}
  \end{equation*}
  We say that $\eta$ \defn{exhibits $G$ as the left Kan extension of $F$ along $f$} if for each object $y \in Y$, the natural transformation $F \circ \pi \Rightarrow \underline{G(y)}$ furnished by the pasting diagram
  \begin{equation*}
    \begin{tikzcd}[row sep=large, column sep=large]
      X_{/y}
      \arrow[r, "\pi", ""{name=LA, swap}]
      \arrow[d, ""{name=LD}]
      & X
      \arrow[d, swap, "f"]
      \arrow[r, "F", ""{below, name=M}]
      & \category{C}
      \\
      \{y\}
      \arrow[r, hook, ""{name=LB}]
      & Y
      \arrow[ur, swap, "G"]
      \arrow[from=M, Rightarrow, shorten=2ex, swap, "\eta"]
      \arrow[from=LA, to=LB, Rightarrow, shorten=1ex, "\alpha"]
    \end{tikzcd}
  \end{equation*}
  exhibits $G(y)$ as the colimit of the functor $F \circ \pi$. Here, $\underline{G(y)}$ is the `counterclockwise' path from $X_{/y}$ to $\category{C}$.
\end{definition}

The following guarantees the existence of local left Kan extensions.
\begin{theorem}[\protect{\cite[Proposition 7.3.5.1]{kerodon}}]
  \label{thm:existence_local_left_kan_exts}
  Suppose that $f\colon X \to Y$ is a map of simplicial sets, and suppose $F\colon X \to \category{C}$ is a map of simplicial sets such that $\category{C}$ is a quasicategory. The functor $F$ admits a left Kan extension along $f$ if and only if for all objects $x \in X$, the colimit of the functor
  \begin{equation*}
    X_{/y} \overset{\pi}{\to} X \overset{F}{\to} \category{C}
  \end{equation*}
  exists in $\category{C}$.
\end{theorem}

The following is a combination of special cases of \cite[Proposition~7.3.1.15]{kerodon} and \cite[Corollary~7.3.1.16]{kerodon}
\begin{example}
  \label{eg:strictly_commuting_left_kan}
  For any homotopy-commuting diagram of simplicial sets
  \begin{equation*}
    \begin{tikzcd}
      X
      \arrow[dr, swap, "f"]
      \arrow[rr, "F", ""{below, name=M}]
      && \category{C}
      \\
      & Y
      \arrow[ur, swap, "G"]
      \arrow[from=M, Rightarrow, "\simeq", "\eta"{swap}]
    \end{tikzcd}
  \end{equation*}
  where $X$, $Y$, and $\category{C}$ are $\infty$-categories and $f\colon X \to Y$ is a categorical equivalence, $\eta$ exhibits $G$ as a left Kan extension of $F$ along $f$.
\end{example}

The following guarantees existence of global left Kan extensions.
\begin{theorem}[\protect{\cite[Corollary~7.3.6.3]{kerodon}}]
  \label{thm:existence_global_left_kan_exts}
  Let $f\colon X \to Y$ be a map of simplicial sets, and suppose that $\category{C}$ is a quasicategory which admits $X_{/y}$-shaped colimits for all $y \in Y$. Then the restriction functor $f^{*}\colon \Fun(Y, \category{C}) \to \Fun(X, \category{C})$ admits a left adjoint $f_{!}$, sending a functor $F\colon X \to \category{C}$ to $f_{!}F\colon Y \to \category{C}$, its left Kan extension along $f$.
\end{theorem}

The following is an easy consequence of \cite[Remark~7.3.1.11]{kerodon}, to be explained later.
\begin{proposition}
  Let $f\colon X \to Y$ be a map of simplicial sets, and let $\category{C}$ be a category such that all functors $X \to \category{C}$ admit left Kan extensions along $f$, so that we have an adjunction $f_{!} \vdash f^{*}$. Let $G\colon Y \to \category{C}$ and $\eta\colon F \Rightarrow G \circ f$. The natural transformation $\eta$ exhibits $G$ as a left Kan extension of $F$ along $f$ if and only if it is adjunct (in the sense of \hyperref[def:adjunct_data]{Definition~\ref*{def:adjunct_data}}) to an equivalence in $\Fun(Y, \category{C})$ relative to $s = \id_{\Delta^{1}}$.
\end{proposition}

\begin{proposition}[\protect{\cite[Proposition 7.3.7.18]{kerodon}}]
  \label{prop:kan_extend_along_composition}
  Consider $\infty$-categories and functors
  \begin{equation*}
    \begin{tikzcd}
      X
      \arrow[dr, swap, "f"]
      \arrow[rrrr, "F"]
      &&&& \category{C}
      \\
      & Y
      \arrow[urrr, "G"]
      \arrow[dr, swap, "g"]
      \\
      && Z
      \arrow[uurr, swap, "H"]
    \end{tikzcd},
  \end{equation*}
  and natural transformations $\alpha\colon F \Rightarrow G \circ f$ and $\beta\colon H \circ g \Rightarrow G$, and suppose that $\alpha$ exhibits $G$ as the left Kan extension of $F$ along $f$. Then $\beta$ exhibits $H$ as the left Kan extension of $G$ along $g$ if and only if $\beta f \circ \alpha$ exhibits $H$ as the left Kan extension of $F$ along $g \circ f$.
\end{proposition}

\begin{proposition}[\protect{\cite[Remark 7.3.1.12]{kerodon}}]
  \label{prop:kan_ext_invariance_of_target}
  Suppose $H\colon \category{C} \to \category{D}$ is an equivalence of categories, $f\colon X \to Y$ is a map of simplicial sets, $G\colon Y \to \category{C}$ is a functor, and $\eta\colon F \Rightarrow G \circ f$ is a natural transformation. Then $\eta$ exhibits $G$ as a left Kan extension of $F$ along $f$ if and only if $H\eta\colon H \circ F \Rightarrow H \circ G \circ f$ exhibits $H \circ G$ as a left Kan extension of $H \circ F$ along $f$.
\end{proposition}

\begin{proposition}[\protect{\cite[Remark 7.3.1.9]{kerodon}}]
  \label{prop:homotopy_invariance_of_witness}
  The property that $\eta$ exhibits $G$ as a left Kan extension of $F$ along $f$ is homotopy-invariant; any homotopic $\eta$ will do just as well.
\end{proposition}

\subsection{Filling cubes relative to their boundaries}
\label{ssc:filling_cubes_relative_to_their_boundary}

Our main goal in this \hyperref[sec:horn_filling_via_left_kan_extensions]{Section~\ref*{sec:horn_filling_via_left_kan_extensions}} is to prove \hyperref[thm:left_kan_implies_globally_left_kan]{Theorem~\ref*{thm:left_kan_implies_globally_left_kan}}. To do this, we must understand lifting problems of the form
\begin{equation*}
  \begin{tikzcd}
    (\Lambda^{n+1}_{0})_{\flat}
    \arrow[r]
    \arrow[d, hook]
    & \ICCat
    \\
    \Delta^{n+1}_{\flat}
    \arrow[ur, dashed]
  \end{tikzcd}.
\end{equation*}

The category $\ICCat$ is defined to be the homotopy coherent nerve of the $\SSetmk$-enriched category $\QCat$ of quasicategories. Therefore, solving such lifting problems is equivalent to solving the adjunct lifting problems
\begin{equation*}
  \begin{tikzcd}
    \C[\Lambda^{n+1}_{0}]
    \arrow[r]
    \arrow[d, hook]
    & \QCat
    \\
    \C[\Delta^{n+1}]
    \arrow[ur, dashed]
  \end{tikzcd}
\end{equation*}
where the diagram in question is now a diagram of $\SSetmk$-enriched categories (where the mapping spaces of $\C[\Lambda^{n+1}_{0}]$ and $\C[\Delta^{n+1}]$ are taken to carry the flat marking). In order to solve such lifting problems, we need to fill the missing data in each mapping space. As we will see in \hyperref[ssc:left_kan_implies_globally_left_kan]{Subsection~\ref*{ssc:left_kan_implies_globally_left_kan}}, these fillings take the form of filling a cube $(\Delta^{1})^{n}$ relative to its boundary missing one face.

In this subsection, we give the combinatorics of filling the $n$-cube. Our main result is \hyperref[prop:cube_filling]{Proposition~\ref*{prop:cube_filling}}, which writes the inclusion the boundary of the $n$-cube missing a certain face into the full cube as a composition of an inner anodyne map and a marked anodyne map.

\begin{definition}
  The \defn{$n$-cube} is the simplicial set $C^{n} := (\Delta^{1})^{n}$.
\end{definition}

\begin{note}
  We will consider the the factors $\Delta^{1}$ of $C^{n}$ to be ordered, so that we can speak about the first factor, the second factor, etc.
\end{note}

Our first step is to understand the nondegenerate simplices of the $n$-cube.

\begin{definition}
  Denote by $a^{i}\colon \Delta^{n} \to \Delta^{1}$ the map which sends $\Delta^{\{0, \ldots, i-1\}}$ to $\{0\}$, and $\Delta^{\{i, \ldots, n\}}$ to $\{1\}$.
\end{definition}

\begin{note}
  The superscript of $a^{i}$ counts how many vertices of $\Delta^{n}$ are sent to $\Delta^{\{0\}} \subset \Delta^{1}$.
\end{note}

Every nondegenerate simplex $\Delta^{n} \to C^{n}$ can be specified by giving a walk along the edges of $C^{n}$ starting at $(0, \ldots, 0)$ and ending at $(1, \ldots, 1)$, and any simplex specified in this way is nondegenerate. We now use this to define a bijection between $S_{n}$, the symmetric group on the set $\{1, \ldots, n\}$, and the nondegenerate simplices $\Delta^{n} \to C^{n}$ as follows.

\begin{definition}
  For any permutation $\tau\colon \{1, \ldots, n\} \to \{1, \ldots, n\}$, we define an $n$-simplex
  \begin{equation*}
    \phi(\tau)\colon \Delta^{n} \to C^{n};\qquad \phi(\tau)_{i} = a^{\tau(i)}.
  \end{equation*}
  That is,
  \begin{equation*}
    \phi(\tau) = (a^{\tau(1)}, \ldots, a^{\tau(n)}).
  \end{equation*}
\end{definition}

\begin{note}
  It is easy to check that the above definition really results in a bijection between $S_{n}$ and the nondegenerate simplices of $C^{n}$.
\end{note}

\begin{notation}
  We will denote any permutatation $\tau\colon \{1, \ldots, n\} \to \{1, \ldots, n\}$ by the corresponding $n$-tuple $(\tau(1), \ldots, \tau(n))$. Thus, the identity permutation is $(1, \ldots, n)$, and the permutation $\gamma_{1,2}$ which swaps $1$ and $2$ is $(2, 1, 3, \ldots, n)$.
\end{notation}

Given a permutation $\tau \in S_{n}$ corresponding to an $n$-simplex
\begin{equation*}
  \phi(\tau) = (a^{\tau(1)}, \ldots, a^{\tau(n)})\colon \Delta^{n} \to C^{n},
\end{equation*}
the $i$th face of $\phi(\tau)$ is a nondegenerate $(n-1)$-simplex in $C^{n}$, i.e.\ a map $\Delta^{n-1} \to C^{n}$. We calculate this as follows: for any $a^{i}\colon \Delta^{n} \to \Delta^{1}$ and any face map $\partial_{j}\colon \Delta^{n-1} \to \Delta^{n}$, we note that
\begin{equation*}
  \partial_{j}^{*}a^{i} =
  \begin{cases}
    a^{i-1}, & j < i \\
    a^{i}, & j \geq i
  \end{cases},
\end{equation*}
where by minor abuse of notation we denote the map $\Delta^{n-1} \to \Delta^{1}$ sending $\Delta^{\{0, \ldots i-1\}}$ to $\{0\}$ and the rest to $\{1\}$ also by $a^{i}$. We then have that
\begin{gather*}
  \label{eq:simplicial_identity_for_cubes}
  d_{i}\tau = d_{i}(a^{\tau(1)}, \ldots, a^{\tau(n)}) := (\partial_{i}^{*} a^{\tau(1)}, \ldots, \partial_{i}^{*} a^{\tau(n)}).
\end{gather*}

\begin{example}
  Consider the 3-simplex $(a^{1}, a^{2}, a^{3})$ in $C^{3}$. This has spine
  \begin{equation*}
    (0, 0, 0) \to (1, 0, 0) \to (1, 1, 0) \to (1, 1, 1),
  \end{equation*}
  as is easy to see quickly:
  \begin{itemize}
    \item The function $\Delta^{3} \to \Delta^{1}$ corresponding to the first coordinate is $a^{1}$, so the first coordinate of the first vertex is 0, and the first coordinate of the rest of the vertices are 1.

    \item The function $\Delta^{3} \to \Delta^{1}$ corresponding to the second coordinate is $a^{2}$, so the second coordinates of the first two points are equal to zero, and the second coordinate of the remaining vertices is 1.

    \item The function $\Delta^{3} \to \Delta^{1}$ corresponding to the third coordinate is $a^{3}$, so the third coordinates of the first three points is equal to zero, and the second coordinate of the final vertex is equal to 1.
  \end{itemize}
  Using \hyperref[eq:simplicial_identity_for_cubes]{Equation~\ref*{eq:simplicial_identity_for_cubes}}, we then calculate that
  \begin{equation*}
    d_{2} (a^{1}, a^{2}, a^{3}) = (a^{1}, a^{2}, a^{2}).
  \end{equation*}
  This corresponds to the 2-simplex in $C^{n}$ with spine
  \begin{equation*}
    (0, 0, 0) \to (1, 0, 0) \to (1, 1, 1).
  \end{equation*}
\end{example}

Our next task is to understand the boundary of the $n$-cube. We can view the boundary of the cube $C^{n}$ as the union of its faces.

\begin{definition}
  \label{def:boundary_of_n-cube}
  The \defn{boundary of the $n$-cube} is the simplicial subset
  \begin{equation}
    \label{eq:boundary_of_cube}
    \partial C^{n} := \bigcup_{i = 1}^{n} \bigcup_{j = 0}^{1} \overset{(1)}{\Delta^{1}} \times \cdots \times \overset{(i)}{\{j\}} \times \cdots \times \overset{(n)}{\Delta^{1}} \subset C^{n}.
  \end{equation}
\end{definition}

The following is easy to see.
\begin{proposition}
  \label{prop:simplex_intersect_faces}
  An $(n-1)$-simplex
  \begin{equation*}
    \gamma = (\gamma_{1}, \ldots, \gamma_{n})\colon \Delta^{n-1} \to C^{n}
  \end{equation*}
  lies entirely within the face
  \begin{equation*}
    \overset{(1)}{\Delta^{1}} \times \cdots \times \overset{(i)}{\{0\}} \times \cdots \times \overset{(n)}{\Delta^{1}}
  \end{equation*}
  if and only if $\gamma_{i} = a^{n} = \const_{0}$, and to the face
  \begin{equation*}
    \overset{(1)}{\Delta^{1}} \times \cdots \times \overset{(i)}{\{1\}} \times \cdots \times \overset{(n)}{\Delta^{1}}
  \end{equation*}
  if and only if $\gamma_{i} = a^0 = \const_{1}$.
\end{proposition}

\begin{definition}
  \label{def:left_box}
  The \defn{left box of the $n$-cube}, denoted $LC^{n}$, is the simplicial subset of $C^{n}$ given by the union of all of the faces in \hyperref[eq:boundary_of_cube]{Equation~\ref*{eq:boundary_of_cube}} except for
  \begin{equation*}
    \Delta^{1} \times \cdots \times \Delta^{1} \times \Delta^{\{1\}}.
  \end{equation*}

  Note that drawing the box with the final coordinate going from left to right, the right face is open here; the terminology is chosen to match with `left horn'.
\end{definition}

\begin{example}
  The simplicial subset $LC^{2} \subset C^{2}$ can be drawn as follows.
  \begin{equation*}
    \begin{tikzcd}
      {(0, 0)}
      \arrow[r]
      \arrow[d]
      & {(0, 1)}
      \\
      {(1, 0)}
      \arrow[r]
      & {(1, 1)}
    \end{tikzcd}
  \end{equation*}
\end{example}

Note that in $LC^{2}$, the morphisms $(0,0) \to (1, 0) \to (1, 1)$ form an inner horn, which we can fill by pushing out along an inner horn inclusion. Our main goal in this section is to show that this is generically true: given a left cube $LC^{n}$, we can fill much of $C^{n}$ using pushouts along inner horn inclusions. To this end, it will be helpful to know how each nondegenerate simplex in $C^{n}$ intersects $LC^{n}$.

\begin{lemma}
  \label{lemma:first_and_last_face}
  Let $\phi\colon \Delta^{n} \to C^{n}$ be a nondegenerate $n$-simplex corresponding to the permutaton $\tau \in S_{n}$. We have the following.
  \begin{enumerate}
    \item The zeroth face $d_{0}\phi$ belongs to $LC^{n}$ if and only if $\tau(n) \neq 1$.

    \item For $0 < i < n$, the face $d_{i}\phi$ never belongs to $LC^{n}$.

    \item The $n$th face $d_{n}\phi$ always belongs to $LC^{n}$.
  \end{enumerate}
\end{lemma}
\begin{proof}
  \begin{enumerate}
    \item We have
      \begin{equation*}
        d_{0}\phi = (a^{\tau(0) - 1}, \ldots a^{\tau(n)-1})
      \end{equation*}
      since $\tau(i) > 0$ for all $i$. For $j = \tau^{-1}(1)$, we have that $\tau(j) = 1$, so the $j$th entry of $d_{0}\phi$ is $a^{0}$. Thus, \hyperref[prop:simplex_intersect_faces]{Proposition~\ref*{prop:simplex_intersect_faces}} guarantees that $d_{0}\phi$ is contained in the face
      \begin{equation*}
        \overset{(1)}{\Delta^{1}} \times \cdots \times \overset{(j)}{\{1\}} \times \cdots \times \overset{(n)}{\Delta^{1}}.
      \end{equation*}
      This face belongs to $LC^{n}$ except when $j = n$.

    \item In this case, the superscript of each entry of $d_{i}\phi$ is between $1$ and $n-1$, hence not equal to $n$ or $0$.

    \item In this case,
      \begin{equation*}
        d_{n}(a^{\tau(1)}, \ldots, a^{\tau(n)}) = (a^{\tau(1)}, \ldots, a^{\tau(n)})
      \end{equation*}
      since $n \geq \tau(i)$ for all $1 \leq i \leq n$. Thus, the $\tau^{-1}(n) = j$th entry is $a^{n}$, so $d_{n}\phi$ belongs to the face
      \begin{equation*}
        \overset{(1)}{\Delta^{1}} \times \cdots \times \overset{(j)}{\{0\}} \times \cdots \times \overset{(n)}{\Delta^{1}}.
      \end{equation*}
  \end{enumerate}
\end{proof}

This is promising: it means that for many simplices in $C^{n}$ which we want to fill relative to $LC^{n}$, we already have the data of the first and last faces. The following lemma will allow us to take advantage of this.

\begin{notation}
  For any subset $T \subseteq [n]$, write
  \begin{equation*}
    \Lambda^{n}_{T} := \bigcup_{t \in T} d_{t} \Delta^{n} \subset \Delta^{n}.
  \end{equation*}
\end{notation}

\begin{lemma}
  \label{lemma:subset_of_faces_inner_anodyne}
  For any proper subset $T \subset [n]$ containing 0 and $n$, the inclusion $\Lambda^{n}_{T} \hookrightarrow \Delta^{n}$ is inner anodyne.
\end{lemma}
\begin{proof}
  Induction. For $n = 2$, $T$ must be equal to $\{0, 2\}$, so $\Delta^{2}_{T} = \Lambda^{2}_{1}$.

  Assume the result holds for $n-1$, and let $T \subset [n]$ be a proper subset containing $0$ and $n$. Using the inductive step, we can fill all but one of the faces $d_{1}\Delta^{n}, \ldots, d_{n-1}\Delta^{n}$. The result follows
\end{proof}

We can use \hyperref[lemma:subset_of_faces_inner_anodyne]{Lemma~\ref*{lemma:subset_of_faces_inner_anodyne}} to show that we can fill much of $C^{n}$ as a sequence of inner anodyne pushouts, but to do this we need to pick an order in which to fill our simplices. We do this as follows.

\begin{definition}
  \label{def:order_on_nondegenerate_simplices}
  We define a total order on $S_{n}$ as follows. For any two permutations $\tau$, $\tau' \in S_{n}$, we say that $\tau < \tau'$ if there exists $k \in [n]$ such that the following conditions are satisfied.
  \begin{itemize}
    \item For all $i < k$, we have that $\tau^{-1}(i) = \tau'^{-1}(i)$.

    \item We have that $\tau^{-1}(k) < \tau'^{-1}(k)$.
  \end{itemize}

  We then define a total order on the nondegenerate simplices of $C^{n}$ by saying that $\phi(\tau) < \phi(\tau')$ if and only if $\tau < \tau'$.
\end{definition}

Note that this is \emph{not} the lexicographic order on $S_{n}$; instead, we have that $\tau < \tau'$ if and only if $\tau^{-1}$ is less than $\tau'^{-1}$ under the lexicographic order.

\begin{example}
  The elements of the permutation group $S_{3}$ have the order
  \begin{equation*}
    (1,2,3) < (1,3,2) < (2,1,3) < (3,1,2) < (2,3,1) < (3,2,1).
  \end{equation*}
\end{example}

We use this ordering to define our filtration.

\begin{notation}
  For $\tau \in S_{n}$, we mean by
  \begin{equation*}
    \bigcup_{\tau' \in S_{n}}^{\tau}
  \end{equation*}
  the union over all $\tau' \in S_{n}$ for which $\tau' < \tau$ with respect to the ordering defined above. Note the strict inequality.
\end{notation}

We would like to show that each step of the filtration
\begin{align}
  \label{eq:inclusions_inner_anodyne}
  \begin{split}
    LC^{n} &\hookrightarrow LC^{n} \cup \phi(1, \ldots, n) \\
    & \hookrightarrow \cdots \\
    & \hookrightarrow LC^{n} \cup \bigcup_{\tau \in S_{n}}^{(2, \ldots, n, 1)} \phi(\tau)
  \end{split}
\end{align}
is inner anodyne. To do this, it suffices to show that for each $\tau < (2, \ldots, n, 1)$, the intersection
\begin{equation}
  \label{eq:intersections_inner_anodyne}
  B_{\tau} = \left(LC^{n} \cup \bigcup_{\tau' \in S_{n}}^{\tau} \phi(\tau')\right) \cap \phi(\tau)
\end{equation}
is of the form $\Lambda^{n}_{T}$ for some proper $T \subseteq [n]$ containing 0 and $n$.

\begin{proposition}
  \label{prop:inner_anodyne_cube_filling}
  Each inclusion in \hyperref[eq:inclusions_inner_anodyne]{Equation~\ref*{eq:inclusions_inner_anodyne}} is inner anodyne.
\end{proposition}
\begin{proof}
  The first inclusion
  \begin{equation*}
    LC^{n} \hookrightarrow LC^{n} \cup \phi(1, \ldots, n)
  \end{equation*}
  is inner anodyne because by \hyperref[lemma:first_and_last_face]{Lemma~\ref*{lemma:first_and_last_face}} the intersection $LC^{n} \cap \phi(1, \ldots, n)$ is of the form $d_{0}\Delta^{n} \cup d_{n} \Delta^{n}$.

  Each subsequent intersection contains the faces $d_{0}\Delta^{n}$ and $d_{n}\Delta^{n}$ (again by \hyperref[lemma:first_and_last_face]{Lemma~\ref*{lemma:first_and_last_face}}), so by \hyperref[lemma:subset_of_faces_inner_anodyne]{Lemma~\ref*{lemma:subset_of_faces_inner_anodyne}}, it suffices to show that for each $\tau$ under consideration, there is at least one face not shared with any previous $\tau'$.

  To this end, fix $0 < i < n$, and consider $\tau \in S_{n}$ such that $\tau \neq (1, \ldots, n)$. We consider the face $d_{i}\phi(\tau)$. Let $j = \tau^{-1}(i)$ and $j' = \tau^{-1}(i+1)$. Then $d_{i} a^{\tau(j)} = d_{i}a^{\tau(j')}$, so $d_{i} \phi(\tau) = d_{i}\phi(\tau \circ \gamma_{jj'})$, where $\gamma_{jj'} \in S_{n}$ is the permutation swapping $j$ and $j'$. Thus, the face $d_{i}\sigma(\tau)$ is equal to the face $d_{i}\sigma(\tau \circ \gamma_{jj'})$, which is already contained in the union under consideration if and only $\tau \circ \gamma_{jj'} < \tau$. This in turn is true if and only if $j < j'$.

  Assume that every face of $\phi(\tau)$ is contained in the union. Then
  \begin{equation*}
    \tau^{-1}(1) < \tau^{-1}(2) < \cdots \tau^{-1}(n),
  \end{equation*}
  which implies that $\tau = (1, \ldots, n)$. This is a contradiction. Thus, each boundary inclusion under consideration is of the form $\Lambda^{n}_{T} \hookrightarrow \Delta^{n}$, for $T \subset [n]$ a proper subset containing 0 and $n$, and is thus inner anodyne.
\end{proof}

Each nondegenerate simplex of $C^{n}$ which we have yet to fill is of the form $\phi(\tau)$, where $\tau(n) = 1$. Thus, each of their spines begins
\begin{equation*}
  (0, \ldots, 0, 0) \to (0, \ldots, 0, 1) \to \cdots,
\end{equation*}
and then remains confined to the face $\Delta^{1} \times \cdots \times \Delta^{1} \times \{1\}$. This implies that the part of $LC^{n}$ yet to be filled is of the form $\Delta^{0} \ast C^{n-1}$, and the boundary of this with respect to which we must do the filling is of the form $\Delta^{0} \ast \partial C^{n-1}$.

Our next task is to show that we can also perform this filling, under the assumption that $(0, \ldots, 0, 0) \to (0, \ldots, 0, 1)$ is marked. We do not give the details of this proof, at is quite similar to the proof of \hyperref[prop:inner_anodyne_cube_filling]{Proposition~\ref*{prop:inner_anodyne_cube_filling}}. One continues to fill simplices in the order prescribed in \hyperref[def:order_on_nondegenerate_simplices]{Definition~\ref*{def:order_on_nondegenerate_simplices}}, and shows that each of these inclusions is marked anodyne. The main tool is the following lemma, proved using the same inductive argument as \hyperref[lemma:subset_of_faces_inner_anodyne]{Lemma~\ref*{lemma:subset_of_faces_inner_anodyne}}.

\begin{lemma}
  \label{lemma:subset_of_faces_marked_anodyne}
  For any subset $T \subset [n]$ containing 1 and $n$, and \emph{not} containing 0, the inclusion
  \begin{equation*}
    (\Lambda^{n}_{T}, \mathcal{E}) \hookrightarrow (\Delta^{n}, \mathcal{F})
  \end{equation*}
  is (cocartesian-)marked anodyne, where by $\mathcal{F}$ we mean the set of degenerate $1$-simplces together with $\Delta^{\{0, 1\}}$, and by $\mathcal{E}$ we mean the restriction of this marking to $\Lambda^{n}_{T}$.
\end{lemma}

We collect our major results from this section in the following proposition.

\begin{proposition}
  \label{prop:cube_filling}
  Denote by $J^{n} \subseteq C^{n}$ the simplicial subset spanned by those nondegenerate $n$-simplices whose spines do not begin $(0, \ldots, 0, 0) \to (0, \ldots, 0, 1)$. We can write the filling $LC^{n} \hookrightarrow C^{n}$ as a composition
  \begin{equation*}
    \begin{tikzcd}
      LC^{n}
      \arrow[r, hook, "i"]
      & J^{n}
      \arrow[r, hook, "j"]
      & C^{n}
    \end{tikzcd},
  \end{equation*}
  where $i$ is inner anodyne, and $j$ fits into a pushout square
  \begin{equation*}
    \begin{tikzcd}
      \Delta^{0} \ast \partial C^{n-1}
      \arrow[dr, phantom, "\ulcorner"{at end}]
      \arrow[r, hook]
      \arrow[d, hook]
      & \Delta^{0} \ast C^{n-1}
      \arrow[d, hook]
      \\
      J^{n}
      \arrow[r, hook]
      & C^{n}
    \end{tikzcd},
  \end{equation*}
  where the top inclusion underlies a marked anodyne morphism
  \begin{equation*}
    (\Delta^{0} \ast \partial C^{n-1}, \mathcal{G}) \hookrightarrow (\Delta^{0} \ast C^{n-1}, \mathcal{G}),
  \end{equation*}
  where the marking $\mathcal{G}$ contains all degenerate morphisms together with the morphism $(0, \ldots, 0, 0) \to (0, \ldots,0, 1)$.

\end{proposition}

\subsection{Adjunct data}
\label{ssc:adjunct_data}

When confronted with a lifting problem, one frequently finds a lift by passing to an adjoint lifting problem which has a solution, and transporting the solution back along the adjunction. This relies on the fact that providing data on one side of an adjunction is, in a certain sense, equivalent to providing adjunct data on the other side. In this section, we give a formalization of the the notion of adjunct data in the $\infty$-categorical context. Our main result is \hyperref[prop:can_transport_adjunct_data]{Proposition~\ref*{prop:can_transport_adjunct_data}}, which shows that under certain conditions, we can use data on one side of an adjunction of an $\infty$-categories to fill in data on the other side. First, we recall somewhat explicitly the definition of an adjunction given in \cite{highertopostheory}.

\begin{definition}
  \label{def:adjunction}
  An \defn{adjunction} is a bicartesian fibration $p\colon \mathcal{M} \to \Delta^{1}$. We say that $p$ is \defn{associated} to functors $f\colon \category{C} \to \category{D}$ and $g\colon \category{D} \to \category{C}$ if there exist equivalences $h_{0}\colon \category{C} \to \category{M}_{0}$ and $h_{1}\colon \category{D} \to \category{M}_{1}$, and a commutative diagram
  \begin{equation*}
    \begin{tikzcd}
      \category{C} \times \Delta^{1}
      \arrow[r, "u"]
      \arrow[dr, swap, "\pr"]
      & \mathcal{M}
      \arrow[d, "p"]
      & \category{D} \times \Delta^{1}
      \arrow[l, swap, "v"]
      \arrow[dl, "\pr"]
      \\
      & \Delta^{1}
    \end{tikzcd}
  \end{equation*}
  such that the following conditions are satisfied.
  \begin{itemize}
    \item \emph{The map $u$ is associated to the functor $f$:} the restriction $u|\category{C} \times \{0\} = h_{0}$, the restriction $u|\category{C} \times \{1\} = f \circ h_{1}$, and for each $c \in \category{C}$, the edge $u|\{c\} \times \Delta^{1}$ is $p$-cocartesian.

    \item \emph{The map $v$ is associated to the functor $g$:} the restriction $v|\category{D} \times \{0\} = g \circ h_{0}$, the restriction $u|\category{D} \times \{1\} = h_{1}$, and for each $d \in \category{D}$, the edge $u|\{d\} \times \Delta^{1}$ is $p$-cartesian.
  \end{itemize}

  If $f$ and $g$ are functors to which an adjunction is associated as above, then we say that $f$ is left adjoint to $g$, and equivalently that $g$ is right adjoint to $f$.
\end{definition}

\begin{note}
  As noted in \cite{highertopostheory}, this terminology is slightly imprecise; it would be more correct to say that $p$ is associated to $f$ and $g$ via the data $(h_{0}, h_{1}, u, v)$.
\end{note}

\begin{note}
  It is shown in \cite{highertopostheory} that if $f$ and $g$ are functors such that $f$ is left adjoint to $g$, then we can choose an adjunction $p\colon \mathcal{M} \to \Delta^{1}$ and data $(h_{0}, h_{1}, u, v)$ such that $h_{0}$ and $h_{1}$ are isomorphisms. We can use this to identify $\mathcal{M}_{0}$ with $\category{C}$, and $\mathcal{M}_{1}$ with $\category{D}$. In what follows we will always assume that we have chosen such data, and will thus leave the isomorphisms $h_{0}$ and $h_{1}$ implicit.
\end{note}

\begin{definition}
  \label{def:adjunct_data}
  Let $s\colon K \to \Delta^{1}$ be a map of simplicial sets whose fibers we denote by $K_{0}$ and $K_{1}$, and let $p\colon \category{M} \to \Delta^{1}$ be a bicartesian fibration associated to adjoint functors
  \begin{equation*}
    f : \category{C} \longleftrightarrow \category{D} : g
  \end{equation*}
  via data $(u\colon \category{C} \times \Delta^{1} \to \category{M}, v\colon \category{D} \times \Delta^{1} \to \category{M})$. We say that a map $\alpha\colon K \to \category{C}$ is \defn{adjunct} to a map $\tilde{\alpha}\colon K \to \category{D}$ relative to $s$, and equivalently that $\tilde{\alpha}$ is adjunct to $\alpha$ relative to $s$, if there exists a map $A\colon K \times \Delta^{1} \to \category{M}$ such that the diagram
  \begin{equation*}
    \begin{tikzcd}
      K \times \Delta^{1}
      \arrow[r, "A"]
      \arrow[dr, swap, "\pr_{\Delta^{1}}"]
      & \category{M}
      \arrow[d, "p"]
      \\
      & \Delta^{1}
    \end{tikzcd}
  \end{equation*}
  commutes, and such that the following conditions are satisfied.
  \begin{enumerate}
    \item The restriction $A|_{K \times \{0\}} = \alpha$.

    \item The restriction $A|_{K \times \{1\}} = \tilde{\alpha}$.

    \item The restriction $A|K_{0} \times \Delta^{1}$ is equal to the composition
      \begin{equation*}
        K_{0} \times \Delta^{1} \hookrightarrow K \times \Delta^{1} \overset{\alpha \circ \times \id}{\to} \category{C} \times \Delta^{1} \overset{u}{\to} \mathcal{M}.
      \end{equation*}

    \item The restriction $A|K_{1} \times \Delta^{1}$ is equal to the composition
      \begin{equation*}
        K_{1} \times \Delta^{1} \hookrightarrow K \times \Delta^{1} \overset{\tilde{\alpha} \times \id}{\to} \category{D} \times \Delta^{1} \overset{v}{\to} \mathcal{M}.
      \end{equation*}
  \end{enumerate}
\end{definition}

In the next examples, fix adjoint functors
\begin{equation*}
  f : \category{C} \longleftrightarrow \category{D} : g
\end{equation*}
corresponding to a bicartesian fibration $p\colon \category{M} \to \Delta^{1}$ via data $(u, v)$.

\begin{example}
  Any morphism in $\category{D}$ of the form $fC \to D$ is adjunct to some morphism in $\category{C}$ of the form $C \to gD$, witnessed by any square
  \begin{equation*}
    \begin{tikzcd}
      C
      \arrow[r, "a"]
      \arrow[d]
      & fC
      \arrow[d]
      \\
      gD
      \arrow[r, "b"]
      & D
    \end{tikzcd}
  \end{equation*}
  in $\category{M}$ where $a$ is $u|\{c\} \times \Delta^{1}$ and $b$ is $v|\{d\} \times \Delta^{1}$. This corresponds to the map $s = \id\colon \Delta^{1} \to \Delta^{1}$.
\end{example}

\begin{example}
  Pick some object $D \in \category{D}$, and consider the identity morphism $\id\colon gD \to gD$ in $\category{C}$. This morphism is adjunct to the component of the unit map $\eta_{D}\colon D \to fgD$ relative to $s = \id\colon \Delta^{1} \to \Delta^{1}$.
\end{example}

\begin{example}
  For any simplicial sets $K$ and $K'$, any diagram $K \ast K' \to \category{D}$ such that $K$ is in the image of $f$ is adjunct to some diagram $K \ast K' \to \category{C}$ such that $K'$ is in the image of $g$. This will follow from \hyperref[prop:can_transport_adjunct_data]{Proposition~\ref*{prop:can_transport_adjunct_data}}.
\end{example}

\begin{lemma}
  \label{lemma:left_cylinder_inclusion_marked_anodyne}
  The inclusion of marked simplicial sets
  \begin{equation}
    \label{eq:cocartesian_anodyne_morphism}
    \left( \Delta^{n} \times \Delta^{\{0\}} \coprod_{\partial \Delta^{n} \times \Delta^{\{0\}}} \partial \Delta^{n} \times \Delta^{1}, \mathcal{E} \right) \hookrightarrow (\Delta^{n} \times \Delta^{1}, \mathcal{F})
  \end{equation}
  where the marking $\mathcal{F}$ is the flat marking together with the edge $\Delta^{\{0\}} \times \Delta^{1}$, and $\mathcal{E}$ is the restriction of this marking, is marked (cocartesian) anodyne.
\end{lemma}

The next proposition shows the power of adjunctions in lifting problems: given data on either side of an adjunction, we can fill in adjunct data on the other side.

\begin{proposition}
  \label{prop:can_transport_adjunct_data}
  Let $\category{M} \to \Delta^{1}$, $\category{C}$, $\category{D}$, $f$, $g$, $u$ and $v$ be as in \hyperref[def:adjunction]{Definition~\ref*{def:adjunction}}, and let
  \begin{equation*}
    \begin{tikzcd}[column sep=small]
      K'
      \arrow[rr, hook, "i"]
      \arrow[dr, swap, "s \circ i"]
      && K
      \arrow[dl, "s"]
      \\
      & \Delta^{1}
    \end{tikzcd},
  \end{equation*}
  be a commuting triangle of simplicial sets, where $i$ is a monomorphim such that $i|_{\{1\}}$ is an isomorphism. Let $\tilde{\alpha}'\colon K' \to \category{D}$ and $\alpha\colon K \to \category{C}$ be maps, and denote $\alpha' = \alpha \circ i$. Suppose that $\tilde{\alpha}'$ is adjunct to $\alpha'$ relative to $s \circ i$. Then there exists a dashed extension
  \begin{equation}
    \label{eq:adjunct_filling_problems}
    \begin{tikzcd}
      K'
      \arrow[r, "\alpha'"]
      \arrow[d, swap, hook, "i"]
      & \category{C}
      \\
      K
      \arrow[ur, swap, "\alpha"]
    \end{tikzcd}
    \quad \rightsquigarrow \quad
    \begin{tikzcd}
      K'
      \arrow[r, "\tilde{\alpha}'"]
      \arrow[d, swap, hook, "i"]
      & \category{D}
      \\
      K
      \arrow[ur, dashed, swap, "\tilde{\alpha}"]
    \end{tikzcd}
  \end{equation}
  such that $\tilde{\alpha}$ is adjunct to $\alpha$ relative to $s$. Furthermore, any two such lifts are equivalent as functors $K \to \category{D}$.

\end{proposition}
\begin{proof}
  Pick some commutative diagram
  \begin{equation*}
    \begin{tikzcd}
      K' \times \Delta^{1}
      \arrow[r, "A"]
      \arrow[dr, swap, "\pr_{\Delta^{1}}"]
      & \category{M}
      \arrow[d]
      \\
      & \Delta^{1}
    \end{tikzcd}
  \end{equation*}
  displaying $\tilde{\alpha}'$ and $\alpha'$ as adjunct. From the map $A$ and the map $\alpha\colon K \to \category{C} \cong \category{M}_{0}$, we can construct the solid commutative square of cocartesian-marked simplicial sets
  \begin{equation*}
    \begin{tikzcd}
      \left(K \times \Delta^{\{0\}} \coprod_{K' \times \Delta^{\{0\}}} K' \times \Delta^{1}\right)^{\heartsuit}
      \arrow[r]
      \arrow[d, swap, "w"]
      & \category{M}^{\natural}
      \arrow[d]
      \\
      (K \times \Delta^{1})^{\heartsuit}
      \arrow[ur, dashed, "\ell"]
      \arrow[r]
      & (\Delta^{1})^{\sharp}
    \end{tikzcd},
  \end{equation*}
  where $(K \times \Delta^{1})^{\heartsuit}$ is the marked simplicial set where the only nondenerate morphisms marked are those of the form $\{k\} \times \Delta^{1}$ for $k \in K|_{s^{-1}\{0\}}$, and the $\heartsuit$-marked pushout-product denotes the restriction of this marking to the pushout-product. We can write the morphism of simplicial sets underlying $w$ as a smash-product $(K' \hookrightarrow K) \wedge (\Delta^{\{0\}} \hookrightarrow \Delta^{1})$. Building $K' \hookrightarrow K$ simplex by simplex, in increasing order of dimension, we can write $w$ as a transfinite composition of pushouts along inclusions of the form given in \hyperref[eq:cocartesian_anodyne_morphism]{Equation~\ref*{eq:cocartesian_anodyne_morphism}} (by assumption each simplex $\sigma$ we adjoin to $K$ is not completely contained in $K'$, so the initial vertex of $\sigma$ must lie in the fiber of $s$ over $0$). We can thus also build our lift $\ell$ by lifting against each of these in turn, and we can always choose each lift so that it is compatible with $v$. We can then take $\tilde{\alpha} = \ell|K \times \{1\}$.

  It is manifest from this technique that the space of lifts is contractible; in particular, any two such lifts are equivalent in $\Map(K \times \Delta^{1}, \category{M})$, so their restrictions to $K \times \{1\}$ are equivalent as functors $K \to \category{D}$.

\end{proof}

\begin{example}
  Let $f : \category{C} \leftrightarrow \category{D} : g$ be an adjunction of $1$-categories, giving an adjunction between quasicategories upon taking nerves. Consider the diagram
  \begin{equation*}
    \begin{tikzcd}[column sep=small]
      K' = \Delta^{\{0,1,3\}} \cup \Delta^{\{0,2,3\}}
      \arrow[dr]
      \arrow[rr, hook, "i"]
      && \Delta^{3} = K
      \arrow[dl]
      \\
      & \Delta^{1}
    \end{tikzcd},
  \end{equation*}
  where both downwards-facing arrows take $0$, $1 \mapsto 0$ and $2$, $3 \mapsto 1$, and fix a map $\tilde{\alpha}\colon K \to N(\category{C})$ and a map $\alpha'\colon K' \to N(\category{D})$ such that $\tilde{\alpha} \circ i$ is adjunct to $\alpha'$. This corresponds to the diagrams
  \begin{equation*}
    \overbrace{
      \begin{tikzcd}[ampersand replacement=\&]
        C
        \arrow[r]
        \arrow[d]
        \& gD
        \arrow[d]
        \\
        C'
        \arrow[r]
        \arrow[ur, dashed]
        \& gD'
      \end{tikzcd}
    }^{\text{in } \category{C}}
    \qquad\qquad
    \overbrace{\begin{tikzcd}[ampersand replacement=\&]
      fC
      \arrow[r]
      \arrow[d]
      \& D
      \arrow[d]
      \\
      fC'
      \arrow[r]
      \& D'
    \end{tikzcd}}^{\text{in }\category{D}}
  \end{equation*}
  where the solid diagrams in each category are adjunct to one another.
  \hyperref[prop:can_transport_adjunct_data]{Proposition~\ref*{prop:can_transport_adjunct_data}} implies that the solution to the lifting problem on the left yields one on the right. Similarly, its dual implies the converse. Thus, \hyperref[prop:can_transport_adjunct_data]{Proposition~\ref*{prop:can_transport_adjunct_data}} implies in particular that such lifting problems are equivalent.
\end{example}


We note that in the proof of \hyperref[prop:can_transport_adjunct_data]{Proposition~\ref*{prop:can_transport_adjunct_data}}, we did not need to use the full power of the statement that $f$ was left adjoint to $g$ (i.e.\ that $p\colon \category{M} \to \Delta^{1}$ was a bicartesian fibration), only that $f$ was \emph{locally} left adjoint to $g$, (i.e.\ that $p$ admitted cocartesian lifts at certain objects). This allows us to generalize \hyperref[prop:can_transport_adjunct_data]{Proposition~\ref*{prop:can_transport_adjunct_data}} somewhat.

\begin{definition}
  \label{def:partial_adjunction}
  Let $g\colon \category{D} \to \category{C}$ be a functor between quasicategories, and pick some cartesian fibration $p\colon \category{M} \to \Delta^{1}$ which classifies it. We assume without loss of generality that $h_{0}\colon \category{M}_{0} \cong \category{C}$ and  $h_{1}\colon \category{M}_{1} \cong \category{D}$ are isomorphisms, and we will notationally suppress them.

  We say that $g$ \defn{admits a left adjoint at $c \in \category{C}$} if $p$ admits a cocartesian lift of the morphism $\id_{\Delta^{1}}$ with source $c \in \category{M}_{0}$.
\end{definition}

\begin{definition}
  \label{def:adjunct_local}
  Let $s\colon K \to \Delta^{1}$ be a map of simplicial sets whose fibers we denote by $K_{0}$ and $K_{1}$, and let $p\colon \category{M} \to \Delta^{1}$ be a cartesian fibration associated to $f\colon \category{D} \to \category{C}$ via a map $v\colon \category{D} \times \Delta^{1} \to \category{M}$. We say that a map $\alpha\colon K \to \category{C}$ is \defn{adjunct} to a map $\tilde{\alpha}\colon K \to \category{D}$ relative to $s$, and equivalently that $\tilde{\alpha}$ is adjunct to $\alpha$ relative to $s$, if there exists a map $A\colon K \times \Delta^{1} \to \category{M}$ such that the diagram
  \begin{equation*}
    \begin{tikzcd}
      K \times \Delta^{1}
      \arrow[r, "A"]
      \arrow[dr, swap, "\pr_{\Delta^{1}}"]
      & \category{M}
      \arrow[d, "p"]
      \\
      & \Delta^{1}
    \end{tikzcd}
  \end{equation*}
  commutes, and such that the following conditions are satisfied.
  \begin{enumerate}
    \item The restriction $A|_{K \times \{0\}} = \alpha$.

    \item The restriction $A|_{K \times \{1\}} = \tilde{\alpha}$.

    \item For each vertex $k \in K_{0}$, the image of $\{k\} \times \Delta^{1}$ under $A$ in $\category{M}$ is $p$-cartesian.

    \item The restriction $A|K_{1} \times \Delta^{1}$ is equal to the composition
      \begin{equation*}
        K_{1} \times \Delta^{1} \hookrightarrow K \times \Delta^{1} \overset{\tilde{\alpha} \times \id}{\to} \category{D} \times \Delta^{1} \overset{v}{\to} \mathcal{M}.
      \end{equation*}
  \end{enumerate}
\end{definition}

\begin{example}
  \label{eg:partial_adjunctions_and_kan_extensions}
  Suppose $f\colon \category{C} \to \category{C}'$ and $F\colon \category{C} \to \category{D}$ are functors. The statement that the left Kan extension $f_{!}F\colon \category{C'} \to \category{D}$ exists is equivalent to the statement that $f^{*}$ admits a left adjoint at $F$.
\end{example}

The proof of \hyperref[prop:can_transport_adjunct_data]{Proposition~\ref*{prop:can_transport_adjunct_data}} can be used as is to show the following statement.

\begin{proposition}
  \label{prop:can_transport_adjunct_data_partial}
  Let $p\colon \category{M} \to \Delta^{1}$ classify $g\colon \category{D} \to \category{C}$ via $v\colon \category{D} \times\Delta^{1} \to \category{M}$, and let
  \begin{equation*}
    \begin{tikzcd}[column sep=small]
      K'
      \arrow[rr, hook, "i"]
      \arrow[dr, swap, "s \circ i"]
      && K
      \arrow[dl, "s"]
      \\
      & \Delta^{1}
    \end{tikzcd},
  \end{equation*}
  be a commuting triangle of simplicial sets, where $i$ is a monomorphim such that $i|_{\{1\}}$ is an isomorphism. Let $\tilde{\alpha}'\colon K' \to \category{D}$ and $\alpha\colon K \to \category{C}$ be maps, and denote $\alpha' = \alpha \circ i$. Suppose that $\tilde{\alpha}'$ is adjunct to $\alpha'$ relative to $s \circ i$. Further suppose that $g$ admits a left adjoint at all vertices belonging to the image $\alpha(K_{0}) \subseteq \category{C}$.

  Then there exists a dashed extension
  \begin{equation}
    \label{eq:adjunct_filling_problems_partial}
    \begin{tikzcd}
      K'
      \arrow[r, "\alpha'"]
      \arrow[d, swap, hook, "i"]
      & \category{C}
      \\
      K
      \arrow[ur, swap, "\alpha"]
    \end{tikzcd}
    \quad \rightsquigarrow \quad
    \begin{tikzcd}
      K'
      \arrow[r, "\tilde{\alpha}'"]
      \arrow[d, swap, hook, "i"]
      & \category{D}
      \\
      K
      \arrow[ur, dashed, swap, "\tilde{\alpha}"]
    \end{tikzcd}
  \end{equation}
  such that $\tilde{\alpha}$ is adjunct to $\alpha$ relative to $s$. Furthermore, any two such lifts are equivalent as functors $K \to \category{D}$.
\end{proposition}

\subsection{Left Kan implies globally left Kan}
\label{ssc:left_kan_implies_globally_left_kan}

In this subsection, we will prove \hyperref[thm:left_kan_implies_globally_left_kan]{Theorem~\ref*{thm:left_kan_implies_globally_left_kan}}. In order to do that, we must show that we can solve lifting problems of the form
\begin{equation*}
  \label{eq:specific_left_horns_in_cat_infty_bicategories}
  \begin{tikzcd}
    \Delta^{\{0, 1, n+1\}}_{\flat}
    \arrow[d, hook]
    \arrow[dr, "\tau"]
    \\
    (\Lambda^{n+1}_{0})_{\flat}
    \arrow[r]
    \arrow[d, hook]
    & \ICCat
    \\
    \Delta^{n+1}_{\flat}
    \arrow[ur, dashed]
  \end{tikzcd},
\end{equation*}
for $n \geq 2$, where $\tau$ is left Kan. Forgetting for a moment the upper triangle singling out the $2$-simplex $\tau$, and using the adjunction $\Csc \dashv \Nsc$, we will do this by solving equivalent lifting problems of the form\footnote{Note that we are implicitly taking $\C[\Lambda^{n+1}_{0}]$ and $\C[\Delta^{n+1}]$ to carry the flat markings on their mapping spaces.}
\begin{equation}
  \label{eq:left_horns_in_cat_infty_scaled}
  \begin{tikzcd}
    \C[\Lambda^{n+1}_{0}]
    \arrow[r]
    \arrow[d, hook]
    & \QCat
    \\
    \C[\Delta^{n+1}]
    \arrow[ur, dashed, swap, "\mathcal{F}"]
  \end{tikzcd}.
\end{equation}
We expect that the reader is familiar with the basics of (scaled) rigidification, so we give only a rough description of this lifting problem here. The objects of the simplicial category $\C[\Delta^{n+1}]$ are given by the set $\{0, \ldots, n+1\}$, and the mapping spaces are defined by
\begin{equation*}
  \C[\Delta^{n+1}](i, j) =
  \begin{cases}
    N(P_{ij}), &i \leq j \\
    \emptyset, & i > j
  \end{cases}.
\end{equation*}

where $P_{ij}$ is the poset of subsets of the linearly ordered set $\{i, \ldots, j\}$ containing $i$ and $j$, ordered by inclusion. The simplicial category $\C[\Lambda^{n+1}_{0}]$ has the same objects as $\C[\Delta^{n+1}]$, and each morphism space $\C[\Lambda^{n+1}_{0}](i, j)$ is a simplicial subset of $\C[\Delta^{n+1}](i, j)$, as we shall soon describe.

The map $\C[\Lambda^{n+1}_{0}](i, j) \hookrightarrow \C[\Delta^{n+1}](i, j)$ is an isomorphism except for $(i, j) = (1, n+1)$ and $(i, j) = (0, n+1)$; in these cases, it is an inclusion. The missing data corresponds in the case of $(1, n+1)$ to the missing face $d_{0}\Delta^{n+1}$ of $\Lambda^{n+1}_{0}$, and in the case of $(0, n+1)$ to the missing interior. Thus, to find a filling as in \hyperref[eq:left_horns_in_cat_infty_scaled]{Equation~\ref*{eq:left_horns_in_cat_infty_scaled}}, we need to solve the lifting problems
\begin{equation}
  \label{eq:lifting_problem_1}
  \begin{tikzcd}
    \C[\Lambda^{n+1}_{0}](1, n+1)
    \arrow[r]
    \arrow[d, hook]
    & \Fun(\mathcal{F}(1), \mathcal{F}(n+1))
    \\
    \C[\Delta^{n+1}](1, n+1)
    \arrow[ur, dashed, swap, "\ell'"]
  \end{tikzcd}
\end{equation}
and
\begin{equation}
  \label{eq:lifting_problem_2}
  \begin{tikzcd}
    \C[\Lambda^{n+1}_{0}](0, n+1)
    \arrow[r]
    \arrow[d, hook]
    & \Fun(\mathcal{F}(0), \mathcal{F}(n+1))
    \\
    \C[\Delta^{n+1}](0, n+1)
    \arrow[ur, dashed, swap, "\ell"]
  \end{tikzcd}.
\end{equation}
However, these problems are not independent; the filling $\ell$ of the full simplex needs to agree with the filling $\ell'$ we found for the missing face of the horn, corresponding to the condition that the square
\begin{equation}
  \label{eq:lifting_problem_compatibility_condition}
  \begin{tikzcd}
    \C[\Delta^{n+1}](1, n+1)
    \arrow[r, "\ell'"]
    \arrow[d, swap, "{\{0,1\}}^{*}"]
    & \Fun(\mathcal{F}(1), \mathcal{F}(n+1))
    \arrow[d, "\mathcal{F}({\{0, 1\}})^{*}"]
    \\
    \C[\Delta^{n+1}](0, n+1)
    \arrow[r, "\ell"]
    & \Fun(\mathcal{F}(0), \mathcal{F}(n+1))
  \end{tikzcd}
\end{equation}
commute.

\begin{notation}
  \label{notation:simplicial_lifting}
  Recall our desired filling $\mathcal{F}\colon \C[\Delta^{n+1}] \to \QCat$ of \hyperref[eq:left_horns_in_cat_infty_scaled]{Equation~\ref*{eq:left_horns_in_cat_infty_scaled}}. We will denote $\mathcal{F}(i)$ by $X_{i}$, and for each subset $S = \{i_{1}, \ldots, i_{k}\} \subseteq [n]$, we will denote $\category{F}(S) \in \Map(X_{i_{1}}, X_{i_{k}})$ by $f_{i_{k}\cdots i_{1}}$. For any inclusion $S' \subseteq S \subseteq [n]$ preserving minimim and maximum elements, we will denote the corresponding morphism by $\alpha^{S'}_{S}$.
\end{notation}

\begin{theorem}
  For any $n \geq 2$ and any globally left Kan $2$-simplex $\tau$, the solid extension problem
  \begin{equation*}
    \begin{tikzcd}
      \Delta^{\{0, 1, n+1\}}_{\flat}
      \arrow[d, hook]
      \arrow[dr, "\tau"]
      \\
      (\Lambda^{n+1}_{0})_{\flat}
      \arrow[r]
      \arrow[d, hook]
      & \ICCat
      \\
      \Delta^{n+1}_{\flat}
      \arrow[ur, dashed]
    \end{tikzcd}
  \end{equation*}
  admits a dashed filler.
\end{theorem}

\begin{note}
  The author would like to offer the friendly recommendation that in reading the proof below, it is helpful to follow along with the explanation given at the beginning of \hyperref[sec:horn_filling_via_left_kan_extensions]{Section~\ref*{sec:horn_filling_via_left_kan_extensions}}.
\end{note}

\begin{proof}
  We need to solve the lifting problems of \hyperref[eq:lifting_problem_1]{Equation~\ref*{eq:lifting_problem_1}} and \hyperref[eq:lifting_problem_2]{Equation~\ref*{eq:lifting_problem_2}}, and check that our solutions satisfy the condition of \hyperref[eq:lifting_problem_compatibility_condition]{Equation~\ref*{eq:lifting_problem_compatibility_condition}}. In the discussuion surrounding these equations, we ignored that the simplex $\tau\colon \Delta^{\{0,1,n+1\}}_{\flat} \to \ICCat$ was left Kan. We now reintroduce this information. The information that the simplex $\tau\colon \Delta^{\{0,1,n+1\}}_{\flat} \to \ICCat$ is left Kan tells us the following.
  \begin{itemize}
    \item  Concretely, it tells us that the map $f_{n,1}$ is the left Kan extension of $f_{n, 0}$ along $f_{1, 0}$, and that $\alpha_{\{0,n+1\}}^{\{0,1,n+1\}}\colon f_{n+1,0} \to f_{n+1, 1} \circ f_{1, 0} $ is the unit map.

    \item In particular, \hyperref[eg:partial_adjunctions_and_kan_extensions]{Example~\ref*{eg:partial_adjunctions_and_kan_extensions}} then tells us that the map $f_{1,0}^{*}$ admits a left adjoint at $f_{n+1,0}$.
  \end{itemize}

  The lifting problems as described above are somewhat unwieldy, given in terms of mapping spaces of simplicial categories $\C[\Delta^{n+1}]$ and $\C[\Lambda^{n+1}_{0}]$. It will be useful to give these mapping spaces more down-to-earth descriptions, in terms of simplicial sets with which we are familiar. We have the following succinct descriptions of the inclusions of \hyperref[eq:lifting_problem_1]{Equation~\ref*{eq:lifting_problem_1}} and \hyperref[eq:lifting_problem_2]{Equation~\ref*{eq:lifting_problem_2}} along which we have to extend.
  \begin{itemize}
    \item There is an isomorphism of simplicial sets
      \begin{equation*}
        \C[\Delta^{n+1}](0, n+1) \overset{\cong}{\to} C^{n}
      \end{equation*}
      specified completely by sending a subset $S \subseteq [n+1]$ to the point
      \begin{equation*}
        (z_{1}, \ldots, z_{n}) \in C^{n},\qquad z_{i} =
        \begin{cases}
          1, & i \in S \\
          0, &\text{otherwise.}
        \end{cases}
      \end{equation*}

      The inclusion $\C[\Lambda^{n+1}_{0}](0, n+1) \hookrightarrow \C[\Delta^{n+1}](0, n+1)$ corresponds to the simplicial subset $LC^{n} \hookrightarrow C^{n}$.

    \item There is a similar isomorphism of simplicial sets
      \begin{equation*}
        \C[\Delta^{n+1}](1, n+1) \overset{\cong}{\to} C^{n-1}
      \end{equation*}
      specified completely by sending $S \subseteq \{1, \ldots, n+1\}$ to the point
      \begin{equation*}
        (z_{1}, \ldots, z_{n-1}) \in C^{n-1},\qquad z_{i} =
        \begin{cases}
          1, & i+1 \in S \\
          0, &\text{otherwise.}
        \end{cases}
      \end{equation*}
      The inclusion $\C[\Lambda^{n+1}_{0}](1, n+1) \hookrightarrow \C[\Delta^{n+1}](1, n+1)$ corresponds to the inclusion $\partial C^{n-1} \hookrightarrow C^{n-1}$.
  \end{itemize}

  Using these descriptions, we can write our lifting problems in the more inviting form
  \begin{equation*}
    (*) =
    \begin{tikzcd}
      LC^{n}
      \arrow[r]
      \arrow[d, hook]
      & \Fun(X_{0}, X_{n+1})
      \\
      C^{n}
      \arrow[ur, dashed]
    \end{tikzcd},\qquad (**) =
    \begin{tikzcd}
      \partial C^{n-1}
      \arrow[r, "\tilde{\alpha}'"]
      \arrow[d, hook]
      & \Fun(X_{1}, X_{n+1})
      \\
      C^{n-1}
      \arrow[ur, dashed]
    \end{tikzcd},
  \end{equation*}
  and our condition becomes that the square
  \begin{equation*}
    (\star) =
    \begin{tikzcd}
      C^{n-1}
      \arrow[r]
      \arrow[d, hook]
      & \Fun(X_{1}, X_{n+1})
      \arrow[d, "f_{1,0}^{*}"]
      \\
      C^{n}
      \arrow[r]
      & \Fun(X_{0}, X_{n+1})
    \end{tikzcd}
  \end{equation*}
  commutes, where the left-hand vertical morphism is the inclusion of the right face.

  Using \hyperref[prop:cube_filling]{Proposition~\ref*{prop:cube_filling}}, we can partially solve the lifting problem $(*)$, reducing it to the lifting problem
  \begin{equation*}
    (*') =
    \begin{tikzcd}
      \Delta^{0} \ast \partial C^{n-1}
      \arrow[r, "\alpha"]
      \arrow[d, hook]
      & \Fun(X_{0}, X_{n+1})
      \\
      \Delta^{0} \ast C^{n-1}
      \arrow[ur, dashed]
    \end{tikzcd}.
  \end{equation*}
  The image of the 1-simplex $(0, \ldots, 0) \to (0, \ldots, 1)$ of $\Delta^{0} \ast \partial C^{n-1} \subseteq C^{n}$ under $\alpha$ is the unit map $\alpha^{\{0, 1, n+1\}}_{\{0, n+1\}}\colon f_{n+1, 0} \to f_{n+1, 1} \circ f_{1, 0}$. This is not in general an equivalence, so we cannot use \hyperref[prop:cube_filling]{Proposition~\ref*{prop:cube_filling}} to solve the lifting problem $(*')$ directly. However, the unit map is adjunct to an equivalence in $\Fun(X_{1}, X_{n+1})$ relative to $s = \id_{\Delta^{1}}$. Furthermore, the restriction of $\alpha$ to $\partial C^{n-1}$ is in the image of $f^{*}_{1,0}$, so by \hyperref[prop:can_transport_adjunct_data_partial]{Proposition~\ref*{prop:can_transport_adjunct_data_partial}} and our assumption that $f_{10}^{*}$ admits a left adjoint at $f_{n+1,0}$, we can augment the map $\tilde{\alpha}'$ to a map $\tilde{\alpha}\colon \Delta^{0} \ast \partial C^{n-1} \to \Fun(X_{1}, X_{n-1})$ which is adjunct to $\alpha$ relative to the map
  \begin{equation*}
    s\colon \Delta^{0} \ast \partial C^{n-1} \to \Delta^{1}
  \end{equation*}
  sending $\Delta^{0}$ to $\Delta^{\{0\}}$ and $\partial C^{n-1}$ to $\Delta^{\{1\}}$; in particular, the image of the morphism $(0, \ldots, 0) \to (0, \ldots 1)$ under $\tilde{\alpha}'$ is an equivalence. This allows us to replace the lifting problem $(**)$ by the superficially more complicated lifting problem
  \begin{equation*}
    (**') =
    \begin{tikzcd}
      \Delta^{0} \ast \partial C^{n-1}
      \arrow[r, "\tilde{\alpha}"]
      \arrow[d, hook]
      & \Fun(X_{1}, X_{n+1})
      \\
      \Delta^{0} \ast C^{n-1}
      \arrow[ur, dashed, swap, "\tilde{\beta}"]
    \end{tikzcd}
  \end{equation*}
  However, since image of $(0, \ldots 0) \to (0, \ldots, 1)$ is an equivalence, \hyperref[prop:cube_filling]{Proposition~\ref*{prop:cube_filling}} implies that we can solve the lifting problem $(**')$. Again using (the dual to) \hyperref[prop:can_transport_adjunct_data_partial]{Proposition~\ref*{prop:can_transport_adjunct_data_partial}}, we can transport this filling to a solution to the lifting problem $(*')$. The condition $(\star)$ amounts to demanding that the restriction $\beta|_{C^{n-1}}$ be the image of $\tilde{\beta}|_{C^{n-1}}$ under $f_{1,0}^{*}$, which is true by construction.
\end{proof}

\subsection{A few other tricks with Kan extensions}
\label{ssc:a_few_other_tricks_with_kan_extensions}

We end this section with a few miscellaneous tricks which we can play with maps $\Delta^{n} \to \ICCat$ involving left Kan simplices.

\begin{lemma}
  \label{lemma:transport_left_kan_simplices}
  Denote by $\mathcal{E} \subseteq \Hom(\Delta^{2}, \Delta^{3})$ the subset containing all degenerate $2$-simplices together with the simplices $\Delta^{\{0,2,3\}}$ and $\Delta^{\{1,2,3\}}$. Let $\sigma\colon \Delta^{3}_{\mathcal{E}} \to \ICCat$ such that $\Delta^{\{2,3\}}$ is mapped to an equivalence. Then $\sigma|\Delta^{\{0,1,2\}}$ is left Kan if and only if $\sigma|\Delta^{\{0,1,3\}}$ is left Kan.
\end{lemma}
\begin{proof}
  Replacing thin simplices by strict compositions using \hyperref[prop:homotopy_invariance_of_witness]{Proposition~\ref*{prop:homotopy_invariance_of_witness}}, this statement reduces to that of \hyperref[prop:kan_ext_invariance_of_target]{Proposition~\ref*{prop:kan_ext_invariance_of_target}}.
\end{proof}

The next lemma is similar but easier.
\begin{lemma}
  \label{lemma:transport_thin_simplices}
  Denote by $\mathcal{E} \subseteq \Hom(\Delta^{2}, \Delta^{3})$ the subset containing all degenerate $2$-simplices together with the simplices $\Delta^{\{0,2,3\}}$ and $\Delta^{\{1,2,3\}}$. Let $\sigma\colon \Delta^{3}_{\mathcal{E}} \to \ICCat$ such that $\Delta^{\{2,3\}}$ is mapped to an equivalence. Then $\sigma|\Delta^{\{0,1,2\}}$ is thin if and only if $\sigma|\Delta^{\{0,1,3\}}$ is thin.
\end{lemma}

\begin{lemma}
  \label{lemma:compose_left_kan_simplices}
  Denote by $\mathcal{E}' \subseteq \Hom(\Delta^{2}, \Delta^{3})$ the subset containing all degenerate $2$-simplices, together with the simplex $\Delta^{\{0,1,2\}}$. Let $\sigma\colon \Delta^{3}_{\mathcal{F}} \to \ICCat$ such that $\sigma|\Delta^{\{0,1,3\}}$ is left Kan. Then $\sigma|\Delta^{\{0,2,3\}}$ is left Kan if and only if $\sigma|\Delta^{\{1,2,3\}}$ is left Kan.
\end{lemma}
\begin{proof}
  Replacing thin simplices by strict compositions using \hyperref[prop:homotopy_invariance_of_witness]{Proposition~\ref*{prop:homotopy_invariance_of_witness}}, this reduces to \hyperref[prop:kan_extend_along_composition]{Propsition~\ref*{prop:kan_extend_along_composition}}.
\end{proof}

\section{Local systems}
\label{sec:local_systems}

For any $\infty$-category $\category{C}$, a $\category{C}$-local system on some space $X$ is simply a functor $X \to \category{C}$. Thus, the $\infty$-category of $\category{C}$-local systems on $X$ is simply $\LS(\category{C})_{X}:=\Fun(X, \category{C})$. We would like to consider local systems on all spaces $X$ simultaneously, defining an $\infty$-category $\LS(\category{C})$ of $\category{C}$-local systems.

We follow the general strategy laid out in \cite{luriehopkins2013ambidexterity}. We will consider a cartesian fibration 
\begin{equation*}
  p'\colon\smallint\Fun(-, -) \to \ICat \times \ICat\op
\end{equation*}
which classifies the functor
\begin{equation*}
  \Fun(-,-)\colon \ICat\op \times \ICat \to \ICat;\qquad (\category{C}, \category{D}) \mapsto \Fun(\category{C}, \category{D}).
\end{equation*}
The strict pullback $p$ of $p'$ in the diagram
\begin{equation*}
  \begin{tikzcd}
    \LS(\category{C})
    \arrow[r]
    \arrow[d, swap, "p"]
    & \int\Fun(-, -)
    \arrow[d, "p'"]
    \\
    \S \times \{\category{C}\}
    \arrow[r]
    & \ICat \times \ICat\op
  \end{tikzcd}
\end{equation*}
will then classify the functor
\begin{equation*}
  \S \to \ICat;\qquad X \mapsto \LS(\category{C})_{X}.
\end{equation*}
The total space $\LS(\category{C})$ of $p$ will thus be our candidate for our $\infty$-category of local systems. The fibration $p$ remembers the source of the local system.

In \cite{garcia2020enhanced}, a suitable model for $\int\Fun(-, -)$ is given: the so-called \emph{enhanced twisted arrow category} $\Tw(\ICCat)$. In \hyperref[ssc:the_twisted_arrow_category]{Subsection~\ref*{ssc:the_twisted_arrow_category}}, we reproduce the pertinent points of \cite{garcia2020enhanced}, defining the enhanced twisted arrow category. In \hyperref[ssc:the_infinity_category_of_local_systems]{Subsection~\ref*{ssc:the_infinity_category_of_local_systems}}, we use the enhanced twisted arrow category to define the category $\LS(\category{C})$ of local systems, together with a cartesian fibration $p\colon \LS(\category{C}) \to \S$ classifying the functor
\begin{equation*}
  \S\op \to \ICat;\qquad X \to \Fun(X, \category{C}).
\end{equation*}
We note that this functor sends a map of spaces $f\colon X \to Y$ to the pullback functor
\begin{equation*}
  f^{*}\colon \Fun(Y, \category{C}) \to \Fun(X, \category{C}).
\end{equation*}
If $\category{C}$ is cocomplete, each functor $f^{*}$ has a left adjoint $f_{!}$ given by left Kan extension. By abstract nonsense, the cartesian fibration $p$ is also a cocartesian fibration, whose cocartesian edges correspond to left Kan extension. In \hyperref[ssc:the_fibration]{Subsection~\ref*{ssc:the_fibration}} we will show this explicitly, using results built up in \hyperref[ssc:two_lemmas_about_marked_scaled_anodyne_morphisms]{Subsection~\ref*{ssc:two_lemmas_about_marked_scaled_anodyne_morphisms}} and \hyperref[ssc:the_infinity_category_of_local_systems]{Subsection~\ref*{ssc:the_infinity_category_of_local_systems}}; if the reader is willing to take this on faith, these sections can be safely skipped.

\begin{note}
  The reader may notice that we are being inefficient here. If we were only interested in the construction outlined above, then using the twisted arrow category would be much more combinatorially strenuous than necessary; the objects of the twisted arrow category $\Tw(\ICCat)$ are functors of $\infty$-categories $\category{D} \to \category{C}$, with $\category{C}$ and $\category{D}$ both allowed to vary. If, in the very next breath, we fix the target $\category{C}$, then it makes more sense to use some sort of overcategory $(\ICCat)_{/\category{C}}$ from the beginning. There is a method to our madness; in \hyperref[sec:the_monoidal_construction]{Section~\ref*{sec:the_monoidal_construction}}, we will introduce a monoidal version of this construction, and here we will need to allow the targets of our functors to vary.
\end{note}

\subsection{The enhanced twisted arrow category}
\label{ssc:the_twisted_arrow_category}

For a 2-category $\mathsf{C}$, the twisted arrow 1-category $\Tw(\mathsf{C})$ has the following description.
\begin{itemize}
  \item The objects of $\Tw(\mathsf{C})$ are the morphisms of $\mathsf{C}$:
    \begin{equation*}
      f\colon c \to c'.
    \end{equation*}

  \item For two objects $(f\colon c \to c')$ and $(g\colon d \to d')$, the morphisms $f \to g$ are given by diagrams
    \begin{equation*}
      \begin{tikzcd}
        c
        \arrow[r, "a"]
        \arrow[d, "f"{swap}, ""{name=L}]
        & d
        \arrow[d, "g", ""{name=R, swap}]
        \\
        c'
        & d'
        \arrow[l, "a'"]
        \arrow[from=L, to=R, Rightarrow, shorten=2ex, "\alpha"]
      \end{tikzcd},
    \end{equation*}
    where $\alpha$ is a 2-morphism $f \Rightarrow a' \circ g \circ a$.

  \item The composition of morphisms is given by concatenating the corresponding diagrams.
    \begin{equation*}
      \begin{tikzcd}
        c
        \arrow[r, "a"]
        \arrow[d, "f"{swap}, ""{name=L}]
        & d
        \arrow[d, "g"{description}, ""{name=LM}, ""{name=RM, swap}]
        \arrow[r, "b"]
        & e
        \arrow[d, "h", ""{name=R, swap}]
        \\
        c'
        & d'
        \arrow[l, "a'"]
        & e'
        \arrow[l, "b'"]
        \arrow[from=L, to=RM, Rightarrow, shorten=2ex, "\alpha"]
        \arrow[from=LM, to=R, Rightarrow, shorten=2ex, "\beta"]
      \end{tikzcd}
    \end{equation*}
\end{itemize}

In \cite{garcia2020enhanced}, a homotopy-coherent version of the twisted arrow category of an $\infty$-bicategory is defined. For any $\infty$-bicategory $\CC$, the twisted arrow $\infty$-category of $\CC$, denoted $\Tw(\CC)$, is a quasicategory whose $n$-simplices are diagrams $\Delta^{n} \star (\Delta^{n})\op \to \CC$, together with a scaling to ensure that the information encoded in an $n$-simplex is determined, up to contractible choice, by data
\begin{equation*}
  \overbrace{
    \begin{tikzcd}[ampersand replacement=\&]
      c_{0}
      \arrow[r, "a_{0}"]
      \arrow[d, "f_{0}"{swap}, ""{name=L1}]
      \& c_{1}
      \arrow[d, "f_{1}"{description}, ""{swap, name=R1}]
      \arrow[r, dotted]
      \& \cdots
      \arrow[r, dotted]
      \& c_{n-1}
      \arrow[r, "a_{n-1}"]
      \arrow[d, "f_{n-1}"{description}, ""{name=Ln}]
      \& c_{n}
      \arrow[d, "f_{n}", ""{swap, name=Rn}]
      \\
      c'_{0}
      \& c'_{1}
      \arrow[l, "a'_{0}"]
      \& \cdots
      \arrow[l, dotted]
      \& c'_{n-1}
      \arrow[l, dotted]
      \& c'_{n}
      \arrow[l, "a'_{n-1}"]
      \arrow[from=L1, to=R1, Rightarrow, "\alpha_{0}", shorten=2ex]
      \arrow[from=Ln, to=Rn, Rightarrow, "\alpha_{n-1}", shorten=2ex]
    \end{tikzcd}
  }^{n\text{ squares}}
\end{equation*}
in $\CC$. Here, the top row of morphisms corresponds to the spine of the $n$-simplex $\Delta^{n} \subset \Delta^{n} \star (\Delta^{n})\op$, and the bottom row of morphisms to the spine of $(\Delta^{n})\op \subset \Delta^{n} \star (\Delta^{n})\op$. To make this correspondence more clear, we will introduce the following notation.

\begin{notation}
  For $i \in [n] \subset [2n+1]$, we will write $\overline{i} := 2n+1-i$. Thus, $\overline{0} = 2n+1$, $\overline{1} = 2n$, etc.
\end{notation}

Thus, the $i$th column in the above diagram corresponds to the image of the morphism $i \to \overline{i}$ in $\Delta^{n} \star (\Delta^{n})\op$.

The scaling mentioned above is defined in the following way.

\begin{definition}
  We define a cosimplicial object $\tilde{Q}\colon \Delta \to \SSetsc$ by sending
  \begin{equation*}
    \tilde{Q}([n]) = (\Delta^{n} \star (\Delta^{n})\op, \dagger),
  \end{equation*}
  where $\dagger$ is the scaling consisting of all degenerate 2-simplices, together with all 2-simplices of the following kinds:
  \begin{enumerate}
    \item All simplices $\Delta^{2} \to \Delta^{n} \star (\Delta^{n})\op$ factoring through $\Delta^{n}$.

    \item All simplices $\Delta^{2} \to \Delta^{n} \star (\Delta^{n})\op$ factoring through $(\Delta^{n})\op$.

    \item All simplices $\Delta^{\{i, j, \overline{k}\}} \subseteq \Delta^{n} \star (\Delta^{n})\op$, $i < j \leq k$.

    \item All simplices $\Delta^{\{k, \overline{j}, \overline{i}\}} \subseteq \Delta^{n} \star (\Delta^{n})\op$, $i < j \leq k$.
  \end{enumerate}
\end{definition}

Note that $\Delta^{n} \star (\Delta^{n})\op \cong \Delta^{2n+1}$. We will use this identification freely.

This extends to a nerve-realization adjunction
\begin{equation*}
  Q : \SSet \longleftrightarrow \SSetsc : \Tw,
\end{equation*}
where the functor $Q$ is the extension by colimits of the functor $\tilde{Q}$, and for any scaled simplicial set $X$, the simplicial set $\Tw(X)$ has $n$-simplices
\begin{equation*}
  \Tw(\CC)_{n} = \Hom_{\SSetsc}(Q(\Delta^{n}), X).
\end{equation*}

\begin{notation}
  For any simplicial set $X$, $Q(X)$ carries the scaling given simplex-wise by that described above. We will denote this also by $\dagger$. Using the identification $\Delta^{n} \star (\Delta^{n})\op \cong \Delta^{2n+1}$, we can write $Q(\Delta^{n})$ explicitly and compactly as $\Delta^{2n+1}_{\dagger}$, which we will often do.
\end{notation}

The inclusion $\Delta^{n} \amalg (\Delta^{n})\op \hookrightarrow \Delta^{n} \star (\Delta^{n})\op$ provides, for any $\infty$-bicategory $\CC$ with underlying $\infty$-category $\category{C}$, a morphism of simplicial sets
\begin{equation*}
  \Tw(\CC) \to \category{C} \times \category{C}\op.
\end{equation*}
In \cite{garcia2020enhanced}, the following is shown.

\begin{theorem}
  \label{thm:mainthm_walker_fernando}
  Let $\CC$ be an $\infty$-bicategory (presented as a fibrant scaled simplicial set), and let $\category{C}$ be the underlying $\infty$-category (presented as a quasicategory). Then the map
  \begin{equation*}
    p_{\CC}\colon \Tw(\CC) \to \category{C} \times \category{C}\op.
  \end{equation*}
  is a cartesian fibration between quasicategories, and a morphism $f\colon \Delta^{1} \to \Tw(\CC)$ is $p_{\CC}$-cartesian if and only if the morphism $\sigma\colon \Delta^{3}_{\dagger} \to \CC$ to which it is adjunct is fully scaled, i.e.\ factors through the map $\Delta^{3}_{\dagger} \to \Delta^{3}_{\sharp}$. Furthermore, the cartesian fibration $p_{\CC}$ classifies the functor
  \begin{equation*}
    \Map_{\CC}(-, -)\colon \category{C}\op \times \category{C} \to \ICat.
  \end{equation*}
\end{theorem}

%

\subsection{Two lemmas about marked-scaled anodyne morphisms}
\label{ssc:two_lemmas_about_marked_scaled_anodyne_morphisms}

In the section following this one, we provide a fairly explicit construction of a family of marked-scaled anodyne inclusions. Constructing these inclusions directly from the classes of generating marked-scaled anodyne inclusions given in \hyperref[def:ms-anodyne_morphisms]{Definition~\ref*{def:ms-anodyne_morphisms}} would be possible but tedious. To lighten this load somewhat, we reproduce in this section two lemmas from \cite{garcia2cartesianfibrationsii}. The majority of this section is a retelling of \cite[Sec.~2.3~and~2.4]{garcia2cartesianfibrationsii}. Note however that our needs will be rather different than those of \cite{garcia2cartesianfibrationsii}, and this is reflected in some minor differences of notation.

We first need a compact notation for specifying simplicial subsets of $\Delta^{n}$. Denote by $P\colon \Set \to \Set$ the power set functor.

\begin{definition}
  Let $T$ be a finite linearly ordered set, so $T \cong [n]$ for some $n \in \N$. For any subset $\mathcal{A} \subseteq  P(T)$, define a simplicial subset
  \begin{equation*}
    \S^{\mathcal{A}}_{T} := \bigcup_{S \in \mathcal{A}} \Delta^{T \smallsetminus S} \subseteq \Delta^{T}.
  \end{equation*}
\end{definition}

\begin{notation}
  If the linearly ordered set $T$ is clear from context, we will drop it, writing $\S^{\mathcal{A}}$.
\end{notation}

\begin{note}
  \label{note:correspondence_simplicial_subsets_power_set}
  The assignment $\mathcal{A} \mapsto \S^{\mathcal{A}}$ does not induce is not a one-to-one correspondence between subsets $\mathcal{A} \subseteq P([n])$ and simplicial subsets $S^{\mathcal{A}} \subseteq \Delta^{n}$, since for any two subsets $S \subsetneq S' \subset [n]$, we have that $\Delta^{[n] \smallsetminus S'} \subsetneq \Delta^{[n] \smallsetminus S}$.
\end{note}

\begin{example}
  Take $T = [2]$. Taking $\mathcal{A} = \{\{2\}, \{1,2\}\}$ gives the simplicial subset $\S^{\mathcal{A}}_{T} = \Delta^{\{0,1\}} \cup \Delta^{\{0\}} = \Delta^{\{0,1\}} \subset \Delta^{2}$, as does taking $\mathcal{A} = \{\{2\}\}$.
\end{example}

It will be useful to consider two subsets $\mathcal{A}$ and $\mathcal{A}' \subseteq P(T)$ to be equivalent if they produce the same simplicial subset of $\Delta^{T}$.

\begin{definition}
  We will write $\mathcal{A} \sim \mathcal{A}'$ if $\S^{\mathcal{A}} = \S^{\mathcal{A}'}$.
\end{definition}

The relation $\sim$ is an equivalence relation, but we will not use this.

The set $P([n])$ forms a poset, ordered by inclusion, and any $\mathcal{A} \subseteq P([n])$ a subposet. An element $q$ of a poset $Q$ is said to be \emph{minimal} if there are no elements of $Q$ which are strictly less than $q$. By \hyperref[note:correspondence_simplicial_subsets_power_set]{Note~\ref*{note:correspondence_simplicial_subsets_power_set}} only the minimal elements of $\mathcal{A}$ contribute to the union defining $\S^{\mathcal{A}}$. We have just shown the following.

\begin{lemma}
  \label{lemma:replace_poset_by_minimal_elements}
  For any $\mathcal{A} \subseteq P([n])$, we have $\mathcal{A} \sim \mathrm{min}(\mathcal{A})$, where $\mathrm{min}(\mathcal{A}) \subseteq \mathcal{A}$ is the set of minimal elements of $\mathcal{A}$.
\end{lemma}

Simplicial subsets of the form $\S^{\mathcal{A}}_{T}$ enjoy the following easily-proved calculation rules.

\begin{lemma}
  \label{lemma:subset_of_simplex_contains_k_simplices}
  Suppose $\mathcal{A} \subseteq P([n])$ is a subset such that for all $S$, $T \in \mathcal{A}$, it holds that $S \cap T = \emptyset$. Then $\S^{\mathcal{A}}$ contains each $k$-simplex of $\Delta^{n}$ for all $k < \abs{\mathcal{A}} - 1$.
\end{lemma}
\begin{proof}
  Let $X \subseteq [n]$. The simplex $\Delta^{X} \to \Delta^{n}$ factors through $\S^{\mathcal{A}}$ if and only if it factors through $\Delta^{[n] \smallsetminus T}$ for some (possibly not unique) $T \in \mathcal{A}$. This in turn is true if and only if $X$ and $T$ do not have any elements in common. For $\abs{X} < \abs{\mathcal{A}}$, there is always a set $T \in \mathcal{A}$ which does not have any elements in common with $X$.
\end{proof}

\begin{lemma}
  \label{lemma:add_a_simplex}
  Let $\mathcal{A} \subseteq P([n])$ and $T \subseteq [n]$. Then
  \begin{equation*}
    \S^{\mathcal{A}}_{[n]} \cup \Delta^{T} = \S^{\mathcal{A} \cup \{[n] \smallsetminus T\}}_{[n]}.
  \end{equation*}
\end{lemma}

The next lemma will be particularly useful in building inclusions $\S^{\mathcal{A}}_{[n]} \hookrightarrow \Delta^{n}$ simplex-by-simplex.

\begin{lemma}
  \label{lemma:bicartesian_square}
  Let $\mathcal{A} \subset P([n])$, and let $T \subset [n]$. Then the square
  \begin{equation*}
    \begin{tikzcd}
      \S^{\mathcal{A}|T}_{T}
      \arrow[r, hook]
      \arrow[d, hook]
      & \Delta^{T}
      \arrow[d, hook]
      \\
      \S^{\mathcal{A}}_{[n]}
      \arrow[r, hook]
      & \S^{\mathcal{A}}_{[n]} \cup \Delta^{T}
    \end{tikzcd}
  \end{equation*}
  is bicartesian, where $\mathcal{A}|T$ is the subset of $P(T)$ given by
  \begin{equation*}
    \mathcal{A}|T = \{S \cap T \mid S \in \mathcal{A}\}.
  \end{equation*}
\end{lemma}
\begin{proof}
  We have an equality
  \begin{equation*}
    \mathcal{S}_{T}^{\mathcal{A}|T} = \S^{\mathcal{A}}_{[n]} \cap \Delta^{T}.
  \end{equation*}
  as subsets of $\Delta^{n}$.
\end{proof}

\begin{notation}
  For any marked-scaled $n$-simplex $(\Delta^{n}, E, T)$ we will by minor abuse of notation reuse the same letters $E$ and $T$ to denote the restriction of the markings and scalings to any simplicial subset $S \subseteq \Delta^{n}$.
\end{notation}

In the remainder of this section we reproduce two lemmas from \cite{garcia2cartesianfibrationsii} which provide criteria for the inclusion $(\S^{\mathcal{A}}, E, T) \subseteq (\Delta^{n}, E, T)$ to be marked-scaled anodyne.

\begin{definition}
  Let $\mathcal{A} \subseteq P([n])$. We call $X \in P([n])$ an \defn{$\mathcal{A}$-basal set} if it contains precisely one element from each $S \in \mathcal{A}$. We denote the set of all $\mathcal{A}$-basal sets by $\Bas(\mathcal{A})$.
\end{definition}

\begin{definition}[\cite{garcia2020enhanced}, Definition 1.3]
  \label{def:inner_dull}
  We will call a subset $\mathcal{A} \subseteq P([n])$ \defn{inner dull} if it satisfies the following conditions.
  \begin{itemize}
    \item It does not include the empty set; $\emptyset \notin \mathcal{A}$.

    \item There exists $0 < i < n$ such that for all $S \in \mathcal{A}$, we have $i \notin S$.

    \item For every $S$, $T \in \mathcal{A}$, it follows that $S \cap T = \emptyset$.

    \item For each $\mathcal{A}$-basal set $X \in P([n])$, there exist $u$, $v \in X$ such that $u < i < v$.
  \end{itemize}

  The element $i \in [n]$ is known as the \emph{pivot point.}
\end{definition}

\begin{note}
  The last condition of \hyperref[def:inner_dull]{Definition~\ref*{def:inner_dull}} is always satisfied if $\mathcal{A}$ contains two singletons $\{u\}$ and $\{v\}$ such that $u < i < v$.
\end{note}

\begin{definition}
  Let $\mathcal{A} \subseteq P([n])$ be an inner dull subset with pivot point $i$, and let $X \in \mathcal{A}$. We will denote the adjacent elements of $X$ surrounding $i$ by $\ell^{X} < i < u^{X}$.
\end{definition}

\begin{lemma}[The Pivot Trick: \cite{garcia2cartesianfibrationsii}, Lemma 2.3.5]
  \label{lemma:pivot_trick}
  Let $\mathcal{A} \subseteq P([n])$ be an inner dull subset with pivot point $i$, and let $(\Delta^{n}, E, T)$ be a marked-scaled simplex. Further suppose that the following conditions hold:
  \begin{enumerate}
    \item Every marked edge $e \in E$ which does not contain $i$ factors through $\S^{\mathcal{A}}$.

    \item For all scaled simplices $\sigma = \{a < b < c\}$ of $\Delta^{n}$ which do not factor through $\S^{\mathcal{A}}$, and which do not contain the pivot point $i$, we have $a < i < c$, and $\sigma \cup \{i\}$ is fully scaled.

    \item For all $X \in \Bas(\mathcal{A})$ and all $r$, $s \in [n]$ such that $\ell^{X} \leq r < i < s \leq u^{X}$, the triangle $\{r, i, s\}$ is scaled
  \end{enumerate}

  Then the inclusion $(\mathcal{S}^{\mathcal{A}}, E, T) \hookrightarrow (\Delta^{n}, E, T)$ is marked-scaled anodyne.
\end{lemma}

\begin{definition}
  We will call a subset $\mathcal{A} \subseteq P([n])$ \defn{right dull} if it satisfies the following conditions.
  \begin{itemize}
    \item It is nonempty; $\mathcal{A} \neq \emptyset$.

    \item It does not include the empty set; $\emptyset \notin \mathcal{A}$.

    \item For all $S \in \mathcal{A}$, we have $n \notin S$.

    \item For every $S$, $T \in \mathcal{A}$, it follows that $S \cap T = \emptyset$.
  \end{itemize}

  In this case, we will refer to $n$ as the pivot point.
\end{definition}

\begin{lemma}[Right-anodyne pivot trick]
  \label{lemma:right-anodyne_pivot_trick}
  Let $\mathcal{A} \subset P([n])$ be a right dull subset (whose pivot point is by definition $n$), and let $(\Delta^{n}, E, T)$ be a marked-scaled simplex. Further suppose that the following conditions hold:
  \begin{itemize}
    \item Every marked edge $e \in E$ which does not contain $n$ factors through $\S^{\mathcal{A}}$.

    \item Let $\sigma = \{a < b < c\}$ be a scaled simplex not containing $n$. Then either $\sigma$ factors through $\S^{\mathcal{A}}$, or $\sigma \cup \{n\}$ is fully scaled, and $c \to n$ is marked.

    \item For all Z in  $\Bas(\category{A})$, For all $r$, $s \in [n]$ with $r \leq \min(Z) \leq \max(Z) \leq s < n$, the triangle $\Delta^{\{r, s, n\}}$ is scaled, and the simplex $\Delta^{\{s, n\}}$ is marked.
  \end{itemize}

  Then the inclusion $(\mathcal{S}^{\mathcal{A}}, E, T) \hookrightarrow (\Delta^{n}, E, T)$ is marked-scaled anodyne.
\end{lemma}

\subsection{The infinity-category of local systems}
\label{ssc:the_infinity_category_of_local_systems}

In \hyperref[ssc:the_twisted_arrow_category]{Subsection~\ref*{ssc:the_twisted_arrow_category}}, we defined the twisted arrow category $\Tw(\CC)$ of an $\infty$-bicategory $\CC$, and noted that the natural cartesian fibration $\Tw(\CC) \to \category{C} \times \category{C}\op$ classifies the mapping functor $\Map_{\CC}(-, -)$. Thus, the objects $\Tw(\CC)$ are simply the morphisms in $\CC$.

We can use this to define, for any cocomplete $\infty$-category $\category{C}$, the $\infty$-category of $\category{C}$-local systems; we simply consider $\Tw(\ICCat)$, whose objects are all functors between $\infty$-categories, and restrict to those functors whose domain is an $\infty$-groupoid, and whose codomain is $\category{C}$.

\begin{definition}
  For any cocomplete $\infty$-category $\category{C}$, we define the $\infty$-category of $\category{C}$-local systems $\LS(\category{C})$ together with a map of simplicial sets $p\colon \LS(\category{C}) \to \S$ by the following pullback diagram.
  \begin{equation}
    \label{eq:pullback_square_defining_local_systems}
    \begin{tikzcd}
      \LS(\category{C})
      \arrow[r, hook]
      \arrow[d, swap, "p"]
      & \Tw(\ICCat)
      \arrow[d, "p'"]
      \\
      \S \times \{\category{C}\}
      \arrow[r, hook]
      & \ICat \times \ICat\op
    \end{tikzcd}
  \end{equation}
\end{definition}

\begin{note}
  It is natural to wonder if we are making life unnecessarily difficult by not simply considering the category $\S \times_{\ICCat}(\ICCat)_{/\category{C}}$, which is after all also an $\infty$-category of functors from spaces into $\category{C}$. The reason for the more complicated construction given here will become apparent in \hyperref[ssc:the_monoidal_twisted_arrow_category]{Subsection~\ref*{ssc:the_monoidal_twisted_arrow_category}}, when we define the monoidal structure on $\LS(\category{C})$.
\end{note}

It is shown in \cite{garcia2020enhanced} that $p'$, hence also $p$, is a cartesian fibration. The cartesian fibration $p$ classifies the functor $\S\op \to \ICat$ sending
\begin{equation*}
  f\colon X \to Y \qquad \longmapsto \qquad f^{*}\colon \Fun(Y, \category{C}) \to \Fun(X, \category{C}).
\end{equation*}
Under the assumption that $\category{C}$ is cocomplete, each pullback map $f^{*}$ has a left adjoint $f_{!}$, given by left Kan extension. It follows on abstract grounds that $p$ is also a cocartesian fibration, whose cocartesian edges represent left Kan extension. In \hyperref[ssc:the_fibration]{Subsection~\ref*{ssc:the_fibration}}, we show this explicitly. Our goal in this section is to do some legwork to facilitate the proof of this result.

In investigating the map $p\colon \LS(\category{C}) \to \S$, it will be useful to factor the pullback square in \hyperref[eq:pullback_square_defining_local_systems]{Equation~\ref*{eq:pullback_square_defining_local_systems}} into the two pullback squares
\begin{equation*}
  \begin{tikzcd}
    \LS(\category{C})
    \arrow[r, hook]
    \arrow[d, swap, "p"]
    & \mathcal{R}
    \arrow[r, hook]
    \arrow[d, "p''"]
    & \Tw(\ICCat)
    \arrow[d, "p'"]
    \\
    \S \times \{\category{C}\}
    \arrow[r, hook]
    & \ICat \times [\category{C}]
    \arrow[r, hook]
    & \ICat \times \ICat\op
  \end{tikzcd},
\end{equation*}
where $[\category{C}]$ denotes the path component of $\category{C}$ in $( \ICat\op )^{\simeq}$. In order to show that $p$ is a cocartesian fibration, it will help us to understand the map $p''$. The $n$-simplices of $\category{R}$ are given by maps
\begin{equation*}
  Q(\Delta^{n}) = (\Delta^{n} \star (\Delta^{n})\op)_{\dagger} \to \ICCat
\end{equation*}
such that
\begin{itemize}
  \item each object in $(\Delta^{n})\op \subseteq Q(\Delta^{n})$ is sent to a quasicategory which is equivalent to $\category{C}$ (and in particular cocomplete), and

  \item each morphism in $(\Delta^{n})\op \subseteq Q(\Delta^{n})$ is mapped to an equivalence in $\ICCat$.
\end{itemize}
We can more usefully encode the second condition by endowing $\ICCat$ and $Q$ with a marking.

\begin{definition}
  We define the following markings.
  \begin{itemize}
    \item We denote by $\ICCat^{\natural}$ the marked-scaled simplicial set whose underlying scaled simplicial set is $\ICCat$, and where all equivalences have been marked.

    \item We denote by $\heartsuit$ the marking on $\Delta^{n} \star (\Delta^{n})\op$ consisting of all morphisms belonging to $(\Delta^{n})\op$, and by $(\Delta^{n} \star (\Delta^{n})\op)^{\heartsuit}_{\dagger}$ the corresponding marked-scaled simplicial set.
  \end{itemize}
\end{definition}

\begin{definition}
  \label{def:cosimplicial_obj_R}
  We define a cosimplicial object
  \begin{equation*}
    \tilde{R}\colon \Delta \to \SSetms;\qquad [n] \mapsto (\Delta^{n} \star (\Delta^{n})\op)^{\heartsuit}_{\dagger}.
  \end{equation*}

  We denote the extension of $\tilde{R}$ by colimits by
  \begin{equation*}
    R\colon \SSet \to \SSetms;\qquad X \mapsto \colim_{\Delta^{n} \to X} R(n).
  \end{equation*}
\end{definition}

The marked-scaled simplicial sets $R(\Delta^{n})$ are rather complicated, and showing that $p''$ is a cartesian fibration by solving the necessary lifting problems explicitly would be impractical. We will instead replace $R(\Delta^{n})$ by something simpler, considering the simplicial subsets coming from the inclusions
\begin{equation}
  \label{eq:inclusion_j_into_r}
  \Delta^{n} \star \Delta^{\{\overline{0}\}} \subseteq \Delta^{n} \star (\Delta^{n})\op,
\end{equation}
which we understand to inherit the marking and scaling.

\begin{definition}
  \label{def:cosimplicial_obj_J}
  We will denote by $\tilde{J}\colon \Delta \to \SSetms$ the cosimplicial object
  \begin{equation*}
    \tilde{J}\colon \Delta \to \SSetms;\qquad [n] \mapsto (\Delta^{n} \star \Delta^{\{\overline{0}\}})^{\heartsuit}_{\dagger},
  \end{equation*}
  and by $J$ the extension by colimits
  \begin{equation*}
    J\colon \SSet \to \SSetms;\qquad X \mapsto \colim_{\Delta^{n} \to X} \tilde{J}(n).
  \end{equation*}
\end{definition}

The inclusions of \hyperref[eq:inclusion_j_into_r]{Equation~\ref*{eq:inclusion_j_into_r}} induce for each $n \geq 0$ an inclusion of marked-scaled simplicial sets
\begin{equation*}
  v_{n}\colon \tilde{J}(n) \to \tilde{R}(n).
\end{equation*}

\begin{lemma}
  \label{lemma:lower_morphism_equivalence}
  For all $n \geq 0$, the map $v_{n}$ is marked-scaled anodyne, hence a weak equivalence in the model structure on marked-scaled simplicial sets.
\end{lemma}

We postpone the proof of \hyperref[lemma:lower_morphism_equivalence]{Lemma~\ref*{lemma:lower_morphism_equivalence}} until the end of this section.

\begin{note}
  \label{note:not_natural}
  Some care is warranted: the $v_{n}$ are not the components of a natural transformation $\tilde{J} \Rightarrow \tilde{R}$! The necessary squares simply do not commute for morphisms $\phi\colon [m] \to [n]$ in $\Delta$ such that $\phi(0) \neq 0$. This means that we do not, for a general simplicial set $X$, get a canonical weak equivalence of marked-scaled simplicial sets $J(X) \to R(X)$.
\end{note}

Despite the warning given in \hyperref[note:not_natural]{Note~\ref*{note:not_natural}}, it is still possible to get maps $J(X) \to R(X)$ in some cases. In the remainder of this section, we check that we can produce a marked-scaled weak equivalence $J(\Lambda^{n}_{0}) \to R(\Lambda^{n}_{0})$.

\begin{notation}
  Denote by $\mathring{\Delta}$ the subcategory of $\Delta$ on morphisms $[m] \to [n]$ which send $0 \mapsto 0$. Denote by $I$ the inclusion $\mathring{\Delta} \hookrightarrow \Delta$.
\end{notation}

It is easy to check the following.
\begin{lemma}
  \label{lemma:restricted_natural_transformation}
  The morphisms $v_{n}$ form a natural transformation $v\colon \tilde{J} \circ I \Rightarrow \tilde{R} \circ I$.
\end{lemma}

\begin{lemma}
  \label{lemma:marked_scaled_equivalence_left_horn}
  For all $n \geq 1$, there is a marked-scaled equivalence $J(\Lambda^{n}_{0}) \to R(\Lambda^{n}_{0})$.
\end{lemma}
\begin{proof}
  We can write $J(\Lambda^{n}_{0})$ as a colimit of the composition $\tilde{J} \circ b$ coming from the bottom of the diagram
  \begin{equation*}
    \begin{tikzcd}
      P_{n}
      \arrow[d, hook]
      \arrow[r, "a"]
      & \mathring{\Delta}
      \arrow[d, swap, "I"]
      \arrow[dr, "\tilde{J} \circ I"]
      \\
      (\Delta \downarrow \Lambda^{n}_{0})^{\mathrm{nd}}
      \arrow[r, "b"]
      & \Delta
      \arrow[r, "\tilde{J}"]
      &\SSetms
    \end{tikzcd},
  \end{equation*}
  where $(\Delta \downarrow \Lambda^{n}_{0})^{\mathrm{nd}}$ is the category of nondegenerate simplices of $\Lambda^{n}_{0}$. Denote by $P_{n}$ the poset of proper subsets of $[n]$ such that $0 \in S$. There is an obvious inclusion $P_{n} \hookrightarrow (\Delta \downarrow \Lambda^{n}_{0})^{\mathrm{nd}}$, and one readily checks using Quillen's Theorem A that this inclusion is cofinal. Further note that the functor $P_{n} \to (\Delta \downarrow \Lambda^{n}_{0}) \to \Delta$ factors through $\mathring{\Delta}$ via a map $a\colon P_{n} \to \mathring{\Delta}$. Thus, we can equally express $J(\Lambda^{n}_{0})$ as the colimit of the functor $\tilde{J} \circ I \circ a$. Precisely the same reasoning tells us that we can express $R(\Lambda^{n}_{0})$ as the colimit of $\tilde{R} \circ I \circ a$. The result now follows from \hyperref[lemma:restricted_natural_transformation]{Lemma~\ref*{lemma:restricted_natural_transformation}}, \hyperref[lemma:lower_morphism_equivalence]{Lemma~\ref*{lemma:lower_morphism_equivalence}}, and the fact that each strict colimit is the model for the homotopy colimit.
\end{proof}

\begin{proof}[Proof of Lemma \ref{lemma:lower_morphism_equivalence}]
  In this proof, all simplicial subsets of $\Delta^{2n+1}$ will be assumed to carry the marking $\heartsuit$ and the scaling $\dagger$.

  We can write each $v_{n}$ as a composition
  \begin{equation*}
    \Delta^{\{0, \ldots, n, \overline{0}\}} \overset{v''_{n}}{\hookrightarrow} \Delta^{\{0, \ldots, n, \overline{0}\}} \cup \Delta^{\{\overline{n}, \ldots, \overline{0}\}} \overset{v'_{n}}{\hookrightarrow} \Delta^{2n+1}.
  \end{equation*}
  Here, the map $v_{n}''$ is a pushout along the inclusion $(\Delta^{\{n\}})^{\sharp}_{\sharp} \hookrightarrow (\Delta^{n})^{\sharp}_{\sharp}$, which is marked-scaled anodyne by \hyperref[prop:sharp_marked_right_anodyne]{Proposition~\ref*{prop:sharp_marked_right_anodyne}}. It remains to show that each of the maps $v_{n}'$ is marked-scaled anodyne.

  To this end, we introduce some notation. Let $M^{n}_{0} = \Delta^{\{0, \ldots, n, \overline{0}\}} \cup \Delta^{\{\overline{n}, \ldots \overline{0}\}} \subset \Delta^{2n+1}$, and for $1 \leq k \leq n$, define
  \begin{equation*}
    M^{n}_{k} := M^{n}_{0} \cup \left(\bigcup_{\ell = 1}^{k} \Delta^{[2n+1] \smallsetminus \{\ell, \overline{\ell}\}}\right).
  \end{equation*}
  There is an obvious filtration
  \begin{equation}
    \label{eq:filtration_by_adding_sides}
    M^{n}_{0} \overset{i^{n}_{0}}{\hookrightarrow} M^{n}_{1} \overset{i^{n}_{1}}{\hookrightarrow} \cdots \overset{i^{n}_{n-1}}{\hookrightarrow} M^{n}_{n} \overset{i^{n}_{n}}{\hookrightarrow} \Delta^{2n+1}.
  \end{equation}

  Define
  \begin{equation*}
    j^{n}_{k} := i^{n}_{n} \circ \cdots \circ i^{n}_{k}\colon M^{n}_{k} \hookrightarrow \Delta^{2n+1},\qquad 0 \leq k \leq n.
  \end{equation*}
  In particular, note that $j^{n}_{0} = v_{n}'$.

  This allows us to replace our goal by something superficially more difficult: we would like to show that, for each $n \geq 0$ and each $0 \leq k \leq n$, the map $i^{n}_{k}$ is marked-scaled anodyne. We proceed by induction. We take as our base case $n=0$, where we have the trivial filtration
  \begin{equation*}
    M^{0}_{0} \overset{i^{0}_{0}}{=} \Delta^{1}.
  \end{equation*}
  This is an equality of subsets of $\Delta^{1}$, hence certainly an equivalence.

  We now suppose that the result holds true for $n-1$; that is, that the maps $i^{n-1}_{k}$ are marked-scaled anodyne for all $0 \leq k \leq n-1$. We aim to show that each $i^{n}_{k}$ is marked-scaled anodyne for each $0 \leq k \leq n$.

  For $0 \leq k < n$, we can write $i^{n}_{k}$ as the inclusion
  \begin{equation*}
    \S^{\mathcal{A}}_{[2n+1]} \hookrightarrow \S^{\mathcal{A}}_{[2n+1]} \cup \Delta^{[2n+1] \smallsetminus \{k, \overline{k}\}},\qquad \mathcal{A}
    = \left\{ \substack{ \{\overline{n}, \ldots, \overline{1}\} \\ \{0, \ldots n\} \\ \{1, \overline{1}\} \\ \vdots \\ \{k-1, \overline{k-1}\} } \right\}.
  \end{equation*}
  By \hyperref[lemma:bicartesian_square]{Lemma~\ref*{lemma:bicartesian_square}}, we have a pushout square
  \begin{equation*}
    \begin{tikzcd}
      \S^{\mathcal{A}|([2n+1] \smallsetminus \{k, \overline{k}\})}_{[2n+1] \smallsetminus \{k, \overline{k}\}}
      \arrow[r, hook]
      \arrow[d, hook]
      & \Delta^{[2n+1] \smallsetminus \{k, \overline{k}\}}
      \arrow[d, hook]
      \\
      \S^{\mathcal{A}}_{[2n+1]}
      \arrow[r, hook]
      & \S^{\mathcal{A}}_{[2n+1]} \cup \Delta^{[2n+1] \smallsetminus \{k, \overline{k}\}}.
    \end{tikzcd}.
  \end{equation*}
  Therefore, it suffices to show that the top morphism is marked-scaled anodyne. One checks that this map is of the form $j^{n-1}_{k}$, and is thus marked-scaled anodyne by the inductive hypothesis. Therefore, it remains only to show that $i^{n}_{n}$ is marked-scaled anodyne.

  The case $n = 0$ is an isomorphism, so there is nothing to show. We treat the case $n = 1$ separately. In this case, $i^{1}_{1}$ takes the form
  \begin{equation*}
    \Delta^{\{0,1,\overline{0}\}} \cup \Delta^{\{\overline{1},\overline{0}\}} \overset{v_{1}'}{\hookrightarrow} \Delta^{3},
  \end{equation*}
  which we construct in the following way.
  \begin{enumerate}
    \item First we fill the simplex $\Delta^{\{1,\overline{1},\overline{0}\}}$ (together with its marking and scaling) as a pushout along a morphism of type \ref{item:outerms}.

    \item Then we fill the simplex $\Delta^{\{0,1,\overline{1}\}}$ (together with its marking and scaling) as a pushout along a morphism of type \ref{item:innerms}.

    \item Finally, we fill the full simplex $\Delta^{3}$ (together with its marking and scaling) as a pushout along a morphism of type \ref{item:innerms}.
  \end{enumerate}

  Now we may assume that $n \geq 2$. In this case, we can write $i^{n}_{n}\colon M^{n}_{n} \hookrightarrow \Delta^{2n+1}$ as an inclusion
  \begin{equation*}
    \S^{\mathcal{A}}_{[2n+1]} \subseteq \Delta^{2n+1},\qquad \mathcal{A} = \left\{ \substack{ \{\overline{n}, \ldots, \overline{1}\} \\ \{0, \ldots n\} \\ \{1, \overline{1}\} \\ \vdots \\ \{n, \overline{n}\} } \right\}.
  \end{equation*}

  We will express this inclusion as the following composition of fillings:
  \begin{enumerate}
    \item We first add the simplices $\Delta^{\{n, \overline{n}, \ldots, \overline{0}\}}$, $\Delta^{\{n-1, n, \overline{n}, \ldots, \overline{0}\}}$, \dots, $\Delta^{\{2, \ldots, n, \overline{n}, \ldots, \overline{0}\}}$.

    \item We next add $\Delta^{\{0, \ldots, n, \overline{n}, \overline{0}\}}$, $\Delta^{\{0, \ldots, n, \overline{n}, \overline{n-1}, \overline{0}\}}$, \dots, $\Delta^{\{0, \ldots, n, \overline{n}, \ldots, \overline{2}, \overline{0}\}}$

    \item We next add $\Delta^{\{1, \ldots n, \overline{n}, \ldots, \overline{0}\}}$.

    \item We finally add $\Delta^{\{0, \ldots n, \overline{n}, \ldots, \overline{0}\}}$.
  \end{enumerate}
  We proceed.
  \begin{enumerate}
    \item
      \begin{itemize}
        \item Using \hyperref[lemma:bicartesian_square]{Lemma~\ref*{lemma:bicartesian_square}}, we see that the square
          \begin{equation*}
            \begin{tikzcd}
              \S^{\mathcal{A}|\{n, \overline{n}, \ldots, \overline{0}\}}_{\{n, \overline{n}, \ldots, \overline{0}\}}
              \arrow[r, hook]
              \arrow[d, hook]
              & \Delta^{\{n, \overline{n}, \ldots, \overline{0}\}}
              \arrow[d, hook]
              \\
              \S^{\mathcal{A}}_{[2n+1]}
              \arrow[r, hook]
              & \S^{\mathcal{A}}_{[2n+1]} \cup \Delta^{\{n, \overline{n}, \ldots, \overline{0}\}}
            \end{tikzcd}
          \end{equation*}
          is pushout. In order to show that the bottom inclusion is marked-scaled anodyne, it thus suffices to show that the top inclusion is marked-scaled anodyne. We have
          \begin{equation*}
            \begin{tikzcd}
              \mathcal{A}|\{n, \overline{n}, \ldots, \overline{0}\}
              = \left\{ \substack{ \{\overline{n}, \ldots, \overline{1}\} \\ \{n\} \\ \{\overline{1}\} \\ \vdots \\ \{\overline{n-1}\} \\ \{n, \overline{n}\} } \right\}
              \sim \left\{ \substack{ \{n\} \\ \{\overline{n-1}\} \\ \vdots \\ \{\overline{1}\} } \right\} =: \mathcal{A}'
            \end{tikzcd}.
          \end{equation*}
          This is a dull subset of $\{n, \overline{n}, \ldots, \overline{0}\}$ with pivot $\overline{n}$. The only $\mathcal{A}'$-basal set is $\{n, \overline{n-1}, \ldots, \overline{1}\}$. One checks that $\mathcal{A}'$, together with the marking $\heartsuit$ and scaling $\dagger$, satisfies the conditions of the \hyperref[lemma:pivot_trick]{Pivot~Trick}:
          \begin{itemize}
            \item For $n = 2$, the only the marked edge not containing the pivot $\overline{2}$ is $\overline{1} \to \overline{0}$, which belongs to $\S^{\mathcal{A}'}_{\{n, \overline{n}, \ldots, \overline{0}\}}$. The only scaled simplex which does not contain $\overline{2}$, and which does not factor through $\S^{\mathcal{A}'}$, is $\sigma = \{2 < \overline{1} < \overline{0}\}$. The simplex $\sigma \cup \{\overline{2}\}$ is fully scaled.

            \item For $n = 3$, each marked edge is contained in $\S^{\mathcal{A}'}$. The only scaled triangle which does not contain the pivot $\overline{3}$, and which does not factor through $\S^{\mathcal{A}'}$, is $\sigma = \{3 < \overline{2} < \overline{1}\}$. The simplex $\sigma \cup \{\overline{3}\}$ is fully scaled.

            \item For $n \geq 4$, all scaled and marked simplices belong to $\S^{\mathcal{A}'}$ by \hyperref[lemma:subset_of_simplex_contains_k_simplices]{Lemma~\ref*{lemma:subset_of_simplex_contains_k_simplices}}.
          \end{itemize}
          In each case, the simplex $\{n, \overline{n}, \overline{n-1}\}$ is scaled, so the top inclusion is marked-scaled anodyne by the \hyperref[lemma:pivot_trick]{Pivot Trick}. We can write
          \begin{equation*}
            \S^{\mathcal{A}}_{[2n+1]} \cup \Delta^{\{n, \overline{n}, \ldots, \overline{0}\}} = \S^{\mathcal{A} \cup ([2n+1] \smallsetminus \{n, \overline{n}, \ldots, \overline{0}\})}_{[2n+1]},
          \end{equation*}
          We see that
          \begin{equation*}
            \mathcal{A} \cup ([2n+1] \smallsetminus \{n, \overline{n}, \ldots, \overline{0}\}) \sim \left\{ \substack{ \{\overline{n}, \ldots, \overline{1}\} \\ \{0, \ldots n-1\} \\ \{1, \overline{1}\} \\ \vdots \\ \{n, \overline{n}\} } \right\},
          \end{equation*}
          which we denote by $\mathcal{A}_{n-1}$.

        \item We proceed inductively. Suppose we have added the simplices $\Delta^{\{n, \overline{n}, \ldots, \overline{0}\}}$, $\Delta^{\{n-1, n, \overline{n}, \ldots, \overline{0}\}}$, \dots, $\Delta^{\{k+1, \ldots, n, \overline{n}, \ldots, \overline{0}\}}$, for $2 \leq k \leq n-1$. Using \hyperref[lemma:add_a_simplex]{Lemma~\ref*{lemma:add_a_simplex}} and \hyperref[lemma:replace_poset_by_minimal_elements]{Lemma~\ref*{lemma:replace_poset_by_minimal_elements}} can write the result of these additions as
          \begin{equation*}
            \S^{\mathcal{A}}_{[2n+1]} \cup \left( \bigcup_{i = k+1}^{n} \Delta^{\{i, \ldots, n, \overline{n}, \ldots, \overline{0}\}}\right) = \S^{\mathcal{A}_{k+1}}_{[2n+1]},
          \end{equation*}
          where
          \begin{equation*}
            \mathcal{A}_{k+1} = \left\{ \substack{ \{\overline{n}, \ldots, \overline{1}\} \\ \{0, \ldots, k\} \\ \{1, \overline{1}\} \\ \vdots \\ \{n, \overline{n}\} } \right\}.
          \end{equation*}

          Using \hyperref[lemma:bicartesian_square]{Lemma~\ref*{lemma:bicartesian_square}}, we see that the square
          \begin{equation*}
            \begin{tikzcd}
              \S^{\mathcal{A}_{k+1}|\{k, \ldots, n, \overline{n}, \ldots, \overline{0}\}}_{\{k, \ldots, n, \overline{n}, \ldots, \overline{0}\}}
              \arrow[r, hook]
              \arrow[d, hook]
              & \Delta^{\{k, \ldots, n, \overline{n}, \ldots, \overline{0}\}}
              \arrow[d, hook]
              \\
              \S^{\mathcal{A}_{k+1}}_{[2n+1]}
              \arrow[r, hook]
              & \S^{\mathcal{A}_{k+1}}_{[2n+1]} \cup \Delta^{\{k, \ldots, n, \overline{n}, \ldots, \overline{0}\}}
            \end{tikzcd}
          \end{equation*}
          is pushout, so in order to show that the bottom morphism is marked-scaled anodyne, it suffices to show that the top morphism is. We see that
          \begin{equation*}
            \mathcal{A}_{k+1}|\{k, \ldots, n, \overline{n}, \ldots, \overline{0}\} =
            \left\{ \substack{ \{\overline{n}, \ldots, \overline{1}\} \\ \{k\} \\ \{\overline{1}\} \\ \vdots \\ \{\overline{k-1}\} \\ \{k, \overline{k}\} \\ \vdots \\ \{n, \overline{n}\} } \right\}
            \sim \left\{ \substack{ \{k\} \\ \{\overline{1}\} \\ \vdots \\ \{\overline{k-1}\} \\ \{k+1, \overline{k+1}\} \\ \vdots \\ \{n, \overline{n}\} } \right\} =: \mathcal{A}_{k+1}'.
          \end{equation*}

          This is a dull subset of $P(\{k, \ldots, n, \overline{n}, \ldots, \overline{0}\})$ with pivot $\overline{k}$, and one checks that the conditions of the \hyperref[lemma:pivot_trick]{Pivot Trick} are satisfied:
          \begin{itemize}
            \item For $n = 3$, where the only value of $k$ is $k = 2$, each marked edge is contained in $\S^{\mathcal{A}_{k+1}'}$ by \hyperref[lemma:subset_of_simplex_contains_k_simplices]{Lemma~\ref*{lemma:subset_of_simplex_contains_k_simplices}}, and one can check that each scaled triangle factors through $\S^{\mathcal{A}_{k+1}'}$.

            \item For $n \geq 4$, all scaled and marked simplices belong to $\S^{\mathcal{A}'}$ by \hyperref[lemma:subset_of_simplex_contains_k_simplices]{Lemma~\ref*{lemma:subset_of_simplex_contains_k_simplices}}.
          \end{itemize}
          The basal sets are of the form
          \begin{equation*}
            \{k, a_{1}, \ldots, a_{n-k}, \overline{k-1}, \ldots, \overline{1}\},
          \end{equation*}
          where each $a_{1}$, \dots, $a_{n-k}$ is of the form $\ell$ or $\overline{\ell}$ for $k+1 \leq \ell \leq n$. In each case, the simplex $\{a_{n-k}, \overline{k}, \overline{k-1}\}$ is scaled. Thus, the conditions of the \hyperref[lemma:pivot_trick]{Pivot Trick}, the top morphism is marked-scaled anodyne.
      \end{itemize}

      We have now added the simplices promised in part 1., and are left with the simplicial subset $\S^{\mathcal{A}_{2}}_{[2n+1]}$, where
      \begin{equation*}
        \mathcal{A}_{2}
        = \left\{ \substack{ \{\overline{n}, \ldots, \overline{1}\} \\ \{0, 1\} \\ \{1, \overline{1}\} \\ \vdots \\ \{n, \overline{n}\} } \right\}.
      \end{equation*}

    \item Each step in this sequence is solved exactly like those above. The calculations are omitted. The end result is the simplicial subset
      \begin{equation*}
        \mathcal{B}_{2} = 
        \left\{ \substack{ \{\overline{1}\} \\ \{0, 1\} \\ \{2, \overline{2}\} \\ \vdots \\ \{n, \overline{n}\} } \right\}.
      \end{equation*}

    \item Using \hyperref[lemma:bicartesian_square]{Lemma~\ref*{lemma:bicartesian_square}}, we see that the square
      \begin{equation*}
        \begin{tikzcd}
          \S^{\mathcal{B}_{2}|\{1, \ldots, n, \overline{n}, \ldots, \overline{0}\}}_{\{1, \ldots, n, \overline{n}, \ldots, \overline{0}\}}
          \arrow[r, hook]
          \arrow[d, hook]
          & \Delta^{\{1, \ldots, n, \overline{n}, \ldots, \overline{0}\}}
          \arrow[d, hook]
          \\
          \S^{\mathcal{B}_{2}}_{[2n+1]}
          \arrow[r, hook]
          & \S^{\mathcal{B}_{2}}_{[2n+1]} \cup \Delta^{\{1, \ldots, n, \overline{n}, \ldots, \overline{0}\}}
        \end{tikzcd}
      \end{equation*}
      is pushout. We thus have to show that the top morphism is marked-scaled anodyne. We have
      \begin{equation*}
        \mathcal{B}_{2}|\{1, \ldots, n, \overline{n}, \ldots, \overline{0}\} \sim
        \left\{ \substack{ \{\overline{1}\} \\ \{2, \overline{2}\} \\ \vdots \\ \{n, \overline{n}\} \\ \{\overline{1}\} } \right\}.
      \end{equation*}
      One readily sees that this is a right-dull subset, and that the conditions of the \hyperref[lemma:right-anodyne_pivot_trick]{Right-Anodyne Pivot Trick} are satisfied.

    \item One solves this as before, checking that the conditions of the \hyperref[lemma:pivot_trick]{Pivot Trick} are satisfied, with pivot point $2$.
  \end{enumerate}
\end{proof}

\subsection{The map governing local systems is a cocartesian fibration}
\label{ssc:the_fibration}

Our aim is to show that the map of quasicategories $p\colon \LS(\category{C}) \to \S$ defined by the diagram
\begin{equation*}
  \begin{tikzcd}
    \LS(\category{C})
    \arrow[r, hook]
    \arrow[d, swap, "p"]
    & \mathcal{R}
    \arrow[r, hook]
    \arrow[d, "p''"]
    & \Tw(\ICCat)
    \arrow[d, "p'"]
    \\
    \S \times \{\category{C}\}
    \arrow[r, hook]
    & \ICat \times [\category{C}]
    \arrow[r, hook]
    & \ICat \times \ICat\op
  \end{tikzcd},
\end{equation*}
in which both squares are pullback, is a cocartesian fibration. We now define the class of morphisms in $\LS(\category{C})$ which we claim are $p$-cocartesian.

\begin{definition}
  \label{def:left_kan_simplex}
  A morphism $\tilde{\sigma}\colon \Delta^{1} \to \LS(\category{C})$ is said to be \defn{left Kan} if the simplex $\sigma\colon \Delta^{3}_{\dagger} \to \ICCat$ to which it is adjoint has the property that the restriction $\sigma|\Delta^{\{0,1,\overline{0}\}}$ is left Kan in the sense of \hyperref[def:left_kan]{Definition~\ref*{def:left_kan}}.
\end{definition}

We draw the `front' and `back' of a general $3$-simplex $\sigma\colon \Delta^{3}_{\dagger} \to \ICCat$ corresponding to some morphism $\Delta^{1} \to \LS(\category{C})$.
\begin{equation*}
  \begin{tikzcd}[column sep=large, row sep=large]
    X
    \arrow[r, "f"]
    \arrow[d, "\mathcal{F}"{swap}, ""{name=L}]
    \arrow[dr, "\mathcal{H}"]
    & Y
    \arrow[d, "\mathcal{G}"]
    \\
    \category{C}
    & \category{C}
    \arrow[l, "\id_{\category{C}}"]
    \arrow[from=L, Rightarrow, shorten=3ex, "\zeta"{swap}]
  \end{tikzcd}
  \qquad\text{and}\qquad
  \begin{tikzcd}[column sep=large, row sep=large]
    X
    \arrow[r, "f"]
    \arrow[d, "\mathcal{F}"{swap}, ""{name=L}]
    & Y
    \arrow[d, "\mathcal{G}"]
    \arrow[dl, "\mathcal{G}'"]
    \arrow[from=L, Rightarrow, shorten=3ex, "\eta"]
    \\
    \category{C}
    & \category{C}
    \arrow[l, "\id_{\category{C}}"]
  \end{tikzcd}
\end{equation*}
Here, the $2$-simplices which are notated without natural transformations are constrained by the scaling $\dagger$ to be thin. The morphism $\tilde{\sigma}$ is left Kan if and only if $\mathcal{G}'$ is a left Kan extension of $\mathcal{F}$ along $f$, and $\eta$ is a unit map.

\begin{notation}
  We endow the quasicategory $\LS(\category{C})$ with the marking $\clubsuit \subseteq \LS(\category{C})_{1}$ consisting of all left Kan edges.
\end{notation}

\begin{theorem}
  \label{thm:cocartesian_edges_of_p}
  The map $p\colon \LS(\category{C}) \to \S$ is a cocartesian fibration, and an edge $\Delta^{1} \to \LS(\category{C})$ is $p$-cocartesian if and only if it is left Kan.
\end{theorem}
\begin{proof}
  By \cite[Prop.~3.1.1.6]{highertopostheory}, it suffices to show that the map $\LS(\category{C})^{\clubsuit} \to \S^{\sharp}$ has the right lifting property with respect to each of the classes of generating marked anodyne morphisms. We check these one-by-one.
  \begin{enumerate}
    \item We need to check that all lifting problems
      \begin{equation*}
        \begin{tikzcd}
          (\Lambda^{n}_{i})^{\flat}
          \arrow[r]
          \arrow[d]
          & \LS(\category{C})^{\clubsuit}
          \arrow[d]
          \\
          (\Delta^{n})^{\flat}
          \arrow[r]
          \arrow[ur, dashed]
          & \S^{\sharp}
        \end{tikzcd},\qquad n \geq 2, \quad 0 < i < n
      \end{equation*}
      admit solutions. This follows from \hyperref[thm:mainthm_walker_fernando]{Theorem~\ref*{thm:mainthm_walker_fernando}}.

    \item We need to show that each of the lifting problems
      \begin{equation*}
        \begin{tikzcd}
          (\Lambda^{n}_{0})^{\mathcal{L}}
          \arrow[r]
          \arrow[d]
          & \LS(\category{C})^{\clubsuit}
          \arrow[d]
          \\
          (\Delta^{n})^{\mathcal{L}}
          \arrow[r]
          \arrow[ur, dashed]
          & \S^{\sharp}
        \end{tikzcd},\qquad n \geq 1.
      \end{equation*}
      has a solution. Given any such lifting problem, it suffices to find a solution to the outer lifting problem
      \begin{equation*}
        \begin{tikzcd}
          (\Lambda^{n}_{0})^{\mathcal{L}}
          \arrow[r]
          \arrow[d]
          & \LS(\category{C})
          \arrow[d]
          \arrow[r]
          & \mathcal{R}
          \arrow[d]
          \\
          (\Delta^{n})^{\mathcal{L}}
          \arrow[r]
          \arrow[urr, dashed]
          & \S
          \arrow[r, hook]
          & \ICat \times [\category{C}]
        \end{tikzcd}.
      \end{equation*}
      Passing to the adjoint lifting problem, we need to find a solution to the lifting problem
      \begin{equation}
        \label{eq:unsimplified_lifting_problem}
        \begin{tikzcd}
          R(\Lambda^{n}_{0}) \displaystyle\coprod_{(\Lambda^{n}_{0})^{\flat}_{\sharp} \amalg ((\Lambda^{n}_{0})\op)^{\sharp}_{\sharp}}(\Delta^{n})^{\flat}_{\sharp} \amalg ((\Delta^{n})\op)^{\sharp}_{\sharp}
          \arrow[r, "\sigma"]
          \arrow[d, hook]
          & \ICCat
          \\
          R(\Delta^{n})
          \arrow[ur, dashed]
        \end{tikzcd}
      \end{equation}
      such that $\sigma|\{0,1,\overline{0}\}$ is left Kan. Here $R$ is the functor of \hyperref[def:cosimplicial_obj_R]{Definition~\ref*{def:cosimplicial_obj_R}}.

      We note the existence of a commutative square
      \begin{equation*}
        \begin{tikzcd}
          J(\Lambda^{n}_{0}) \displaystyle\coprod_{(\Lambda^{n}_{0})^{\flat}_{\sharp} \amalg (\Delta^{\{\overline{0}\}})^{\sharp}_{\sharp}}(\Delta^{n})^{\flat}_{\sharp} \amalg (\Delta^{\{\overline{0}\}})^{\sharp}_{\sharp}
          \arrow[r, hook, "b"]
          \arrow[d, hook]
          & R(\Lambda^{n}_{0}) \displaystyle\coprod_{(\Lambda^{n}_{0})^{\flat}_{\sharp} \amalg ((\Lambda^{n}_{0})\op)^{\sharp}_{\sharp}}(\Delta^{n})^{\flat}_{\sharp} \amalg ((\Delta^{n})\op)^{\sharp}_{\sharp}
          \arrow[d, hook]
          \\
          J(\Delta^{n})
          \arrow[r, hook, "v_{n}"]
          & R(\Delta^{n})
        \end{tikzcd},
      \end{equation*}
      where $J$ is the functor defined in \hyperref[def:cosimplicial_obj_J]{Definition~\ref*{def:cosimplicial_obj_J}}. Here, $v_{n}$ is the morphism of \hyperref[lemma:lower_morphism_equivalence]{Lemma~\ref*{lemma:lower_morphism_equivalence}}. The morphism $b$ is defined component-wise via the following morphisms.
      \begin{itemize}
        \item The morphism $J(\Lambda^{n}_{0}) \to R(\Lambda^{n}_{0})$ comes from \hyperref[lemma:marked_scaled_equivalence_left_horn]{Lemma~\ref*{lemma:marked_scaled_equivalence_left_horn}}, where it is proven to be a weak equivalence in the marked-scaled model structure.

        \item The map $(\Delta^{\{\overline{0}\}})^{\sharp}_{\sharp} \to ((\Lambda^{n}_{0})\op)^{\sharp}_{\sharp}$ is marked-scaled anodyne by \hyperref[prop:sharp_marked_right_anodyne]{Proposition~\ref*{prop:sharp_marked_right_anodyne}}.

        \item The maps $(\Delta^{\{\overline{0}\}})^{\sharp}_{\sharp} \to ((\Delta^{n})\op)^{\sharp}_{\sharp}$ is marked-scaled anodyne by \hyperref[prop:sharp_marked_right_anodyne]{Proposition~\ref*{prop:sharp_marked_right_anodyne}}.

        \item The rest of the morphisms connecting the components are isomorphisms.
      \end{itemize}

      We showed in \hyperref[lemma:lower_morphism_equivalence]{Lemma~\ref*{lemma:lower_morphism_equivalence}} that the lower morphism in this square is a marked-scaled equivalence. We would now like to show that the upper morphism is a marked-scaled equivalence. Recall that, since the cofibrations in the model structure on marked-scaled simplicial sets are simply those morphisms whose underlying morphism of simplicial sets is a monomorphism, the colimits defining $b$ are models for the homotopy colimits, so the result follows from the observation that each component is a weak equivalence in the marked-scaled model structure.

      Using \cite[Prop.~A.2.3.1]{highertopostheory}, we see that in order to solve the lifting problem of \hyperref[eq:unsimplified_lifting_problem]{Equation~\ref*{eq:unsimplified_lifting_problem}}, it suffices to show that the lifting problem
      \begin{equation*}
        \begin{tikzcd}
          J(\Lambda^{n}_{0}) \displaystyle\coprod_{(\Lambda^{n}_{0})^{\flat}_{\sharp} \amalg (\Delta^{\{\overline{0}\}})^{\sharp}_{\sharp}}(\Delta^{n})^{\flat}_{\sharp} \amalg (\Delta^{\{\overline{0}\}})^{\sharp}_{\sharp}
          \arrow[r, "b \circ \sigma"]
          \arrow[d, hook]
          & \ICCat
          \\
          J(\Delta^{n})
          \arrow[ur, dashed, swap, "\ell"]
        \end{tikzcd}
      \end{equation*}
      has a solution whose restriction $\ell|\{0,1,\overline{0}\}$ is left Kan. However, we note that there exists a pushout square
      \begin{equation*}
        \begin{tikzcd}
          (\Lambda^{\{0, \ldots, n, \overline{0}\}}_{0})^{\flat}_{\flat}
          \arrow[r]
          \arrow[d]
          & J(\Lambda^{n}_{0}) \displaystyle\coprod_{(\Lambda^{n}_{0})^{\flat}_{\sharp} \amalg (\Delta^{\{\overline{0}\}})^{\sharp}_{\sharp}}(\Delta^{n})^{\flat}_{\sharp} \amalg (\Delta^{\{\overline{0}\}})^{\sharp}_{\sharp}
          \arrow[d]
          \\
          (\Delta^{\{0, \ldots, n, \overline{0}\}})^{\flat}_{\flat}
          \arrow[r]
          & J(\Delta^{n})
        \end{tikzcd}
      \end{equation*}
      in which the rightward-facing morphisms are isomorphisms on underlying simplicial sets, and where the only new decorations being added are $(\Delta^{\{0, \ldots, n\}})^{\flat}_{\flat} \hookrightarrow (\Delta^{\{0, \ldots, n\}})^{\flat}_{\sharp}$. Therefore, it suffices to solve the lifting problems
      \begin{equation*}
        \begin{tikzcd}
          (\Lambda^{\{0, \ldots, n, \overline{0}\}}_{0})^{\flat}_{\flat}
          \arrow[r]
          \arrow[d]
          & \ICCat
          \\
          (\Delta^{\{0, \ldots, n, \overline{0}\}})^{\flat}_{\flat}
          \arrow[ur, dashed, swap, "\ell'"]
        \end{tikzcd},\qquad n \geq 1
      \end{equation*}
      such that $\ell'|\{0,1,\overline{0}\}$ is left Kan. That these lifting problems admit solutions for $n \geq 2$ is the content of \hyperref[thm:left_kan_implies_globally_left_kan]{Theorem~\ref*{thm:left_kan_implies_globally_left_kan}}. The case $n=1$ is the statement that left Kan extensions of functors into cocomplete categories exist along functors between small categories.

    \item We need to show that every lifting problem of the form
      \begin{equation*}
        \begin{tikzcd}
          (\Lambda^{2}_{1})^{\sharp} \amalg_{( \Lambda^{2}_{1} )^{\flat}}(\Delta^{2})^{\sharp}
          \arrow[r]
          \arrow[d]
          & \LS(\category{C})^{\clubsuit}
          \arrow[d]
          \\
          (\Delta^{2})^{\sharp}
          \arrow[r]
          \arrow[ur, dashed]
          & \S^{\sharp}
        \end{tikzcd}
      \end{equation*}
      has a solution. Considering the adjoint lifting problem, we find that it suffices to show that for any $\sigma\colon  \Delta^{5}_{\dagger} \to \ICCat$ such that the restrictions $\sigma|\{0,1,\overline{0}\}$ and $\sigma|\{1,2,\overline{1}\}$ are left Kan and the morphisms belonging to $\sigma|\{\overline{2}, \overline{1}, \overline{0}\}$ are equivalences, the restriction $\sigma|\{0,2,\overline{0}\}$ is left Kan. Applying \hyperref[lemma:transport_left_kan_simplices]{Lemma~\ref*{lemma:transport_left_kan_simplices}} to $\sigma|\{1,2,\overline{1},\overline{0}\}$, we see that $\sigma|\{1,2,\overline{1}\}$ is left Kan if and only if $\sigma|\{1,2,\overline{0}\}$ is left Kan. Applying \hyperref[lemma:compose_left_kan_simplices]{Lemma~\ref*{lemma:compose_left_kan_simplices}} to $\sigma|\{0,1,2,\overline{0}\}$ guarantees that $\sigma|\{0,2,\overline{0}\}$ is left Kan as required.

    \item We need to show that for all Kan complexes $K$, the lifting problem
      \begin{equation*}
        \begin{tikzcd}
          K^{\flat}
          \arrow[r]
          \arrow[d]
          & \LS(\category{C})^{\clubsuit}
          \arrow[d]
          \\
          K^{\sharp}
          \arrow[r]
          \arrow[ur, dashed]
          & \S^{\sharp}
        \end{tikzcd}
      \end{equation*}
      has a solution. To see this, note that each morphism in $K$ must be mapped to an equivalence in $\LS(\category{C})$, and a morphism $a\colon \mathcal{F} \to \mathcal{G}$ in $\LS(\category{C})$ is an equivalence and only if it is $p$-cartesian, and lies over an equivalence in $\S$.
      \begin{itemize}
        \item By \hyperref[thm:mainthm_walker_fernando]{Theorem~\ref*{thm:mainthm_walker_fernando}}, the morphism $a$ is $p$-cartesian if and only if the map $\sigma\colon \Delta^{3}_{\dagger} \to \ICCat$ to which it is adjoint factors through the map $\Delta^{3}_{\dagger} \to \Delta^{3}_{\sharp}$. In particular, the restriction $\sigma|\Delta^{\{0,1,\overline{0}\}}$ is thin.

        \item The image of $a$ in $\S$ is the restriction $\sigma|\Delta^{\{0,1\}}$, which is therefore an equivalence.
      \end{itemize}
      It follows from \hyperref[eg:strictly_commuting_left_kan]{Example~\ref*{eg:strictly_commuting_left_kan}} that the morphism $a$ is automatically left Kan. Thus, each morphism in $K$ is mapped to a left Kan morphism in $\LS(\category{C})$, and we are justified in marking them; our lifting problems admit solutions.
  \end{enumerate}
\end{proof}

\section{Pull-push of local systems}
\label{sec:the_non_monoidal_construction}

For any cocomplete $\infty$-category $\category{C}$, we have constructed an $\infty$-category $\LS(\category{C})$ of local systems on $\category{C}$, together with a map $p\colon \LS(\category{C}) \to \S$. We have further shown that this map is a bicartesian fibration.
\begin{itemize}
  \item As a cartesian fibration, it classifies the pullback functor $(-)^{*}\colon \S\op \to \ICat$; for any morphism $X \to Y$ in $\S$, this functoriality gives us a map
    \begin{equation*}
      \begin{tikzcd}
        f^{*}\colon \Fun(Y, \category{C}) \to \Fun(X, \category{C}).
      \end{tikzcd}
    \end{equation*}

  \item As a cocartesian fibration, it classifies the left Kan extension functor $(-)_{!}\colon \S \to \ICat$; for any morphism $X \to Y$, in $\S$, this functoriality gives us a map
    \begin{equation*}
      f_{!}\colon \Fun(X, \category{C}) \to \Fun(Y, \category{C}).
    \end{equation*}
\end{itemize}

We might wonder if we can combine these functorialities. Given a \emph{span} of morphisms in $\S$, i.e.\ a diagram in $\S$ of the form
\begin{equation}
  \label{eq:span_of_spaces}
  \begin{tikzcd}
    & Y
    \arrow[dl, swap, "g"]
    \arrow[dr, "f"]
    \\
    X
    && X'
  \end{tikzcd},
\end{equation}
we can pull back along $g$ and push forward along $f$, giving a map $f_{!} \circ g^{*}\colon \Fun(X, \category{C}) \to \Fun(X', \category{C})$ via the composition
\begin{equation*}
  \begin{tikzcd}[column sep=tiny]
    & \Fun(Y, \category{C})
    \arrow[dr, "f_{!}"]
    \\
    \Fun(X, \category{C})
    \arrow[rr, dashed, swap, "f_{!} \circ g^{*}"]
    \arrow[ur, "g^{*}"]
    && \Fun(X', \category{C})
  \end{tikzcd}.
\end{equation*}
Our goal in this section is to upgrade this construction to a functor $\hat{r}\colon \Span(\S) \to \ICat$, where $\Span(\S)$ is an $\infty$-category with the following rough description.
\begin{itemize}
  \item The objects of $\Span(\S)$ are the same as the objects of $\S$, i.e.\ spaces $X$, $Y$, etc.

  \item The morphisms from $X$ to $X'$ are given by spans in $\S$, i.e.\ diagrams of the form given in \hyperref[eq:span_of_spaces]{Equation~\ref*{eq:span_of_spaces}}.

  \item The $2$-simplices witnessing the composition of two morphisms
    \begin{equation*}
      X \overset{h}{\leftarrow} Y \overset{f}{\rightarrow} X' \quad\text{and}\quad X' \overset{g}{\leftarrow} Y \overset{j}{\rightarrow} X''
    \end{equation*}
    are diagrams of the form
    \begin{equation}
      \label{eq:diagram_witnessing_composition}
      \begin{tikzcd}
        && Z
        \arrow[dl, swap, "g'"]
        \arrow[dr, "f'"]
        \\
        & Y
        \arrow[dl, swap, "h"]
        \arrow[dr, "f"]
        && Y'
        \arrow[dl, swap, "g"]
        \arrow[dr, "j"]
        \\
        X
        && X'
        && X''
      \end{tikzcd},
    \end{equation}
    where the square formed is pullback. The corresponding composition is then given by the span
    \begin{equation*}
      \begin{tikzcd}
        & Z
        \arrow[dl, swap, "h \circ g'"]
        \arrow[dr, "j \circ f'"]
        \\
        X
        && X''
      \end{tikzcd}.
    \end{equation*}
\end{itemize}
The non-trivial part of constructing such a functor $\hat{r}$ will be showing that it respects composition in a homotopy-coherent way; we need that for any diagram of the form given in \hyperref[eq:diagram_witnessing_composition]{Equation~\ref*{eq:diagram_witnessing_composition}}, both ways of composing morphisms from left to right in the diagram
\begin{equation*}
  \begin{tikzcd}[column sep=tiny]
    && \Fun(Z, \category{C})
    \arrow[dr, "(f')_{!}"]
    \\
    & \Fun(Y, \category{C})
    \arrow[ur, "(g')^{*}"]
    \arrow[dr, "f_{!}"]
    && \Fun(Y', \category{C})
    \arrow[dr, "j_{!}"]
    \\
    \Fun(X, \category{C})
    \arrow[ur, "h^{*}"]
    && \Fun(X', \category{C})
    \arrow[ur, "g^{*}"]
    && \Fun(X'', \category{C})
  \end{tikzcd},
\end{equation*}
agree up to a specified natural equivalence. The arrows $h^{*}$ and $j_{!}$ are the same in both cases, and can be ignored. This condition can thus be distilled down to the so-called \emph{Beck-Chevalley condition:}
\begin{itemize}
  \item For any pullback square in $\S$
    \begin{equation*}
      \begin{tikzcd}
        Z
        \arrow[r, "f'"]
        \arrow[d, swap, "g'"]
        & Y'
        \arrow[d, "g"]
        \\
        Y
        \arrow[r, "f"]
        & X
      \end{tikzcd},
    \end{equation*}
    the comparison map
    \begin{equation*}
      f_{!} \circ g^{*} \overset{\eta}{\Rightarrow} f_{!} \circ g^{*} \circ (f')^{*} \circ (f')_{!} \overset{\simeq}{\Rightarrow} f_{!} \circ f^{*} \circ (g')^{*} \circ (f')_{!} \overset{\epsilon}{\Rightarrow} (g')^{*} \circ (f')_{!}
    \end{equation*}
    is an equivalence.
\end{itemize}

One can phrase the Beck-Chevalley condition at the level of fibrations rather than functors into $\ICat$. This is a classical definition, here more or less lifted from \cite{luriehopkins2013ambidexterity}.

\begin{definition}
  \label{def:beck_chevalley_fibration}
  A bicartesian fibration of quasicategories $p\colon \category{X} \to \category{T}$ such that $\category{T}$ admits pullbacks is called a \defn{Beck-Chevalley fibration} if it has the following property.
  \begin{itemize}
    \item[(BC)]\label{item:beck_chevalley_condition} For any commuting square in $\category{X}$
      \begin{equation*}
        \begin{tikzcd}
          z
          \arrow[r, "f'"]
          \arrow[d, swap, "g'"]
          & y
          \arrow[d, "g"]
          \\
          y'
          \arrow[r, "f"]
          & x
        \end{tikzcd}
      \end{equation*}
      lying over a pullback square in $\category{T}$, if the morphism $f$ is $p$-cocartesian and the morphisms $g$ and $g'$ are $p$-cartesian, then the morphism $f'$ is $p$-cocartesian.
  \end{itemize}
\end{definition}

In fact, it turns out that this condition is sufficient to guarantee that the the pull-push procedure is functorial.

\begin{proposition}
  \label{prop:beck_chevalley_fibration_straightens_to_pp_functoriality}
  Let $p\colon \category{X} \to \category{T}$ be a Beck-Chevalley fibration. Then there is a functor $\Span(\category{T}) \to \ICat$ sending an object $t \in \Span(\category{T})$ to the fiber $\category{X}_{t}$, and a span $t \overset{b}{\leftarrow} s \overset{a}{\rightarrow} t'$ to the composition $a_{!} \circ b^{*}\colon \category{X}_{t} \to \category{X}_{t'}$.
\end{proposition}
\begin{proof}
  It follows immediately from \hyperref[thm:new_barwick]{Theorem~\ref*{thm:new_barwick}} that there is a cocartesian fibration
  \begin{equation*}
    \tilde{p}\colon \Span^{\mathrm{cart}}(\category{X}) \to \Span(\category{T}),
  \end{equation*}
  where $\Span^{\mathrm{cart}}$ denotes the category of spans whose backwards-facing legs are constrained to be $p$-cartesian; that is, $\category{X}^{\dagger} = \category{X}^{\mathrm{cart}}$. Straightening gives the result that we want.
\end{proof}

We would like to show that pull-push of local systems is functorial. According to \hyperref[prop:beck_chevalley_fibration_straightens_to_pp_functoriality]{Proposition~\ref*{prop:beck_chevalley_fibration_straightens_to_pp_functoriality}}, it suffices to show the following, the proof of which will come at the end of this section.

\begin{proposition}
  \label{prop:local_systems_are_beck_chevalley}
  The functor $p\colon \LS(\category{C}) \to \S$ is a Beck-Chevalley fibration.
\end{proposition}

We first prove a helpful lemma. This is simply a reformulation of the definition of a left Kan extension (\hyperref[def:nat_xfo_exhibiting_left_kan_ext]{Definition~\ref*{def:nat_xfo_exhibiting_left_kan_ext}}) in the special case that we consider left Kan extensions along maps of Kan complexes.
\begin{lemma}
  \label{lemma:left_kan_ext_along_kan_complexes}
  Let $f\colon X \to Y$ be a map between Kan complexes, let $F\colon X \to \category{C}$ and $G\colon Y \to \category{C}$ be functors, and let $\eta\colon F \to G \circ y$ be a natural transformation. Then $\eta$ exhibits $G$ as a left Kan extension of $F$ along $f$ if and only if, for all $y \in Y$, the natural transformatin $F \circ \pi \Rightarrow \underline{G(y)}$ defined by the pasting diagram
  \begin{equation*}
    \begin{tikzcd}[row sep=large, column sep=large]
      X_{/y}
      \arrow[r, "\pi", ""{name=LA, swap}]
      \arrow[d, ""{name=LD}]
      & X
      \arrow[d, swap, "f"]
      \arrow[r, "F", ""{below, name=M}]
      & \category{C}
      \\
      \{y\}
      \arrow[r, hook, ""{name=LB}]
      & Y
      \arrow[ur, swap, "G"]
      \arrow[from=M, Rightarrow, shorten=2ex, swap, "\eta"]
    \end{tikzcd},
  \end{equation*}
  exhibits $G(y)$ as the colimit of $F \circ \pi$, where the left-hand square is homotopy pullback.
\end{lemma}
\begin{proof}
  We note that we can factor the above square into three squares
  \begin{equation*}
    \begin{tikzcd}
      X_{/y}
      \arrow[r, "\simeq"]
      \arrow[d]
      & X^{/y}
      \arrow[r]
      \arrow[d]
      & Y^{\Delta^{1}} \times_{Y} X
      \arrow[r, "\simeq"]
      \arrow[d]
      & X
      \arrow[d, "f"]
      \\
      \{y\}
      \arrow[r, equals]
      & \{y\}
      \arrow[r]
      & Y
      \arrow[r, equals]
      & Y
    \end{tikzcd},
  \end{equation*}
  where the middle square is a strict pullback. Since the map $Y^{\Delta^{1}} \times_{Y}X \to Y$ is a Kan fibration, the middle square is also a homotopy pullback. We note that the left and right squares are also homotopy pullbacks, since the horizontal morphisms are weak equivalences. Thus, the outer square is a homotopy pullback; one easily checks that the natural equivalence $f \circ \pi \Rightarrow \underline{y}$ defined by pasting the three squares agrees with the natural transformation $\alpha$ defined in \hyperref[notation:rund_um_undercategories]{Notation~\ref*{notation:rund_um_undercategories}}.
\end{proof}

\begin{proof}[Proof of Theorem \ref{prop:local_systems_are_beck_chevalley}]
  We have already shown that $p$ is a bicartesian fibration, and we know that $\S$ admits pullbacks. Therefore, it suffices to show that $p$ has \hyperref[item:beck_chevalley_condition]{Property~(BC)}.

  We consider a square $\sigma\colon \Delta^{1} \times \Delta^{1} \cong \Delta^{\{0,1,2\}} \amalg_{\Delta^{\{0,2\}}}\Delta^{\{0,1',2\}} \to \LS(\category{C})$ with the following properties.
  \begin{enumerate}
    \item The restriction $\sigma|\Delta^{\{0,1\}}$ is $p$-cartesian.

    \item The restriction $\sigma|\Delta^{\{1',2\}}$ is $p$-cartesian.

    \item The restriction $\sigma|\Delta^{\{1,2\}}$ is $p$-cocartesian.

    \item The square $\sigma$ lies over a pullback square
      \begin{equation*}
        p(\sigma) =
        \begin{tikzcd}
          X_{0}
          \arrow[r, "f'"]
          \arrow[d, swap, "g'"]
          & X_{1'}
          \arrow[d, "g"]
          \\
          X_{1}
          \arrow[r, "f"]
          & X_{2}
        \end{tikzcd}
      \end{equation*}
      in $\S$.
  \end{enumerate}
  We need to show that $\sigma|\Delta^{\{0,1'\}}$ is $p$-cocartesian.

  The map $\sigma$ adjunct to a map
  \begin{equation*}
    \tau\colon \Delta^{\{0,1,2,\overline{2}, \overline{1}, \overline{0}\}}_{\dagger}\amalg_{\Delta^{\{0,2,\overline{2}, \overline{0}\}}_{\dagger}} \Delta^{\{0,1',2,\overline{2}, \overline{1}', \overline{0}\}}_{\dagger} \to \ICCat
  \end{equation*}
  such that $\tau|\Delta^{\{\overline{2}, \overline{1}, \overline{0}\}} \amalg_{\Delta^{\{\overline{2},\overline{0}\}}} \Delta^{\{\overline{2}, \overline{1}', \overline{0}\}}$ is the constant functor with value $\category{C}$. We can read off the following further properties of $\tau$, corresponding to the properties of $\sigma$ above.
  \begin{enumerate}
    \item The simplices $\tau|\Delta^{\{0,1,\overline{0}\}}$ and $\tau|\Delta^{\{0,\overline{1},\overline{0}\}}$ are thin.

    \item The simplices $\tau|\Delta^{\{1',2,\overline{1}'\}}$ and $\tau|\Delta^{\{1',\overline{2},\overline{1}'\}}$ are thin.

    \item The simplex $\tau|\Delta^{\{1,2,\overline{1}\}}$ is left Kan.

    \item The restriction $\tau|\Delta^{\{0,1,2\}} \amalg_{\Delta^{\{0,2\}}}\Delta^{\{0,1',2\}}$ is equal to $p(\sigma)$.
  \end{enumerate}
  The condition that $\sigma|\Delta^{\{0, 1'\}}$ is $p$-cocartesian corresponds to the condition that the simplex $\tau|\Delta^{\{0,1,\overline{0}\}}$ is left Kan.

  The diagram $\tau$ contains a lot of redundant data, which we would now like to consolidate. We first shuffle some data around our diagram. Applying \hyperref[lemma:transport_left_kan_simplices]{Lemma~\ref*{lemma:transport_left_kan_simplices}} to the simplex $\tau|\Delta^{\{1,2,\overline{1}, \overline{0}\}}$, we see that $\tau|\Delta^{\{1,2,\overline{0}\}}$ is left Kan. Applying \hyperref[lemma:transport_thin_simplices]{Lemma~\ref*{lemma:transport_thin_simplices}} to the simplex $\tau|\Delta^{\{1',2, \overline{1}', \overline{0}\}}$, we see that $\tau|\Delta^{\{1',2, \overline{0}\}}$ is thin.

  We can now study a subdiagram which contains all the information we need. We consider the restriction
  \begin{equation*}
    \tau|\Delta^{\{0,1,2,\overline{0}\}} \amalg_{\Delta\{0,2,\overline{0}\}} \Delta^{\{0,1',2,\overline{0'}\}}.
  \end{equation*}
  This is determined, up to specific choices of compositions, by the diagram
  \begin{equation*}
    \begin{tikzcd}[row sep=large, column sep=large]
      X_{0}
      \arrow[r, "f'"]
      \arrow[d, swap, "g'"]
      & X_{1'}
      \arrow[d, swap, "g"]
      \arrow[r, "F", ""{swap, name=M}]
      & \category{C}
      \\
      X_{1}
      \arrow[r, "f"]
      & X_{2}
      \arrow[ur, swap, "G"]
      \arrow[from=M, Rightarrow, shorten=2ex, swap, "\eta"]
    \end{tikzcd}.
  \end{equation*}
  Here, the square pictured is $\sigma$, and the triangle is $\tau|\Delta^{\{1',2,\overline{0}\}}$.

  Since the simplices $\tau|\Delta^{\{0,1,\overline{0}\}}$ and $\tau|\Delta^{\{1,2, \overline{0}\}}$ are thin, the pasting of the square and the triangle above is homotopic to the restriction $\tau|\Delta^{\{0,1,\overline{0}\}}$. This is the triangle which we wish to show is left Kan. According to \hyperref[lemma:left_kan_ext_along_kan_complexes]{Lemma~\ref*{lemma:left_kan_ext_along_kan_complexes}} it suffices to show that for all $x \in X_{1}$, the pasting diagram
  \begin{equation*}
    \begin{tikzcd}[row sep=large, column sep=large]
      (X_{0})_{/x}
      \arrow[r, "\pi"]
      \arrow[d]
      & X_{0}
      \arrow[r, "f'"]
      \arrow[d, swap, "g'"]
      & X_{1'}
      \arrow[d, swap, "g"]
      \arrow[r, "F", ""{swap, name=M}]
      & \category{C}
      \\
      \{x\}
      \arrow[r, hook]
      & X_{1}
      \arrow[r, "f"]
      & X_{2}
      \arrow[ur, swap, "G"]
      \arrow[from=M, Rightarrow, shorten=2ex, swap, "\eta"]
    \end{tikzcd}
  \end{equation*}
  exhibits $G(f(x))$ as the colimit of $F \circ f' \circ \pi$. But applying the pasting law for homotopy pullbacks, this follows directly from the assumption that $\eta$ exhibits $G$ as a left Kan extension of $F$ along $f$.
\end{proof}

\section{Monoidal pull-push of local systems}
\label{sec:the_monoidal_construction}

The classical Grothendieck construction provides an equivalence between functors $\hat{p}\colon \category{D} \to \Cat$ and cartesian fibrations $p\colon \int \hat{p} \to \category{D}\op$. There are many conditions one can impose upon the functor $\hat{p}$, and many structures one can endow it with; it is natural to wonder whether these properties and structures can be captured in the fibration $p$.

An answer is known in the case that $\hat{p}$ is lax monoidal (the author learned about this fact from \cite{moeller2018monoidal}). Suppose $p\colon \category{C} \to \category{D}$ is a cartesian fibration of 1-categories which classifies a pseudofunctor $\hat{p}\colon \category{D}\op \to \Cat$. Further suppose the following:
\begin{itemize}
  \item The category $\category{C}$ carries a symmetric monoidal structure $(\otimes, I_{\category{C}}, \ldots)$.

  \item The category $\category{D}$ carries a symmetric monoidal structure $(\boxtimes, I_{\category{D}}, \ldots)$.

  \item The functor $p$ is strong monoidal: for objects $x$, $y \in \category{C}$, we have that $p(x \otimes y) = p(x) \boxtimes p(y)$, and $p(I_{\category{C}}) = I_{\category{D}}$.\footnote{The reason for the strict equality in this condition is that we have chosen specific monoidal functors $\otimes$ and $\boxtimes$, and specific unit objects $I_{\category{C}}$ and $I_{\category{D}}$. Later, we will replace these strict choices by fibrations which give the same data up to coherent homotopy.}

  \item The monoidal product $\otimes$ preserves $p$-cartesian morphisms in the sense that if $f$ and $g$ are $p$-cartesian morphisms, then $f \otimes g$ is also $p$-cartesian.
\end{itemize}
Under these conditions, the pseudofunctor $\hat{p}$ carries a lax monoidal structure $(\category{D}\op, \boxtimes) \to (\Cat, \times)$, described as follows. 
\begin{itemize}
  \item Restricting the tensor product
    \begin{equation*}
      \otimes\colon \category{C} \times \category{C} \to \category{C}
    \end{equation*}
    to the fibers $\category{C}_{d}$ and $\category{C}_{d'}$ of $p$ over $d$ and $d'$ yields the structure maps
    \begin{equation*}
      \hat{p}(d) \times \hat{p}(d') \to \hat{p}(d \boxtimes d').
    \end{equation*}
    That these structure maps form a pseudonatural transformation, i.e.\ that the necessary squares commute up to specified homotopy, follows from the assumption that the monoidal product preserves $p$-cartesian morphisms in each slot.

  \item The requirement that $p(I_{\category{C}}) = I_{\category{D}}$ can be rephrased to say that $I_{\category{D}} \in \hat{p}(I_{\category{C}})$. This gives us the structure map $\ast \to \hat{p}(I_{\category{D}})$, where $\ast$ is the unit object in $\Cat$.
\end{itemize}

If we want to study covariant rather than contravariant pseudofunctors, we should replace cartesian fibrations by cocartesian fibrations; of course, now we must assume that $\otimes$ preserves cocartesian edges rather than cartesian. The rest of the theory remains unchanged. These results also remain true in the $\infty$-categorical case, as we will show in \hyperref[ssc:the_lax_monoidal_grothendieck_construction]{Subsection~\ref*{ssc:the_lax_monoidal_grothendieck_construction}}.

As we saw in the last chapter, if we have a bicartesian fibration which satisfies the Beck-Chevalley condition (i.e.\ a Beck-Chevalley fibration), we can combine both cartesian and cocartesian functorialities into pull-push functoriality. Our aim in \hyperref[ssc:monoidal_beck_chevalley_fibrations]{Subsection~\ref*{ssc:monoidal_beck_chevalley_fibrations}} will be to show that the theory of Beck-Chevalley fibrations admits a monoidal generalization. More specifically, we will show the following. Consider a Beck-Chevalley fibration $r\colon \category{X} \to \category{T}$ such that $\category{X}$ carries a symmetric monoidal structure $\otimes$, and $\category{T}$ carries a symmetric monoidal structure $\boxtimes$. Under the assumption that $r$ is strong monoidal, and that $\otimes$ preserves both cartesian and cocartesian edges, the induced functor
\begin{equation*}
  \Span(\category{T}) \to \ICat
\end{equation*}
constructed in the previous chapter carries a lax monoidal structure. Note that this statement is not entirely original. A similar result in a somewhat different context is proved in \cite{spectralmackeyfunctors2}.

We will then apply our results to local systems. We will show that if $\category{C}$ is a symmetric monoidal $\infty$-category, then the $\infty$-category of $\category{C}$-local systems carries a symmetric monoidal structure, defined on objects by
\begin{equation*}
  (\mathcal{F}\colon X \to \category{C}) \otimes (\mathcal{G}\colon Y \to \category{C}) \quad=\quad X \times Y \overset{\mathcal{F} \times \mathcal{G}}{\to} \category{C} \times \category{C} \overset{\otimes}{\to} \category{C}
\end{equation*}
and that the functor $\LS(\category{C}) \to \S$ is can be given the structure of a monoidal Beck-Chevalley fibration, thus classifying a lax monoidal functor
\begin{equation*}
  (\Span(\S), \tilde{\times}) \to (\ICat, \times).
\end{equation*}


\subsection{The lax monoidal Grothendieck construction}
\label{ssc:the_lax_monoidal_grothendieck_construction}

Our aim in this section is to show an $\infty$-categorical version of the statement that a monoidal cartesian fibration $p\colon (\category{C}, \otimes) \to (\category{D}, \boxtimes)$ such that $\otimes$ preserves $p$-cartesian morphisms straightens to a lax monoidal pseudofunctor $\category{D}\op \to \Cat$ (as explained more concretely in the introduction to \hyperref[sec:the_monoidal_construction]{Section~\ref*{sec:the_monoidal_construction}}). We will assume familiarity with the theory of symmetric monoidal $\infty$-categories as laid out in \cite[Chap.~2]{luriehigheralgebra}; roughly, a symmetric monoidal $\infty$-category is defined to be a cocartesian fibration classifying a commutative monoid in $\ICat$. Our first goal will be to rephrase some of the results there in terms of cartesian fibrations rather than cocartesian fibrations.

\begin{notation}
  We will frequently refer to the following maps in $\Finp\op$.
  \begin{itemize}
    \item Denote by $\rho^{i}\colon \langle n \rangle \to \langle 1 \rangle$ the map in $\Finp$ sending $i \mapsto 1$ and everything else to $\ast$. Denote the same map in $\Finp\op$ by $\rho_{i}$.

    \item Denote by $\mu\colon \langle 2 \rangle \to \langle 1 \rangle$ the active map in $\Finp$. We will denote the same map in $\Finp\op$ also by $\mu$.
  \end{itemize}
\end{notation}

\begin{definition}
  A \defn{CSMC} (contravariantly-presented symmetric monoidal $\infty$-category) is a cartesian fibration $p\colon \category{C}_{\otimes} \to \Finp\op$ such that the contravariant transport maps $\rho_{i}^{*}\colon (\category{C}_{ \otimes})_{\langle n \rangle} \to (\category{C}_{ \otimes})_{\langle 1 \rangle}$ are the canonical projections exhibiting $(\category{C}_{ \otimes})_{\langle n \rangle}$ as an $n$-fold homotopy product.
\end{definition}

\begin{notation}
  For any CSMC $\category{C}_{\otimes} \to \Finp\op$, we will denote the fiber $(\category{C}_{\otimes})_{\langle 1 \rangle}$ simply by $\category{C}$.
\end{notation}

The equivalence $(\category{C}_{\otimes})_{\langle n \rangle} \simeq \category{C}^{n}$ allow us to trade maps into $(\category{C}_{\otimes})_{\langle n \rangle}$ for $n$ maps into $\category{C}$, well-defined up to equivalence. Given a diagram $a\colon K \to (\category{C}_{\otimes})_{\langle n \rangle}$, we will often abuse this terminology by calling any such corresponding diagrams $a_{i}\colon K \to \category{C}$ `the' components of $a$, as long as we are making reference only to properties of these components which are preserved under equivalence.

\begin{definition}
  \label{def:terminology_about_csmcs}
  A \defn{map} between CSMCs $q\colon \category{C}_{\otimes} \to \Finp\op$ and $p\colon \category{D}_{\boxtimes} \to \Finp\op$ is a functor $r$ making the diagram
  \begin{equation*}
    \begin{tikzcd}
      \category{C}_{\otimes}
      \arrow[rr, "r"]
      \arrow[dr, swap, "q"]
      && \category{D}_{\boxtimes}
      \arrow[dl, "p"]
      \\
      & \Finp\op
    \end{tikzcd}
  \end{equation*}
  commute. Given a map of CSMCs as above, we will further make use of the following terminology.
  \begin{itemize}
    \item The map $r$ is a \defn{monoidal functor} if it sends $q$-cartesian morphisms to $p$-cartesian morphisms.

      \textit{This implies, for example, that $r(x \otimes y) \simeq r(x) \boxtimes r(y)$, and that $r(I_{\category{C}}) \simeq I_{\category{D}}$. It also automatically implies that the diagrams encoding associativity, etc.,\ commute up to coherent homotopy.}

    \item An edge $f$ in $(\category{C}_{\otimes})_{\langle n \rangle}$ is \defn{componentwise cartesian (resp.\ componentwise cocartesian)} if for each $1 \leq i \leq n$, the transport $\rho_{i}^{*}(f)$ is $r|\langle 1 \rangle$-cartesian (resp.\ cocartesian).

      \textit{We can think of a morphism $f$ in $(\category{C}_{\otimes})_{\langle n \rangle}$ as an $n$-tuple of morphisms $f_{i}$ in $\category{C}$. We say that $f$ is componentwise (co)cartesian if each component $f_{i}$ is (co)cartesian as an edge of the underlying fibration $\category{C} \to \category{D}$.}

    \item The tensor product $\otimes$ \defn{preserves cartesian (resp.\ cocartesian) edges} if for all $\phi\colon \langle n \rangle \leftarrow \langle m \rangle \in \Finp\op$, the associated functor $\phi^{*}\colon (\category{C}_{\otimes})_{\langle m \rangle} \to (\category{C}_{\otimes})_{\langle n \rangle}$ sends componentwise cartesian (resp.\ componentwise cocartesian) morphisms contained in the fiber $(\category{C})_{\langle m \rangle}$ to componentwise cartesian (resp.\ componentwise cocartesian) morphisms contained in the fiber $(\category{C}_{\otimes})_{\langle n \rangle}$.

      \textit{We will show in \hyperref[lemma:co_cartesian_preservation_determined_by_bifunctor]{Lemma~\ref*{lemma:co_cartesian_preservation_determined_by_bifunctor}} that this is equivalent to demanding that if morphisms $f$ and $g$ in $\category{C}$ are $r|\langle 1 \rangle$-(co)cartesian, then $f \otimes g$ is as well.}

    \item The CSMC $q\colon \category{C}_{\otimes} \to \Finp\op$ is \defn{cartesian-compatible (resp.\ cocartesian-compatible)} an edge $f$ in the fiber $(\category{C}_{\otimes})_{\langle n \rangle}$ is $r$-cartesian (resp.\ cocartesian) if and only if it is componentwise cartesian.

      \textit{This is a technical condition which is true for all CSMCs which we will consider; it says that the way we encode the monoidal structure on $\category{C}$ into a cartesian fibration is well-behaved with respect to the fibration $r$.}
  \end{itemize}
\end{definition}

\begin{lemma}
  \label{lemma:co_cartesian_preservation_determined_by_bifunctor}
  Let
  \begin{equation*}
    \begin{tikzcd}
      \category{C}_{\otimes}
      \arrow[rr, "r"]
      \arrow[dr, swap, "q"]
      && \category{D}_{\boxtimes}
      \arrow[dl, "p"]
      \\
      & \Finp\op
    \end{tikzcd}
  \end{equation*}
  be a map between CSMCs. The tensor product $\otimes$ preserves cartesian (resp.\ cocartesian) morphisms if and only if for all $r|\langle 1 \rangle$-cartesian (resp.\ $r|\langle 1 \rangle$-cocartesian) morphisms $f$ and $g$ in $\category{C}$, the morphism $f \otimes g$ is also $r|\langle 1 \rangle$-cartesian (resp.\ $r|\langle 1 \rangle$-cocartesian).
\end{lemma}
\begin{proof}
  We consider the cartesian case. The cocartesian case is identical.

  Suppose that $\otimes$ preserves $r$-cartesian morphisms, and let $f$ and $g$ be $r|\langle 1 \rangle$-cartesian morphisms in $\category{C}$. Then there is a morphism $[f, g]$ in $( \category{C}_{\otimes} )_{\langle 2 \rangle}$ with components $f$ and $g$. Taking $\phi$ to be the active map $\langle 2 \rangle \to \langle 1 \rangle$ shows that $f \otimes g$ is $r|\langle 1 \rangle$-cartesian.

  Conversely, suppose that for any $r|\langle 1 \rangle$-cartesian morphisms $f$ and $g$, the morphism $f \otimes g = \phi^{*}([f, g])$ in $\category{C}$ is $r|\langle 1 \rangle$-cartesian. By associativity of the tensor product and induction, we have that the $n$-ary product of $r|\langle 1 \rangle$-cartesian morphisms is again $r|\langle 1 \rangle$-cartesian for $n \geq 2$. The image of the map $\alpha^{*}$, for $\alpha\colon \langle 0 \rangle \to \langle 1 \rangle$, i.e.\ the 0-ary tensor product, is equivalent to $\id_{I}$, where $I$ is the unit object of $\category{C}$, and is thus an equivalence, hence also $r|\langle 1 \rangle$-cartesian.

  Now let $\psi\colon \langle m \rangle \to \langle n \rangle$ be a general map in $\Finp$, and let $f = [f_{1}, \ldots, f_{m}]$ be a componentwise cartesian map. The $i$th component of $\psi^{*}(f)$ is given by $\bigotimes_{\psi(j) = i} f_{j}$, which is $r|\langle 1 \rangle$-cartesian no matter the cardinality of $\psi^{-1}(i)$. Thus, $\psi^{*}(f)$ is $r$-cartesian.
\end{proof}


\begin{lemma}
  \label{lemma:equivalent_conditions_to_be_cartesan}
  Let $q\colon \category{C}_{\otimes} \to \Fin\op$ and $p\colon \category{D}_{\otimes} \to \Finp\op$ be CSMCs, and let $r$ be an inner fibration $\category{C}_{\otimes} \to \category{D}_{\boxtimes}$ such that the diagram
  \begin{equation*}
    \begin{tikzcd}
      \category{C}_{\otimes}
      \arrow[rr, "r"]
      \arrow[dr, swap, "q"]
      && \category{D}_{\boxtimes}
      \arrow[dl, "p"]
      \\
      & \Finp\op
    \end{tikzcd}
  \end{equation*}
  commutes. Then the following are equivalent.
  \begin{enumerate}
    \item The map $r$ is a cartesian fibration.

    \item The map $r$ has the following properties.
      \begin{enumerate}
        \item The restriction $r|\langle 1 \rangle$ is a cartesian fibration.

        \item The map $r$ is a monoidal functor.

        \item The tensor product $\otimes$ preserves cartesian edges.

        \item The CSMC $q$ is cartesian-compatible.
      \end{enumerate}
  \end{enumerate}
\end{lemma}
\begin{proof}
  Suppose that 1.\ holds, i.e.\ that $r$ is a cartesian fibration. Then a) holds: $r|\langle 1 \rangle$ is a cartesian fibration because the pullback of a cartesian fibration is again a cartesian fibration.

  We now show that b) holds, i.e.\ that the map $r$ sends $q$-cartesian morphisms to $p$-cartesian morphisms. To this end, let $f\colon c \to c' \in \category{C}_{\otimes}$ be a $q$-cartesian morphism, and consider the image $p(f)\colon r(c) \to r(c')$ in $\category{D}_{\otimes}$. Let $g\colon \tilde{d} \to r(c')$ be a $p$-cartesian lift of $q(f) \in \Finp\op$, and $\hat{g}\colon \hat{d} \to c'$ a $r$-cartesian lift of $g$. By \cite[Prop.~2.4.1.3]{highertopostheory}, $\hat{g}$ is a $q$-cartesian lift of $q(f)$, so $f$ and $\hat{g}$ are equivalent as morphisms in $\category{C}_{\otimes}$. Thus $r(f)$ and $g$ are equivalent as morphisms in $\category{D}_{\otimes}$, so $r(f)$ is $p$-cartesian since $g$ is. This proves b).

  We now show that d) holds. We do this in two steps. First, we show that if some morphism $f$ in $(\category{C}_{\otimes})_{\langle n \rangle}$ is $r$-cartesian, then for all $\phi\colon \langle n \rangle \to \langle m \rangle$, the morphism $\phi^{*}(f)$ is $r$-cartesian. This follows from the fact that for a square in $\category{C}_{\otimes}$
  \begin{equation*}
    \begin{tikzcd}
      c_{0}'
      \arrow[r, "\phi^{*}(f)"]
      \arrow[d, swap, "u"]
      & c_{1}'
      \arrow[d, "v"]
      \\
      c_{0}
      \arrow[r, "f"]
      & c_{1}
    \end{tikzcd}
  \end{equation*}
  in which $u$ and $v$ are $q$-cartesian (and hence $r$-cartesian) and $f$ is $r$-cartesian, $\phi^{*}(f)$ is also $r$-cartesian. Applying this result to the inert maps $\rho^{i}\colon \langle n \rangle \to \langle 1 \rangle$ implies that if a morphism $f$ in $(\category{C}_{\otimes})_{\langle n \rangle}$ is $r$-cartesian, then the components $f_{i}$ are $r$-cartesian, hence also $r|\langle 1 \rangle$-cartesian. This is one direction of d).

  We now show the other direction of d): if a morphism in $(\category{C}_{\otimes})_{\langle n \rangle}$ is componentwise cartesian, then it is $r$-cartesian. To see this, let $f_{i}$ be $r|\langle 1 \rangle$-cartesian morphisms, $1 \leq i \leq n$. We can find a morphism $f \in (\category{C}_{\otimes})_{\langle n \rangle}$ with components $f_{i}$. We would like to show that $f$ is $r$-cartesian. To this end, consider the morphisms $r(f_{i}) = \overline{f}_{i}$. We can find a morphism $\overline{f}$ in $(\category{D}_{\boxtimes})_{\langle n \rangle}$ with components $\overline{f}_{i}$, and can take a $r$-cartesian lift $\hat{f}$ of $\overline{f}$. The components of $\hat{f}$ are equivalent to the $f_{i}$, and are hence $r$-cartesian. But then $f$ is equivalent to $\hat{f}$, so $f$ is $r$-cartesian as promised. This proves d).

  We have now shown that under the assumption that 1.\ holds, a the pushforward maps preserve $r$-cartesian morphisms. This, together with d), proves c).

  Now, suppose that 2.\ holds. We immediately note that c) implies that for each $\langle n \rangle \in \Finp\op$, the restriction $r|\langle n \rangle$ is a cartesian fibration, and an edge is $r|\langle n \rangle$-cartesian if and only if it is $r$-cartesian.

  We now show that $r$ admits cartesian lifts. Let $f\colon d \to d'$ be an edge in $\category{D}_{\boxtimes}$ lying over an edge $\phi\colon \langle n \rangle \leftarrow \langle m \rangle$ in $\Finp\op$, and let $c'$ be a lift of $d'$ to $\category{C}_{\otimes}$. We can take a $q$-cartesian lift $g\colon c'' \to c$ of $f$, whose image in $\category{D}_{\boxtimes}$ is by b) a $p$-cartesian map $h\colon d'' \to d'$. This gives us the solid data
  \begin{equation*}
    \begin{tikzcd}
      & c''
      \arrow[dr, "g"]
      \\
      c
      \arrow[ur, dashed, "k"]
      \arrow[rr, dashed, swap, "\ell"]
      && c'
    \end{tikzcd}
    \qquad\text{in }\category{C}_{\otimes},
  \end{equation*}
  \begin{equation*}
    \begin{tikzcd}
      & d''
      \arrow[dr, "h"]
      \\
      d
      \arrow[ur, dashed, "j"]
      \arrow[rr, swap, "f"]
      && d'
    \end{tikzcd}
    \qquad\text{in }\category{D}_{\boxtimes}
  \end{equation*}
  \begin{equation*}
    \begin{tikzcd}
      & \langle n \rangle
      \arrow[dl, equals]
      \\
      \langle n \rangle
      && \langle m \rangle
      \arrow[ul, swap, "\phi"]
      \arrow[ll, "\phi"]
    \end{tikzcd}.
    \qquad\text{in }\Finp\op.
  \end{equation*}
  Using the fact that $h$ is $p$-cartesian, we can fill the $2$-simplex in $\category{D}_{\boxtimes}$, giving us in particular a map $j\colon d \to d''$. We can lift $j$ to an $r|\langle n \rangle$-cartesian morphism $k\colon c \to c''$, which is therefore also $r$-cartesian. Using that $r$ is an inner fibration, we can compose $g$ and $k$ relative to the simplex in $\category{D}_{\boxtimes}$, giving us a lift $\ell$ of $f$. But $\ell$ is the composition of two $r$-cartesian morphisms, and hence itself $r$-cartesian.
\end{proof}

\begin{note}
  \hyperref[lemma:equivalent_conditions_to_be_cartesan]{Lemma~\ref*{lemma:equivalent_conditions_to_be_cartesan}} gives us an $\infty$-categorical version of the monoidal Grothendieck construction of \cite{moeller2018monoidal} discussed at the begininning of this section. There, a monoidal structure is modelled as tuple $(\category{C}, \otimes, \dots)$. Because of the specific choice of a bifunctor $\otimes$ and the surrounding corresponding coherence data, strong conditions must be placed on the functor $r$: it must be monoidal `on the nose.' Here, we model our monoidal structure as a cartesian fibration, leaving these choices unmade. The strong notion of monoidality is thus replaced by the usual one.

  The connection between the above result and the monoidal Grothendieck construction, somewhat explicitly, is as follows. \hyperref[lemma:equivalent_conditions_to_be_cartesan]{Lemma~\ref*{lemma:equivalent_conditions_to_be_cartesan}} tells us that the data of a functor $(\category{C}, \otimes) \to (\category{D}, \boxtimes)$ satisfying conditions analogous to those given in \cite{moeller2018monoidal} is equivalent to a cartesian fibration between the corresponding CSMCs. A cartesian fibration $r$ between CSMCs $p\colon \category{C}_{\otimes} \to \Finp\op$ and $q\colon \category{D}_{\boxtimes} \to \Finp\op$ as above straightens to a $\category{D}_{\boxtimes}\op$-monoid in $\ICat$. But a $\category{D}_{\boxtimes}\op$-monoid in $\ICat$ can be essentially uniquely extended to an $\category{D}_{\boxtimes}\op$-algebra object in $\ICatCart$ \cite[Prop~2.4.2.4--2.4.2.6]{luriehigheralgebra}, which is to say, a lax monoidal functor $(\category{D}\op, \boxtimes) \to (\ICat, \times)$.
\end{note}

In the remainder of this section, we study a special type of CSMC, those whose tensor product is given by the cartesian product.

\begin{definition}
  A \defn{CSMC} $p\colon \category{D}_{\boxtimes} \to \Finp\op$ is \defn{cartesian} if the unit object $\ast$ is final, and if for all objects $X$ and $Y$ in $\category{D}$, the canonical maps $X \times \ast \leftarrow X \boxtimes Y \to \ast \times Y$ exhibit $X \boxtimes Y$ as the product of $X$ and $Y$.
\end{definition}

\begin{example}
  Let $\category{T}$ be a category with finite products. Then $\category{T}\op$ admits finite coproducts, and we can consider the cocartesian monoidal structure
  \begin{equation*}
    (\category{T}\op)^{\amalg} \to \Finp
  \end{equation*}
  of \cite[Construction~2.4.3.1]{luriehigheralgebra}. Taking the opposite of this functor gives us a cartesian fibration
  \begin{equation*}
    \category{T}_{\times} \to \Finp\op.
  \end{equation*}
  That this is a cartesian CSMC follows immediately from the fact that $(\category{T}\op)^{\amalg} \to \Finp$ is a cocartesian symmetric monoidal category.
\end{example}

In any category $\category{T}$ with pullbacks, one can form a category $\Span(\category{T})$ of spans in $\category{T}$. If $\category{T}$ also has a terminal object, hence all finite limits (including, of course, finite products), then the category of spans inherits a monoidal structure via
\begin{equation*}
  \left(
  \begin{tikzcd}[row sep=small, column sep=small]
    & Z
    \arrow[dl]
    \arrow[dr]
    \\
    X
    && Y
  \end{tikzcd}
  \right) \otimes \left(
  \begin{tikzcd}[row sep=small, column sep=small]
    & Z'
    \arrow[dl]
    \arrow[dr]
    \\
    X'
    && Y'
  \end{tikzcd}
  \right) \quad =
  \begin{tikzcd}[row sep=small, column sep=small]
    & Z \times Z'
    \arrow[dl]
    \arrow[dr]
    \\
    X \times X'
    && Y \times Y'
  \end{tikzcd}
\end{equation*}

We now construct the monoidal structure on this category of spans explicitly, starting from any cartesian CSMC $\category{T}_{\times} \to \Finp\op$.

\begin{proposition}
  \label{prop:cartesian_CSMC_always_exists}
  Suppose $\category{T}_{\times} \to \Finp\op$ is a cartesian CSMC whose underlying category $\category{T}$ admits pullbacks. Then there exists a (cocartesian-presented) symmetric monoidal category $\Span(\category{T})^{\times} \to \Finp$.
\end{proposition}
\begin{proof}
  We upgrade this functor to a functor of triples. We will consider the following triples.
  \begin{itemize}
    \item We define a triple structure $\triple{F}$ on $\Finp\op$, where
      \begin{itemize}
        \item $\category{F} = \Finp\op$

        \item $\category{F}\downdag = (\Finp\op)^{\simeq}$

        \item $\category{F}\updag = \Finp\op$
      \end{itemize}
      This is obviously adequate.

    \item We define a triple $\triple{T}$ as follows.
      \begin{itemize}
        \item $\category{T} = \category{T}_{\times}$

        \item $\category{T}\downdag = \category{T}_{\times} \times_{\Finp\op}(\Finp\op)^{\simeq}$

        \item $\category{T}\updag = \category{T}_{\times}$
      \end{itemize}
  \end{itemize}

  To see that this is adequate, we first note that $p$ admits relative pullbacks. To see this, note that each fiber admits pullbacks by virtue our assumption that $\category{T}$ admits pullbacks, and the identifications $( \category{T}_{\times} )_{\langle n \rangle} \simeq \category{T}^{n}$. The functoriality coming from $p$ implements products, and therefore commute with pullbacks. Thus, $\category{T}_{\times}$ admits pullbacks, and $p$ preserves pullbacks. This immediately implies that $\triple{T}$ is adequate:
  \begin{enumerate}
    \item Pullbacks of this form are simply squares with horizontal morphisms given by equivalences.

    \item Any pullback square lies over a pullback square in $\Finp\op$, so this reduces to the lemma about cartesian morphisms.
  \end{enumerate}

  We thus consider the map of triples
  \begin{equation*}
    \pi\colon \triple{T} \to \triple{F}.
  \end{equation*}

  One readily checks that $\pi$ satisfies the conditions of \hyperref[thm:old_barwick]{Theorem~\ref*{thm:old_barwick}}:
  \begin{itemize}
    \item The first condition holds because cocartesian morphisms lying over equivalences are themselves equivalences, hence also cartesian, so we can solve the relevant lifting problems using cartesian lifts.

    \item The second condition holds because a square whose bottom-horizontal morphism is an equivalence is pullback if and only if its top-horizontal morphism is an equivalence.
  \end{itemize}

  This gives us a functor $\Span\triple{T} \to \Span\triple{F}$. Pulling back along the equivalence $\Finp \to \Span\triple{F}$ gives us the cocartesian fibration functor $\Span(\category{T})^{\times} \to \Finp$. It remains only to check that the maps $\rho^{i}_{*}\colon \Span(\category{T})^{\times}_{\langle n \rangle} \to \Span(\category{T})^{\times}_{\langle 1 \rangle}$ are the canonical projections exhibiting
  \begin{equation*}
    \Span(\category{T})^{\times}_{\langle n \rangle} \simeq \left( \Span(\category{T})^{\times}_{\langle 1 \rangle} \right)^{n}.
  \end{equation*}

  In order to show this, we should show that for each simplicial set $K$, the map $a$ in the diagram
  \begin{equation*}
    \begin{tikzcd}
      \Map(K, \Span(\category{T})^{\times}_{\langle n \rangle})
      \arrow[r, "a"]
      \arrow[d, swap, "\simeq"]
      & \prod_{i = 1}^{n} \Map(K, \Span(\category{T})^{\times}_{\langle 1 \rangle})
      \arrow[d, "\simeq"]
      \\
      \Map^{\aCart}(\sd(K), (\category{T}_{\times})_{\langle n \rangle}) \arrow[r, "b"]
      \arrow[d, hook, "i"{swap}]
      & \prod_{i = 1}^{n} \Map^{\aCart}(\sd(K), (\category{T}_{\times})_{\langle 1 \rangle})
      \arrow[d, hook, "j"]
      \\
      \Map(\sd(K), (\category{T}_{\times})_{\langle n \rangle}) \arrow[r, "c", "\simeq"{swap}]
      & \prod_{i = 1}^{n} \Map(\sd(K), (\category{T}_{\times})_{\langle 1 \rangle})
    \end{tikzcd}
  \end{equation*}
  is an isomorphism in the category $\h \Kan$. In order to show this, it suffices to show that the map $b$ is an isomorphism in $\h \Kan$. We note that the morphisms $i$ and $j$ are inclusions of connected components, and the map $c$ is an isomorphism because $\category{T}_{\times} \to \Finp\op$ is a CSMC, so in order to show that $b$ is an isomorphism in $\h\Kan$, it suffices to show that it is essentially surjective. This follows from the fact that a square in $(\category{T}_{\times})_{\langle n \rangle} \simeq ((\category{T}_{\times})_{\langle 1 \rangle})^{n}$ is pullback if and only if each component is pullback.
\end{proof}

\subsection{Monoidal Beck-Chevalley fibrations}
\label{ssc:monoidal_beck_chevalley_fibrations}

In this section, we combine several results from earlier sections.

\begin{itemize}
  \item In \hyperref[sec:the_non_monoidal_construction]{Section~\ref*{sec:the_non_monoidal_construction}}, we showed that a bicartesian fibration $p\colon \category{X} \to \category{T}$ satisfying the Beck-Chevalley condition allows us to combine the functoriality $\category{T} \to \ICat$ and $\category{T}\op \to \ICat$ into push-pull functoriality $\Span(\category{T}) \to \ICat$. We called such bicartesian fibrations \emph{Beck-Chevalley fibrations.}

  \item In \hyperref[ssc:the_lax_monoidal_grothendieck_construction]{Subsection~\ref*{ssc:the_lax_monoidal_grothendieck_construction}} we showed that a cartesian fibration $p\colon \category{C} \to \category{D}$ between monoidal categories whose tensor products were subject to certain compatibility conditions classifies a lax monoidal functor $\category{D}\op \to \ICat$; and dually, that a cocartesian fibration $p$ satisfying dual compatibility conditions classifies a lax monoidal functor $\category{D} \to \ICat$.
\end{itemize}

We will now show that a Beck-Chevalley fibration $p\colon \category{X} \to \category{T}$ between monoidal categories whose tensor products satisfy appropriate compatibility conditions classifies a lax monoidal functor $\Span(\category{T}) \to \ICat$. We will call such fibrations \emph{monoidal Beck-Chevalley fibrations.}

\begin{definition}
  \label{def:monoidal_beck_chevalley_fibration}
  A \defn{monoidal Beck-Chevalley fibration} is a functor $r$ of CSMCs
  \begin{equation*}
    \begin{tikzcd}
      \category{X}_{\otimes}
      \arrow[rr, "r"]
      \arrow[dr, swap, "q"]
      && \category{T}_{\times}
      \arrow[dl, "p"]
      \\
      & \Finp\op
    \end{tikzcd},
  \end{equation*}
  where $\category{T}_{\times}$ is a cartesian CSMC and $\category{X}_{\otimes}$ is cartesian-compatible, with the following characteristics.
  \begin{enumerate}
    \item[(M1)] The map $r|\langle 1 \rangle$ is a Beck-Chevalley fibration.

    \item[(M2)] The map $r$ is monoidal.

    \item[(M3)] The tensor product $\otimes$ preserves $r$-cartesian morphisms.

    \item[(M4)] The tensor product $\otimes$ preserves $r$-cocartesian morphisms.

    \item[(M5)] The CSMC $q$ is cartesian-compatible.\footnote{By \hyperref[lemma:cocartesian_iff_componentwise_cocartesian]{Lemma~\ref*{lemma:cocartesian_iff_componentwise_cocartesian}}, $q$ is automatically cocartesian-compatible, so we needn't add this as a separate condition.}






  \end{enumerate}
\end{definition}

\begin{lemma}
  \label{lemma:cocartesian_iff_componentwise_cocartesian}
  Let $p$ and $q$ be CSMCs as below. Then for any monoidal functor $r$, the CSMC $q$ is cocartesian-compatible.
  \begin{equation*}
    \begin{tikzcd}
      \category{C}_{\otimes}
      \arrow[rr, "r"]
      \arrow[dr, swap, "q"]
      && \category{D}_{\boxtimes}
      \arrow[dl, "p"]
      \\
      & \Finp\op
    \end{tikzcd}
  \end{equation*}
\end{lemma}
\begin{proof}
  We need to show that $f$ in $(\category{C}_{\otimes})_{\langle n \rangle}$ is $r$-cocartesian if and only if it is componentwise cocartesian. Consider the following diagram of pullback squares.
  \begin{equation*}
    \begin{tikzcd}
      (\category{C}_{\otimes})_{\langle n \rangle}
      \arrow[r]
      \arrow[d, swap, "r|\langle n \rangle"]
      & \category{C}_{\otimes}
      \arrow[d, "r"]
      \\
      (\category{D}_{\otimes})_{\langle n \rangle}
      \arrow[r]
      \arrow[d, swap, "p|\langle n \rangle"]
      & \category{D}_{\otimes}
      \arrow[d, "p"]
      \\
      \{\langle n \rangle\}
      \arrow[r]
      & \Finp\op
    \end{tikzcd}
  \end{equation*}
  It follows from the dual to \cite[Cor.~4.3.1.15]{highertopostheory}\footnote{Note an unfortunate notational clash: our maps $p$, $q$, and $r$ do not agree with Lurie's.} that a morphism in $(\category{C}_{\otimes})_{\langle n \rangle}$ is $r|\langle n \rangle$-cocartesian if and only if its image in $\category{C}_{\otimes}$ is $r$-cocartesian. But a morphism is $r|\langle n \rangle$-cocartesian if and only if each component is $r|\langle 1 \rangle$-cocartesian.
\end{proof}

The conditions in \hyperref[def:monoidal_beck_chevalley_fibration]{Definition~\ref*{def:monoidal_beck_chevalley_fibration}} are to do only with the properties of the functor $r|\langle 1 \rangle$, together with properties of the tensor product $\otimes$; the only exception is (M5), which is a compatibility condition with how the monoidal structure is encoded as a cartesian fibration. We can also express these properties in terms directly in terms of the fibrations $p$, $q$, and $r$.
\begin{lemma}
  \label{lemma:equivalent_conditions_beck_chevalley_fibration}
  Consider a functor $r$ of CSMCs, where $\category{T}_{\times}$ is a cartesian CSMC.
  \begin{equation*}
    \begin{tikzcd}
      \category{X}_{\otimes}
      \arrow[rr, "r"]
      \arrow[dr, swap, "q"]
      && \category{T}_{\times}
      \arrow[dl, "p"]
      \\
      & \Finp\op
    \end{tikzcd}
  \end{equation*}
  The following are equivalent.
  \begin{enumerate}
    \item The map $r$ is a monoidal Beck-Chevalley fibration.

    \item The map $r$ has the following properties.
      \begin{enumerate}
        \item The category $\category{T}$ underlying $\category{T}_{\times}$ admits pullbacks.

        \item The map $r$ is a cartesian fibration.

        \item The functor $r|\langle 1 \rangle$ is a cocartesian fibration.

        \item The functor $r$ obeys the following interchange law: for any diagram
          \begin{equation}
            \label{eq:diagram_in_total_space_interchange_law}
            \begin{tikzcd}
              \vec{z}
              \arrow[r, "f'"]
              \arrow[d, swap, "g'"]
              & \vec{y}
              \arrow[d, "g"]
              \\
              \vec{y}'
              \arrow[r, "f"]
              & \vec{x}
            \end{tikzcd}
          \end{equation}
          in $\category{X}_{\otimes}$ whose image in $\category{T}_{\boxtimes}$ is pullback, and which lies over a square
          \begin{equation}
            \label{eq:diagram_in_base_space_interchange_law}
            \begin{tikzcd}
              \langle n \rangle
              & \langle n \rangle
              \arrow[l, "\alpha"{above}, "\simeq"{below}]
              \\
              \langle m \rangle
              \arrow[u, "\phi"]
              & \langle m \rangle
              \arrow[l, "\beta"{above}, "\simeq"{below}]
              \arrow[u, swap, "\psi"]
            \end{tikzcd},
          \end{equation}
          in $\Finp\op$ such that $\alpha$ and $\beta$ are equivalences, if $g$ and $g'$ are $r$-cartesian and $f$ is $r$-cocartesian, then $f'$ is $r$-cocartesian.
      \end{enumerate}
  \end{enumerate}
\end{lemma}
\begin{proof}
  That 1.\ implies a) and c) is clear, and b) follows from \hyperref[lemma:equivalent_conditions_to_be_cartesan]{Lemma~\ref*{lemma:equivalent_conditions_to_be_cartesan}}.

  We now show 1.\ implies d). Consider the diagram
  \begin{equation*}
    \begin{tikzcd}
      \langle n \rangle
      \arrow[rr, equals]
      \arrow[dd, equals]
      \arrow[dr, equals]
      && \langle n \rangle
      \arrow[dd, equals]
      \arrow[dr, leftarrow, "\alpha"]
      \\
      & \langle n \rangle
      \arrow[rr, leftarrow, crossing over, near end, "\alpha"]
      && \langle n \rangle
      \arrow[dd, leftarrow, "\psi"]
      \\
      \langle n \rangle
      \arrow[rr, equals]
      \arrow[dr, leftarrow, swap, "\phi"]
      && \langle n \rangle
      \arrow[dr, leftarrow, "\phi \circ \beta"]
      \\
      & \langle m \rangle
      \arrow[uu, near start, crossing over, swap, "\phi"]
      \arrow[rr, leftarrow, "\beta"]
      && \langle m \rangle
    \end{tikzcd}
    \qquad \text{in }\Finp\op,
  \end{equation*}
  which we should think of as a natural transformation from \hyperref[eq:diagram_in_base_space_interchange_law]{Diagram~\ref*{eq:diagram_in_base_space_interchange_law}} to a constant diagram.

  Beginning with the solid diagram below coming from \hyperref[eq:diagram_in_total_space_interchange_law]{Diagram~\ref*{eq:diagram_in_total_space_interchange_law}}, we can find $q$-cartesian lifts of the diagonal arrows. Filling we find a dashed cube
  \begin{equation*}
    \begin{tikzcd}
      \vec{u}
      \arrow[rr, dashed, "h'"]
      \arrow[dd, dashed]
      \arrow[dr, dashed, "\simeq"]
      && \vec{v}
      \arrow[dd, dashed]
      \arrow[dr, dashed, "\simeq"]
      \\
      & \vec{x}
      \arrow[rr, crossing over, near end, "f'"]
      && \vec{y}
      \arrow[dd]
      \\
      \vec{v}'
      \arrow[rr, dashed, "h", near start]
      \arrow[dr, dashed]
      && \vec{w}
      \arrow[dr, dashed]
      \\
      & \vec{y}'
      \arrow[uu, near start, crossing over, leftarrow]
      \arrow[rr, "f"]
      && \vec{z}
    \end{tikzcd}
    \qquad \text{in }\category{X}_{\otimes}
  \end{equation*}
  lying over the cube in $\Finp\op$, whose diagonal arrows are $q$-cartesian, and with equivalences as marked. Mapping this cube down to $\category{T}_{\times}$, we find a cube
  \begin{equation*}
    \begin{tikzcd}
      r(\vec{u})
      \arrow[rr, dashed]
      \arrow[dd, dashed]
      \arrow[dr, dashed, "\simeq"]
      && r(\vec{v})
      \arrow[dd, dashed]
      \arrow[dr, dashed, "\simeq"]
      \\
      & r(\vec{x})
      \arrow[rr, crossing over, near end]
      && r(\vec{y})
      \arrow[dd]
      \\
      r(\vec{v}')
      \arrow[rr, dashed]
      \arrow[dr, dashed]
      && r(\vec{w})
      \arrow[dr, dashed]
      \\
      & r(\vec{y}')
      \arrow[uu, near start, crossing over, leftarrow]
      \arrow[rr]
      && r(\vec{z})
    \end{tikzcd}
    \qquad \text{in }\category{T}_{\times}
  \end{equation*}
  lying over the cube in $\Finp\op$, whose diagonal morphisms are $p$-cartesian, and whose front face is pullback by assumption. The bottom face is pullback since it lies over a pullback square in $\Finp\op$ and the diagonal morphsims are $p$-cartesian, and the top face is pullback because the diagonal morphisms are equivalences. The cube lemma thus implies that the back square is pullback. Note that the back square is entirely contained in the fiber over $\langle n \rangle$, and thus can be thought of as consisting of $n$ component squares in $\category{T}$; by the equivalence $(\category{T}_{\otimes})_{\langle n \rangle} \simeq (\category{T})^{n}$, these component squares are themselves pullback.

  We now return our attention to the diagram in $\category{X}_{\otimes}$. The morphism $f$ is $r$-cocartesian by assumption. That the tensor product $\otimes$ preserves $r$-cocartesian morphisms implies that $h$ is $r$-cocartesian. Applying the Beck-Chevalley condition componentwise to the back face yields that $h'$ is componentwise cocartesian, hence $r$-cocartesian. Since $f'$ is equivalent to $h'$, $f'$ is also $r$-cocartesian. Thus, d) holds.

  We now show that 2.\ implies 1. Restricting each of the conditions of 2.\ to the fibers over $\langle 1 \rangle$ immediately implies that $r|\langle 1 \rangle$ is a Beck-Chevalley fibration, and b) implies that $r$ is monoidal, and that $q$ is cartesian-compatible. It remains to show that the tensor product $\otimes$ preserves $r$-cartesian and $r$-cocartesian morphisms. To this end, consider a square
  \begin{equation*}
    \sigma =
    \begin{tikzcd}
      \vec{x}
      \arrow[r, "f'"]
      \arrow[d, swap, "g'"]
      & \vec{y}
      \arrow[d, "g"]
      \\
      \vec{y}'
      \arrow[r, "f"]
      & \vec{z}
    \end{tikzcd}
    \qquad \text{in } \category{X}_{\otimes}
  \end{equation*}
  such that $g$ and $g'$ are $q$-cartesian, lying over a square
  \begin{equation}
    \begin{tikzcd}
      \langle n \rangle
      & \langle n \rangle
      \arrow[l, equals]
      \\
      \langle m \rangle
      \arrow[u, "\phi"]
      & \langle m \rangle
      \arrow[l, equals]
      \arrow[u, swap, "\phi"]
    \end{tikzcd}
    \qquad \text{in } \Finp\op.
  \end{equation}
  Note that since $r$ is monoidal, the image of $\sigma$ in $\category{T}_{\times}$ is automatically pullback. We now note that if $f$ is $r$-cartesian, then $f'$ is $r$-cartesian by \cite[Prop.~2.4.1.7]{highertopostheory}, and if $f$ is $r$-cocartesian, then $f'$ is $r$-cocartesian (hence $r|\langle 1 \rangle$-cocartesian) by the interchange law.
\end{proof}

%
%
%

\begin{proposition}
  For any monoidal Beck-Chevalley fibration, there is a diagram
  \begin{equation*}
    \begin{tikzcd}
      \Span'(\category{X})^{\otimes}
      \arrow[rr, "\rho"]
      \arrow[dr, swap, "\varpi"]
      && \Span(\category{T})^{\otimes}
      \arrow[dl, "\pi"]
      \\
      & \Finp
    \end{tikzcd}
  \end{equation*}
  where $\pi$ and $\varpi$ are symmetric monoidal categories and $\rho$ exhibits $\Span'(\category{X})^{\otimes}$ as a $\Span(\category{T})^{\otimes}$-monoidal category. Straightening, one finds a lax monoidal functor
  \begin{equation*}
    \hat{r}\colon (\Span(\category{T}), \widetilde{\times}) \to (\ICat, \times)
  \end{equation*}
  with the following description up to equivalence.
  \begin{itemize}
    \item On objects, the functor $\hat{r}$ sends $t \in \Span(\category{T})$ to the fiber $\category{X}_{t} \in \ICat$

    \item On morphisms, $\hat{r}$ sends a span $t \overset{g}{\leftarrow} s \overset{f}{\to} t'$ to the composition $f_{!} \circ g^{*}\colon \category{X}_{t} \to \category{X}_{t'}$.

    \item The structure morphisms
      \begin{equation*}
        \category{X}_{t} \times \category{X}_{t'} \to \category{X}_{t \times t'}
      \end{equation*}
      of the lax monoidal structure on $\hat{r}$ are given by the restriction of the tensor product $\otimes$ to the fibers over $t$ and $t'$.
  \end{itemize}
\end{proposition}
\begin{proof}
  Recall the triple structure $\triple{T}$ from the proof of \hyperref[prop:cartesian_CSMC_always_exists]{Proposition~\ref*{prop:cartesian_CSMC_always_exists}}. We further define a triple structure $\triple{X}$ as follows.
  \begin{itemize}
    \item $\category{X} = \category{X}_{\otimes}$.

    \item $\category{X}\downdag = (\category{X}_{\otimes})_{\Finp\op}(\Finp\op)^{\simeq}$.

    \item $\category{X}\updag$ consists only of $r$-cartesian morphisms.
  \end{itemize}

  One sees that this is adequate, since pullbacks of the necessary form exist by the usual procedure:
  \begin{itemize}
    \item Map the diagram down to $\Finp\op$, take the pullback there.

    \item Take a relative pullback in $\category{T}_{\times}$.

    \item Take an $r$-cartesian lift to produce an $r$-relative pullback in $\category{X}_{\times}$. This lies over a pullback in $\category{T}_{\times}$, hence is a pullback.
  \end{itemize}

  Thus, we have a map of adequate triples $\triple{X} \to \triple{T}$.
  The fact that this map satisfies the conditions of \hyperref[thm:new_barwick]{Theorem~\ref*{thm:new_barwick}} follows immediately from \hyperref[lemma:equivalent_conditions_beck_chevalley_fibration]{Lemma~\ref*{lemma:equivalent_conditions_beck_chevalley_fibration}}. This gives us a cocartesian fibration
  \begin{equation*}
    \rho\colon \Span\triple{X} \to \Span\triple{T}.
  \end{equation*}
  Combining this with the map $\pi$ of \hyperref[prop:cartesian_CSMC_always_exists]{Proposition~\ref*{prop:cartesian_CSMC_always_exists}} and pulling back along the map $\Finp \to \Span\triple{F}$ gives us the triangle
  \begin{equation*}
    \begin{tikzcd}
      \Span'(\category{X})^{\otimes}
      \arrow[rr, "\rho"]
      \arrow[dr, swap, "\varpi"]
      && \Span(\category{T})^{\otimes}
      \arrow[dl, "\pi"]
      \\
      & \Finp
    \end{tikzcd},
  \end{equation*}
  where $\pi$ is the map shown in \hyperref[prop:cartesian_CSMC_always_exists]{Proposition~\ref*{prop:cartesian_CSMC_always_exists}} to be a cocartesian fibration. It remains only to show that $\varpi = \pi \circ \rho$ is a monoidal category. The same argument as in \hyperref[prop:cartesian_CSMC_always_exists]{Proposition~\ref*{prop:cartesian_CSMC_always_exists}} shows that it suffices to show that square in $(\category{X}_{\otimes})_{\langle n \rangle}$ is ambigressive pullback if and only if each component is ambigressive pullback. The pullback part of the condition holds because of the equivalence $(\category{X}_{\otimes})_{\langle n \rangle} \simeq \category{X}^{n}$, and ambigressivity holds because $\category{X}_{\otimes} \to \category{T}_{\times}$ is by assumption cartesian-compatible.
\end{proof}

\subsection{The monoidal twisted arrow category}
\label{ssc:the_monoidal_twisted_arrow_category}

We have seen that for any $\infty$-bicategory $\CC$ (presented as a fibrant scaled simplicial set), there is an $\infty$-category $\Tw(\CC)$, the \emph{twisted arrow category} of $\CC$. If $\CC$ carries a monoidal structure, this structure is inherited by the twisted arrow $\infty$-category $\Tw(\CC)$ by defining
\begin{equation*}
  \left(
  \begin{tikzcd}
    c_{1}
    \arrow[r]
    \arrow[d, ""{name=L}]
    & c_{2}
    \arrow[d, ""{name=R, swap}]
    \\
    d_{1}
    & d_{2}
    \arrow[l]
    \arrow[from=L, to=R, Rightarrow, shorten=1.5ex, "\eta"{description}]
  \end{tikzcd}
  \right) \otimes \left(
  \begin{tikzcd}
    c_{1}'
    \arrow[r]
    \arrow[d, ""{name=L}]
    & c_{2}'
    \arrow[d, ""{name=R, swap}]
    \\
    d_{1}'
    & d_{2}'
    \arrow[l]
    \arrow[from=L, to=R, Rightarrow, shorten=1.5ex, "\eta'"{description}]
  \end{tikzcd}
  \right)
  =
  \begin{tikzcd}
    c_{1} \otimes c_{1}'
    \arrow[r]
    \arrow[d, ""{name=L}]
    & c_{2} \otimes c_{2}'
    \arrow[d, ""{name=R, swap}]
    \\
    d_{1} \otimes d_{1}'
    & d_{2} \otimes d_{2}'
    \arrow[l]
    \arrow[from=L, to=R, Rightarrow, shorten=1.5ex, "\eta \otimes \eta'"{description}]
  \end{tikzcd}
\end{equation*}
In this section, we construct this monoidal structure explicitly. 

For our construction, will need to upgrade the twisted arrow category construction to a functor 
\begin{equation*}
  \Tw\colon \ITCat \to \ICat,
\end{equation*}
where $\ITCat$ is the $\infty$-category of $\infty$-bicategories. To this end, we note that we can certainly express the twisted arrow category as an ordinary functor
\begin{equation*}
  \Tw'\colon \SSetsc \to \SSetmk.
\end{equation*}
Here we take $\SSetsc$ and $\SSetmk$ to carry their standard model structures.

\begin{lemma}
  The functor $\Tw'$ preserves weak equivalences between fibrant objects.
\end{lemma}
\begin{proof}
  Let $f\colon \CC \to \DD$ be a weak equivalence between fibrant objects in the scaled model structure. Thus, $f$ is a bicategorical equivalence between $\infty$-bicategories. Denote the map on the underlying quasicategories by $\mathring{f}\colon \category{C} \to \category{D}$. We note that $\mathring{f}$ is an equivalence of quasicategories.

  We consider the diagram
  \begin{equation*}
    \begin{tikzcd}[row sep=large, column sep=large]
      \Tw'(\CC)
      \arrow[drr, bend left=20, "\Tw'(f)"]
      \arrow[ddr, bend right=20, swap, "q"]
      \arrow[dr, dashed, "h"]
      \\
      & \mathcal{P}
      \arrow[r, "g"]
      \arrow[d, "r"]
      & \Tw'(\DD)
      \arrow[d, "p"]
      \\
      & \category{C} \times \category{C}\op
      \arrow[r, "\mathring{f} \times \mathring{f}\op"]
      & \category{D} \times \category{D}\op
    \end{tikzcd}
  \end{equation*}
  where the square formed is pullback, and $h$ is the map guaranteed us by the universal property of the pullback. We note that since $p$ is a cartesian fibration, $g$ is a weak equivalence since $\mathring{f} \times \mathring{f}\op$ is. Thus, in order to show that $\Tw'(f)$ is a weak equivalence, it suffices to show that $h$ is.

  Since $h$ is a morphism $q \to r$ of cartesian fibrations, in order to show that it is an equivalence it suffices to check that it is a fiberwise equivalence. Fix some object $C = (c, c') \in \category{C} \times \category{C}\op$. The restriction of $h$ to the fibers of $\Tw'(\CC)$ and $\category{P}$ over $C$ is the map
  \begin{equation*}
    \Map_{\CC}(c, c') \to \Map_{\DD}(f(c), f(c')).
  \end{equation*}
  This is a weak equivalence for all $C$ because $f$ is an equivalence of $\infty$-bicategories.
\end{proof}

By \hyperref[thm:quillen_equiv_ms_and_scaled]{Theorem~\ref*{thm:quillen_equiv_ms_and_scaled}}, there is a right Quillen equivalence $G\colon \SSetms \to \SSetsc$. Composing this with the functor $\Tw'$ above gives us a functor
\begin{equation*}
  \Tw' \circ G\colon \SSetms \to \SSetmk.
\end{equation*}
Since $G$ is a right Quillen functor, it preserves weak equivalences between fibrant objects. Thus, the composite $\Tw' \circ G$ does as well. Restricting to fibrant objects and taking simplicial localizations yields a functor
\begin{equation*}
  \SSetms[W^{-1}] \to \SSetmk[W^{-1}].
\end{equation*}

By \cite[Example~1.3.4.8]{luriehigheralgebra}, we have equivalences of $\infty$-categories $\SSetms[W^{-1}] \simeq \ITCat$ and $\SSetmk[W^{-1}] \simeq \ICat$, giving us our functor
\begin{equation*}
  \Tw\colon \ITCat \to \ICat.
\end{equation*}

We are now ready to construct a model for the monoidal twisted arrow category.

\begin{construction}
  The category $\ICat$ admits a Cartesian monoidal structure, which can be expressed as a commutative monoid $\ICatCocart\colon \Finp \to \ITCat$. The starting point of our construction is the composition
  \begin{equation*}
    \begin{tikzcd}
      \Finp
      \arrow[r, "\ICat^{\times}"]
      & \Cat_{(\infty, 2)}
      \arrow[r, "\Tw"]
      & \ICat
    \end{tikzcd}.
  \end{equation*}
  Note that this composition yields a commutative monoid in $\ICat$ since the functor $\Tw$ preserves products. The relative nerve of this composition is a CSMC $\Tw(\ICat)_{\otimes} \to \Finp\op$ with the following description: an $n$-simplex $\sigma$ corresponding to a diagram
  \begin{equation*}
    \begin{tikzcd}
      \Delta^{n}
      \arrow[rr, "\sigma"]
      \arrow[dr, swap, "\phi"]
      && \Tw(\ICat)_{\otimes}
      \arrow[dl, "p'"]
      \\
      & \Finp\op
    \end{tikzcd}
  \end{equation*}
  corresponds to the data of, for each subset $I \subseteq [n]$ having minimal element $i$, a map
  \begin{equation*}
    \tau(I)\colon \Delta^{I} \to \Tw(\ICat)^{i}
  \end{equation*}
  such that For nonempty subsets $I' \subseteq I \subseteq [n]$, the diagram
  \begin{equation*}
    \begin{tikzcd}
      \Delta^{I'}
      \arrow[r]
      \arrow[d, hook]
      & \Tw(\ICat)^{i'}
      \arrow[d]
      \\
      \Delta^{I}
      \arrow[r]
      & \Tw(\ICat)^{i}
    \end{tikzcd}
  \end{equation*}
  commutes.
\end{construction}

\begin{example}
  An object of $\Tw(\ICat)_{\otimes}$ lying over $\langle n \rangle$ corresponds to a collection of functors $\category{C}_{i} \to \category{D}_{i}$, $i \in \langle n \rangle^{\circ}$.
\end{example}

\begin{example}
  A morphism in $\Tw(\ICat)_{\otimes}$ lying over the active map $\langle 1 \rangle \leftarrow \langle 2 \rangle$ in $\Finp\op$ consists\footnote{Here we mean the morphism in $\Finp\op$ corresponding to the active map $\langle 2 \rangle \to \langle 1 \rangle$ in $\Finp$.} of
  \begin{itemize}
    \item A `source' object $F\colon \category{C} \to \category{C}'$

    \item A pair of `target' objects $G_{i}\colon \category{D}_{i} \to \category{D}'_{i}$, $i = 1$, $2$.

    \item A morphism $F \to G_{1} \times G_{2}$ in $\Tw(\ICat)$ corresponding to a diagram $\Delta^{3}_{\dagger} \to \ICat$ with front and back
      \begin{equation*}
        \begin{tikzcd}
          \category{C}
          \arrow[r, "\alpha"]
          \arrow[d, ""{name=L, right}, "F"{left}]
          \arrow[dr]
          & \category{D}_{1} \times D_{2}
          \arrow[d, ""{name=R, left}, "G_{1} \times G_{2}"{right}]
          \\
          \category{C}'
          & \category{D}'_{1} \times \category{D}'_{2}
          \arrow[l, "\beta"]
        \end{tikzcd}
        \qquad\text{and}\qquad
        \begin{tikzcd}
          \category{C}
          \arrow[r, "\alpha"]
          \arrow[d, ""{name=L, right}, "F"{left}]
          & \category{D}_{1} \times D_{2}
          \arrow[d, ""{name=R, left}, "G_{1} \times G_{2}"{right}]
          \arrow[dl]
          \\
          \category{C}'
          & \category{D}'_{1} \times \category{D}'_{2}
          \arrow[l, "\beta"]
        \end{tikzcd}.
      \end{equation*}
  \end{itemize}

  A morphism of the above form is cartesian if and only if the corresponding simplex $\Delta^{3}_{\dagger} \to \ICat$ is thin.
\end{example}

\begin{lemma}
  \label{lemma:triangle_of_cartesian_fibrations}
  Let $\category{C}$ be a small 1-category (and, by abuse of notation, its nerve), let $F$, $G\colon \category{C} \to \SSet$ be functors, and let $\alpha\colon F \Rightarrow G$. Suppose that $\alpha$ satisfes the following conditions.
  \begin{enumerate}
    \item For each object $c \in \category{C}$, $\alpha_{c}\colon F(c) \to G(c)$ is a cartesian fibration.

    \item For each morphism $f\colon c \to d$ in $\category{C}$, the map $Ff\colon F(c) \to F(d)$ takes $\alpha_{c}$-cartesian morphisms to $\alpha_{d}$-cartesian morphisms.
  \end{enumerate}
  Then taking the relative nerve gives a diagram
  \begin{equation*}
    \begin{tikzcd}
      N_{F}(\category{C})
      \arrow[rr, "\rho"]
      \arrow[dr, swap, "\Phi"]
      &&
      N_{G}(\category{C})
      \arrow[dl, "\Gamma"]
      \\
      & \category{C}\op
    \end{tikzcd}
  \end{equation*}
  with the following properties.
  \begin{enumerate}
    \item The maps $\Phi$, $\Gamma$, and $\rho$ are all cartesian fibrations.

    \item The $\rho$-cartesian morphisms in $N_{F}(\category{C})$ admit the following description: a morphism in $N(F)(\category{C})$ lying over a morphim $f\colon d \leftarrow c$ in $\category{C}\op$ consists of a triple $(x, y, \phi)$, where $x \in F(d)$, $y \in F(c)$, and $\phi\colon x \to Ff(y)$. Such a morphism is $\rho$-cartesian if the morphism $\phi$ is $\alpha_{d}$-cartesian.
  \end{enumerate}
\end{lemma}
\begin{proof}
  We first prove 2. We already know \cite[Lemma~3.2.5.11]{highertopostheory} that $\rho$ is an inner fibration, so it remains only to show that we can solve lifting problems
  \begin{equation*}
    \begin{tikzcd}
      \Lambda^{n}_{n}
      \arrow[r]
      \arrow[d]
      & N_{F}(\category{C})
      \arrow[d]
      \\
      \Delta^{n}
      \arrow[r]
      \arrow[ur, dashed]
      \arrow[dr, swap, "\gamma"]
      & N_{G}(\category{C})
      \arrow[d]
      \\
      & \category{C}\op
    \end{tikzcd},
  \end{equation*}
  where $\Delta^{\{n-1, n\}} \subset \Lambda^{n}_{n}$ is mapped to a cartesian morphism as described above, i.e.\ a triple $(x, y, \phi)$, where $x \in F(\gamma(n-1))$, $y \in F(\gamma(n))$, and $\phi\colon x \to F(\gamma_{n})(y)$ is $\alpha_{\gamma(n-1)}$-cartesian. This is equivalent to solving the lifting problem
  \begin{equation*}
    \begin{tikzcd}
      \Lambda^{n}_{n}
      \arrow[r]
      \arrow[d]
      & F(\gamma(0))
      \arrow[d, "\alpha_{\gamma(0)}"]
      \\
      \Delta^{n}
      \arrow[r]
      \arrow[ur, dashed]
      & G(\gamma(0))
    \end{tikzcd},
  \end{equation*}
  where $\Lambda^{n}_{n}$ is $F(\gamma_{n-1, 0})(\phi)$. But by assumption $F(\gamma_{n-1, 0})(\phi)$ is $\alpha_{\gamma(0)}$-cartesian, so this lifting problem has a solution.

  Now we show 1. The assumption that each $\alpha_{c}$ is a cartesian fibration guarantees that we have enough cartesian lifts, so $\rho$ is indeed a cartesian fibration. The maps $\Phi$ and $\Gamma$ are cartesian fibrations by definition.
\end{proof}

Since the category $\ICat$ admits products, it admits a cartesian monoidal structure, which we can write as a commutative monoid $G_{1}\colon \Finp \to \ICat$. We can equally view the cartesian monoidal structure as a cocartesian monoidal structure on $\ICat\op$, giving us a commutative monoid $G_{0}\colon \Finp \to \ICat$. Taking these together gives a commutative monoid
\begin{equation*}
  G = G_{0} \times G_{1}\colon \Finp \to \ICat;\qquad \langle n \rangle \mapsto ( \ICat\op )^{n} \times \ICat^{n}.
\end{equation*}
For each $\langle n \rangle \in \Finp\op$ there is a cartesian fibration $\alpha_{n}\colon \Tw(\ICCat)^{n} \to (\ICat^{n})\op \times \ICat^{n}$, which form the components of a natural transformation $\alpha$ from the functor
\begin{equation*}
  F\colon \Finp \to \ICat;\qquad \langle n \rangle \mapsto \Tw(\ICCat)^{n}
\end{equation*}
to the functor $G$.

We now apply the relative nerve to the data $\alpha\colon F \Rightarrow G$. Because the pointwise product of any number of thin 1-simplices in $\Tw(\ICCat)$ is again a thin 1-simplex in $\Tw(\ICCat)$, the conditions of \hyperref[lemma:triangle_of_cartesian_fibrations]{Lemma~\ref*{lemma:triangle_of_cartesian_fibrations}} are satisfied. Unrolling, we find a commuting triangle of cartesian fibrations
\begin{equation}
  \label{eq:base_triangle}
  \begin{tikzcd}
    \Tw(\ICat)_{\otimes}
    \arrow[rr, "r'"]
    \arrow[dr, swap, "q'"]
    && \TICatCart \times (\TICatCocart)\op
    \arrow[dl, "p'"]
    \\
    & \Finp\op
  \end{tikzcd},
\end{equation}
where
\begin{itemize}
  \item $\TICatCart \to \Finp\op$ is the relative nerve (as a cartesian fibration) of the functor $F_{0}$

  \item $(\TICatCocart) \to \Finp$ is the relative nerve (as a \emph{cocartesian} fibration) of the same.
\end{itemize}

Because both $\TICatCocart \to \Finp$ and $\ICatCocart \to \Finp$ classify the same functor $\Finp \to \ICat$, they are related by a fiberwise equivalence. Composing a commutative monoid $\Finp \to \ICatCocart$ with this equivalence allows us to express symmetric monoidal $\infty$-categories as commutative monoids in $\TICatCocart$.

\subsection{The monoidal category of local systems}
\label{ssc:the_monoidal_category_of_local_systems}

In this subsection, we show that the monoidal structure on the twisted arrow category, defined in \hyperref[ssc:the_monoidal_twisted_arrow_category]{Subsection~\ref*{ssc:the_monoidal_twisted_arrow_category}}, can be used to construct a monoidal structure on the category of local systems.

\begin{construction}
  We will denote the full subcategory of $\TICatCart$ on those objects $[\category{D}_{1}, \ldots, \category{D}_{n}]$ such that $\category{D}_{i}$ is an $\infty$-groupoid for all $1 \leq i \leq n$ by $\TSCart$. Note that $\TSCart \to \Finp$ can be understood as the relative nerve (as a cartesian fibration) of a commutative monoid $\Finp \to \ICat$ giving the cartesian monoidal structure on $\S$.

  Fix some monoidal $\infty$-category $\category{C}$ which admits colimits, and such that the tensor product $\otimes\colon \category{C} \to \category{C}$ preserves colimits in each slot. We express this monoidal $\infty$-category as a commutative monoid $\category{C}^{\otimes}\colon \Finp \to \TICatCocart$. Using this, we define a functor
  \begin{equation*}
    \TSCart \to \TICatCart \times (\TICatCocart)\op
  \end{equation*}
  which is the inclusion $\TSCart \hookrightarrow \TICatCart$ on the first component of the product, and given by the composition
  \begin{equation*}
    \TSCart \to \Finp\op \overset{(\category{C}^{\otimes})\op}{\to} (\TICatCocart)\op
  \end{equation*}
  on the second. Forming the pullback square
  \begin{equation*}
    \begin{tikzcd}
      \LS(\category{C})_{\otimes}
      \arrow[r]
      \arrow[d, swap, "r"]
      & \Tw(\ICat)_{\otimes}
      \arrow[d]
      \\
      \TSCart
      \arrow[r]
      & \TICatCart \times (\TICatCocart)\op
    \end{tikzcd}
  \end{equation*}
  gives us a commutative triangle
  \begin{equation*}
    \begin{tikzcd}
      \LS(\category{C})_{\otimes}
      \arrow[rr, "r"]
      \arrow[dr, swap, "q"]
      && \TSCart
      \arrow[dl, "p"]
      \\
      & \Finp\op
    \end{tikzcd}.
  \end{equation*}

  We claim that $\LS(\category{C})_{\otimes}$ is a CSMC, $\TSCart$ is a cartesian CSMC. Since the (cartesian) relative nerve construction produces a cartesian fibration by definition, the map $p$ is a cartesian fibration; by its definition, it is even a CSMC. Furthermore, because the map $r$ is a pullback of the horizontal map in \hyperref[eq:base_triangle]{Diagram~\ref*{eq:base_triangle}}, it is a cartesian fibration. Hence $q$ is also a cartesian fibration. It remains only to show that $q$ is a CSMC. It suffices to show that $r$ is a cartesian fibration of $\infty$-operads. Note that because both $p'$ and $q'$ are CSMCs, $r'$ is a cartesian fibration of $\infty$-operads. The claim follows because $r$ is a pullback of $r'$.
\end{construction}

\begin{example}
  An object of $\LS(\category{C})_{\otimes}$ lying over $\langle n \rangle$ corresponds to a collection of functors $X \to \category{C}$, $i \in \langle n \rangle^{\circ}$.
\end{example}

\begin{example}
  A morphism in $\Tw(\ICat)_{\otimes}$ lying over the active map $\langle 1 \rangle \leftarrow \langle 2 \rangle$ in $\Finp\op$ consists of
  \begin{itemize}
    \item A `source' object $F\colon X \to \category{C}$

    \item A pair of `target' objects $G_{i}\colon Y_{i} \to \category{C}$, $i = 1$, $2$.

    \item A morphism $\Tw(\ICat)$ corresponding to a diagram $\Delta^{3}_{\dagger} \to \ICat$ with front and back
      \begin{equation*}
        \begin{tikzcd}
          X
          \arrow[r, "\alpha"]
          \arrow[d, ""{name=L, right}, "F"{left}]
          \arrow[dr]
          & Y_{1} \times Y_{2}
          \arrow[d, ""{name=R, left}, "G_{1} \times G_{2}"{right}]
          \\
          \category{C}
          & \category{C} \times \category{C}
          \arrow[l, "\otimes"]
        \end{tikzcd}
        \qquad\text{and}\qquad
        \begin{tikzcd}
          X
          \arrow[r, "\alpha"]
          \arrow[d, ""{name=L, right}, "F"{left}]
          & Y_{1} \times Y_{2}
          \arrow[d, ""{name=R, left}, "G_{1} \times G_{2}"{right}]
          \arrow[dl]
          \\
          \category{C}
          & \category{C} \times \category{C}
          \arrow[l, "\otimes"]
        \end{tikzcd}.
      \end{equation*}
  \end{itemize}

  A morphism of the above form is $r$-cartesian if and only if the corresponding simplex $\Delta^{3}_{\dagger} \to \ICat$ is thin, and $q$-cartesian if and only if the corresponding simplex is thin, and $\alpha$ is an equivalence.
\end{example}

\begin{example}
  Examining the $q$-cartesian case more closely, one sees that the tensor product of two local systems $F\colon X \to \category{C}$ and $G\colon Y \to \category{C}$ is the local system given the composition
  \begin{equation*}
    F \otimes G\colon X \times Y \to \category{C} \times \category{C} \overset{\otimes}{\to} \category{C}
  \end{equation*}
\end{example}

\begin{lemma}
  \label{lemma:tensor_product_of_local_systems_preserves_cocart_edges}
  The tensor product of local systems preserves cocartesian edges (in the sense of \hyperref[def:terminology_about_csmcs]{Definition~\ref*{def:terminology_about_csmcs}}).
\end{lemma}
\begin{proof}
  In order to show that the tensor product of local systems preserves cocartesian edges, it suffices by \hyperref[lemma:co_cartesian_preservation_determined_by_bifunctor]{Lemma~\ref*{lemma:co_cartesian_preservation_determined_by_bifunctor}} to check that the tensor product of two $r|\langle 1 \rangle$-cocartesian edges is again $r|\langle 1 \rangle$-cocartesian. Let $e\colon F \to G$ and $e'\colon F' \to G'$ be $r|\langle 1 \rangle$-cocartesian edges of $\LS(\category{C})_{\otimes}$.
  \begin{equation*}
    e =
    \begin{tikzcd}
      X
      \arrow[r]
      \arrow[d, swap, "F"]
      & Y
      \arrow[d, "G"]
      \\
      \category{C}
      & \category{C}
      \arrow[l, "\id"]
    \end{tikzcd},\qquad e' =
    \begin{tikzcd}
      X'
      \arrow[r]
      \arrow[d, swap, "F'"]
      & Y'
      \arrow[d, "G'"]
      \\
      \category{C}
      & \category{C}
      \arrow[l, "\id"]
    \end{tikzcd}.
  \end{equation*}
  We wish to show that $e \otimes e'$ is $r|\langle 1 \rangle$-cocartesian.

  The tensor product $e \otimes e'$ is given, up to homotopy, by the pasting diagram
  \begin{equation*}
    \begin{tikzcd}
      X \times X'
      \arrow[r]
      \arrow[d]
      & Y \times Y'
      \arrow[d]
      \\
      \category{C} \times \category{C}
      \arrow[d, swap, "\otimes"]
      & \category{C} \times \category{C}
      \arrow[d, "\otimes"]
      \arrow[l]
      \\
      \category{C}
      & \category{C}
      \arrow[l]
    \end{tikzcd}
  \end{equation*}
  
  In order to show that this map is $r$-cocartesian, we need to show that the outer triangle
  \begin{equation*}
    \begin{tikzcd}[row sep=large]
      X \times X'
      \arrow[rr, "F \times F'", ""{below, name=M}]
      \arrow[dr, swap, "f \times f'"]
      && \category{C} \times \category{C}
      \arrow[r, "\otimes"]
      & \category{C}
      \\
      & Y \times Y'
      \arrow[urr, bend right, swap, "\otimes \circ (G \times G')"]
      \arrow[ur, swap, "G \times G'"]
      \arrow[from=M, Rightarrow, shorten=1ex]
    \end{tikzcd}
  \end{equation*}
  exhibits $\otimes \circ (G \times G')$ as the left Kan extension of $\otimes \circ (F \times F')$ along $f \times f'$. But this follows from the fact that $\otimes$ preserves colimits in both slots.
\end{proof}

\begin{proposition}
  The map $r$ above is a monoidal Beck-Chevalley fibration.
\end{proposition}
\begin{proof}
  We need to check the conditions (M1)--(M5) of \hyperref[def:monoidal_beck_chevalley_fibration]{Definition~\ref*{def:monoidal_beck_chevalley_fibration}}. The condition (M1) is the content of \hyperref[prop:local_systems_are_beck_chevalley]{Proposition~\ref*{prop:local_systems_are_beck_chevalley}}. Conditions (M2), (M3), and (M5) follow immediately from \hyperref[lemma:equivalent_conditions_to_be_cartesan]{Lemma~\ref*{lemma:equivalent_conditions_to_be_cartesan}}, using the fact that $r$ is a cartesian fibration. That (M4) holds is the content of \hyperref[lemma:tensor_product_of_local_systems_preserves_cocart_edges]{Lemma~\ref*{lemma:tensor_product_of_local_systems_preserves_cocart_edges}}.
\end{proof}

\begin{corollary}
  There is a lax monoidal functor
  \begin{equation*}
    \hat{r}\colon (\Span(\S), \widetilde{\times}) \to (\ICat, \times),
  \end{equation*}
  with the following description up to equivalence.
  \begin{itemize}
    \item On objects, the functor $\hat{r}$ sends a space $X$ to the $\infty$-category $\LS(\category{C})_{X}$ of $\category{C}$-local systems on $X$.

    \item On morphisms, the functor $\hat{r}$ sends a span of spaces $X \overset{g}{\leftarrow} Y \overset{f}{\to} X'$ to the pull-push
      \begin{equation*}
        f_{!} \circ g^{*}\colon \LS(\category{C})_{X} \to \LS(\category{C})_{X'}.
      \end{equation*}

    \item The structure morphisms of the lax monoidal structure are the maps
      \begin{equation*}
        \LS(\category{C})_{X} \times \LS(\category{C})_{X'} \to \LS(\category{C})_{X \times X'}
      \end{equation*}
      given by the composition
      \begin{equation*}
        \Fun(X, \category{C}) \times \Fun(Y, \category{C}) \overset{\times}{\to} \Fun(X \times Y, \category{C} \times \category{C}) \overset{\otimes}{\to} \Fun(X \times Y, \category{C}),
      \end{equation*}
      under the identification $\LS(\category{C})_{X} \cong \Fun(X, \category{C})$.
  \end{itemize}
\end{corollary}

\printbibliography

\end{document}